\documentclass[preprint,12pt]{article}

\usepackage[english]{babel}	
\usepackage[utf8x]{inputenc}		
\usepackage{enumerate}
\usepackage{textcomp}                
\usepackage{typearea}
\usepackage{pdfpages}
\usepackage[toc]{appendix}
\usepackage[]{todonotes}
\usepackage{subcaption}

\usepackage{amsfonts, amsmath, amssymb, amsthm, dsfont}
\usepackage{amsthm}			
\usepackage{thumbpdf}	
\usepackage{natbib}
\usepackage{booktabs}
\usepackage[noblocks]{authblk}

\usepackage{pgf, tikz}
\usepackage{tikz-cd}		
\usepackage{bbm}

\theoremstyle{definition}
\newtheorem{definition}{Definition}[section]

\newtheorem*{assumption*}{Assumption}

\newtheorem*{condition*}{Condition}

\theoremstyle{plain}
\newtheorem{theorem}[definition]{Theorem}
\newtheorem{proposition}[definition]{Proposition}
\newtheorem{lemma}[definition]{Lemma}
\newtheorem{cor}[definition]{Corollary}
\theoremstyle{remark}
\newtheorem{remark}{Remark}

\usepackage{graphicx}
\usepackage{hyperref}

\newcommand{\N}{\mathbb{N}}
\newcommand{\Z}{\mathbb{Z}}
\newcommand{\Q}{\mathbb{Q}}
\newcommand{\R}{\mathbb{R}}

\newcommand{\E}{\mathbb{E}}

\newcommand{\pconv}{\xrightarrow{\mathbb{P}}}

\newcommand{\Var}{\operatorname{Var}}
\newcommand{\Cov}{\operatorname{Cov}}

\newcommand{\deq}{\overset{d}{=}}

\newcommand{\wconv}{\Rightarrow}

\def\doubleunderline#1{\underline{\underline{#1}}}

\newcommand{\diag}{\mathrm{diag}}

\newif\ifhideproofs
\ifhideproofs
  \usepackage{environ}
  \NewEnviron{hide}{}

\fi

\title{Estimation of mixed fractional stable processes using high-frequency data}
\author{Fabian Mies}
\affil{RWTH Aachen University, Germany}
\author{Mark Podolskij}
\affil{Universit\'e du Luxembourg}

\begin{document}
	
\maketitle

\begin{abstract}
    \noindent
    The linear fractional stable motion generalizes two prominent classes of stochastic processes, namely stable Lévy processes, and fractional Brownian motion.
	For this reason it may be regarded as a basic building block for continuous time models.	
	We study a stylized model consisting of a superposition of independent linear fractional stable motions and our focus is on parameter estimation of the model.
	Applying an estimating equations approach, we construct estimators for the whole set of parameters and derive their asymptotic normality in a high-frequency regime. 
	The conditions for consistency turn out to be sharp for two prominent special cases: (i) for Lévy processes, i.e.\ for the estimation of the successive Blumenthal-Getoor indices, and (ii) for the mixed fractional Brownian motion introduced by Cheridito. 
	In the remaining cases, our results reveal a delicate interplay between the Hurst parameters and the indices of stability. 
	Our asymptotic theory is based on new limit theorems for multiscale moving average processes.\\ 
	
	
	
\end{abstract}

\section{Introduction}

\subsection{Overview}

The linear fractional stable motion (lfsm) is a self-similar stochastic process defined as
\begin{align*}
	Y_{t}^{H,\beta} = \int_{-\infty}^t (t-s)_+^{H-\frac{1}{\beta}} - (-s)_+^{H-\frac{1}{\beta}}\, dZ_s^\beta,\quad t\in\R,
\end{align*}
for a Hurst parameter $H\in(0,1)$ and a standard symmetric $\beta$-stable Lévy motion $Z^\beta$ with $\beta\in(0,2]$, i.e. 
$$\E \exp(i\lambda Z_t^\beta) = \exp(-t|\lambda|^\beta).$$ 
Here, we use the notation $(x)_+ = x \mathds{1}_{x>0}$.
Notable special cases of the lfsm are fractional Brownian motion ($\beta=2$), and the $\beta$-stable Lévy motion itself ($H=1/\beta$).
For general parameters $H$ and $\beta$, the increments of the lfsm exhibit long memory and heavy tails.
Another prominent feature of lfsm is its self-similarity, namely $(Y_{\gamma t}^{H,\beta})_{t\geq 0} \deq (\gamma^H Y_{t}^{H,\beta})_{t\geq 0}$ for any $\gamma>0$.
The lfsm is in some sense prototypical, since it arises as a scaling limit of various moving average processes in discrete time, see e.g. \cite{astrauskas1984}. 
Hence, parameter estimation for the lfsm may be regarded as an idealized testbed for more general, non-semimartingale processes.

In this paper, we study the mixed fractional stable process of the form
\begin{align}
	X_t = \sum_{j=1}^q b_j Y_t^{H_j,\beta_j}, \label{eqn:def-mixed-lfsm}
\end{align}
where $Y^{H_j, \beta_j}$, $j=1,\ldots, q$, are independent fractional stable motions and $b_j>0$.
Unless all $H_j$ are identical, the mixed process is no longer self-similar. Throughout this work we assume that $H_1 < \ldots < H_q$.
In this case we have  $\gamma^{-H_1} X_{\gamma t}\wconv b_1Y_t^{H_1, \beta_1}$ as $\gamma\downarrow 0$, and $\gamma^{-H_q} X_{\gamma t}\wconv b_qY_t^{H_q, \beta_q}$ as $\gamma\uparrow \infty$.
Loosely speaking, the process looks different when zooming in ($\gamma\downarrow 0$) or zooming out ($\gamma\uparrow \infty$).
In the sequel, we are mostly interested in the scaling as $\gamma\downarrow 0$.
In particular, we want to estimate the parameters $(b_j, H_j, \beta_j)$ based on high-frequency, discrete observations of $X$ in the interval $[0,1]$.
Based on the scaling limit, it is not surprising that $(b_1, H_1, \beta_1)$ may be estimated consistently as $n\to \infty$.
Can we also estimate the remaining parameters, and at which rate?

Estimation of the lfsm has been recently studied by \cite{Mazur2018} for high-frequency observations,  and by \cite{Ljungdahl2020} for low-frequency observations.
To see why estimation of the mixed lfsm is more complicated, we briefly review the methodology of \cite{Mazur2018}.
Crucially, they exploit the self-similarity of the lfsm, $h^{-H}(Y_{t+h}^{H,\beta}-Y_{t}^{H,\beta}) \deq (Y_{t+1}^{H,\beta}-Y_t^{H,\beta})$ to transfer the high-frequency setting to the low-frequency setting.
In particular, they suggest to first estimate $H$ by a log-ratio statistic, and then estimate $(b,\beta)$ based on the empirical characteristic function of the rescaled increments. 
For the mixed lfsm, on the other hand, the log-ratio estimator will only estimate the dominant, i.e.\ the smallest Hurst exponent $H_1$. 
It is thus not clear how to extend their procedure to the mixed case.
Notably, the mixed lfsm is  no longer self-similar, and we may not switch between the low-frequency and the high-frequency regime.

In this paper, we propose the first estimators for the mixed lfsm, based on $n$ discrete observations $X_{\Delta_n},\ldots, X_{n\Delta_n}$ at frequency $\Delta_n$. 
Of particular interest is the high-frequency case $\Delta_n=1/n$, but we also derive asymptotic results for the regime $\Delta_n\to 0$, $n \Delta_n \to \infty$. 
Our estimation method are based upon a system of nonlinear equation of the form 
$\E f_n(X_{i\Delta_n},\ldots, X_{(i+k)\Delta_n})$ for a suitable choice of functions $f_n$. 
These moments may be estimated by their sample means, and a parameter estimator may be obtained via the generalised method of moments. 
In contrast to the classical literature, the function $f_n$ as well as the process $X_{i \Delta_n}$ now depend explicitly on $n$.
To be precise, we will choose $f_n(x_0,\ldots, x_k) = f(u_n \sum_{r=1}^k \binom{k}{r} (-1)^r x_r)$, i.e.\ we take the $k$-th order increments of the process, scaled by a factor $u_n\to \infty$, and the function $f$ will be bounded. In the first step we will derive the asymptotic theory for such class of functionals, which will provide the basis for statistical inference. This limit theory is new and has interest in its own right. In the second step we will define an estimator for the unknown parameters of the model via a system of estimating equations. We will study two related approaches, where the second one (smooth threshold) explicitly accounts for the presence of a Gaussian component ($\beta=2$). We prove the asymptotic normality of our estimator and discuss the identifiability issues.

\subsection{Related work}

While the estimation problem under consideration is new, there exists a large body of work in more restrictive settings. This builds a natural comparison basis for our new methodology. \\

\noindent
(i) \textbf{Linear fractional stable motion.} Estimation of 
a single lfsm, i.e.\ $q=1$, has been studied in several papers; 
early references include a non-rigorous treatment by \cite{abry1999estimation}, and a study by \cite{abry1999} who obtain a suboptimal rate of convergence for the parameter $H$. 
An asymptotically normal estimator for $H$ is proposed by \cite{stoev2002, stoev2005a}, using a central limit theorem published in \cite{pipiras2007}.
A consistent estimator for the stability index $\beta$, provided that $H$ is known, is given by \cite{Ayache2012}, and consistent estimators for $(H,\beta)$ are suggested by \cite{grahovac2015,Basse-OConnor2017,dang2017}.
Estimation of the full parameter $(b,H,\beta)$ has been investigated by \cite{Mazur2018}.
Their approach is based on power variations with negative exponents, and they derive the asymptotic normality of their estimator in the low- and high-frequency regime. 
For the low-frequency regime, these results have been further refined in \cite{Ljungdahl2019} an \cite{Ljungdahl2020}. We remark that there is currently no statistical lower bound for estimation of the general lfsm.\\

\noindent
(ii) \textbf{Mixed fractional Brownian motion.}
The mixed fractional Brownian motion corresponds to the model \eqref{eqn:def-mixed-lfsm} with $\beta_i=2$ for all $i$. It has been originally studied in \cite{Cheridito2001}.
It has been shown in \cite{VanZanten2007} that the measures induced by $C_t^1 = b_1Y_t^{H_1,2}$ and $C_t^2 = b_1Y_t^{H_1,2} +  b_2 Y_t^{H_2,2}$, $t\in[0,1]$, are equivalent if $H_2 > H_1 + \frac{1}{4}$, and singular otherwise.
Hence, in the Gaussian case, the parameter $(b_j, H_j)$ is identifiable if and only if $H_i < H_1 + \frac{1}{4}$.
Estimation of the parameters of a mixed fractional Brownian motion has been studied by \cite{Xiao2011}, although in a very restricted setting where the $H_j$ are assumed to be  known.
An estimator for a much more general nonstationary model with $d=2$ and $H_2 < H_1=\frac{1}{2}$ has been recently suggested by \cite{chong2021}.
To the best of our knowledge, the latter paper is the first investigation of inference for the general mixed fractional Brownian motion. 
An extension has been studied in \cite{chong2022}, where the two processes are allowed to be correlated, and optimal rates are derived for the special case of non-zero correlation.\\

\noindent
(iii) \textbf{Mixed L\'evy processes.}
Some statistical results have been obtained in the setting of 
mixed L\'evy processes, which corresponds to model \eqref{eqn:def-mixed-lfsm} with 
$H_j=1/\beta_j$ (in other words, $Y^{H,\beta}=Z^{\beta}$).
If $\beta_1<2$ denotes the largest stability index, \cite{ait2012identifying} showed
that the measures induced by $D_t^1 = b_1Z_t^{\beta_1}$ and $D_t^2 = b_1Z_t^{\beta_1} + b_2 Z_t^{\beta_2}$ are equivalent if $\beta_2<\beta_1/2$, and singular otherwise. 
Thus, the parameter $(b_j,\beta_j$) is identifiable if and only if $\beta_i > \beta_1/2$.
See also \cite{ait2008fisher} for an earlier treatment of the special case $q=2$.
Another known case is given by mixed L\'evy processes with the additional restriction that $\beta_1=2$, and $\beta_2<2$ being the largest of the remaining indices. In this setting,
the identifiability condition becomes $\beta_j>\beta_2/2$. 
In other words, the identifiability of the stable components is unaffected by the presence of a Gaussian component. This can be explained by the observation that
given a full trajectory of the process $X$, the continuous and the discontinuous components could be perfectly separated.
In discrete samples, however, the presence of the Gaussian component has an adverse effect on the rate of convergence of any estimator of $\beta_j, j\geq 2$; see \cite{ait2012identifying}.
Note that this case is of particular interest for financial econometrics, where many models for high-frequency asset prices contain both, continuous and discontinuous components. 
Here, the sup of a Brownian motion and a stable Lévy motion may be regarded as the prototype of a semimartingale with infinite jump activity. 
The corresponding statistical problem has been first studied in the econometric literature in \cite{ait2009estimating}, and more recently by \cite{bull2016near} and   \citep{Mies2019}. \\

\noindent
Besides these three special cases, the mixed lfsm has not been studied in the literature. 
The currently unexplored regimes include, for example: the sum of fBm and a stable Lévy process; the sum of Brownian motion and lfsm; the sum of non-Gaussian lfsms; the sum of fBm and non-Gaussian lfsm. 
Our unified theory covers all these regimes, while matching the results in the cases which have already been investigated.

\subsection{Outline of the paper}
The paper is structured as follows. In Section 2, we present a new limit theory for functionals of a mixed fractional stable motion. These results not only provide a necessary basis for statistical estimation, but also have an interest in their own right. Section 3 is devoted to statistical inference for mixed fractional stable motions. Here we employ the idea of estimating equations that may specifically account for the presence of a Gaussian component (Section 3.2). 
In Section 3.3, we discuss the rates of convergence of our estimators, and we present a result on the singularity of measures induced by the mixed lfsm.
In the appendix, we derive a general result on consistency and asymptotic normality of solutions of estimating equations, which generalizes existing results in the literature (Appendix A).
The proofs of the main results are deferred to Appendix B.

\subsection{Notation}

Throughout the paper we use the following notations. We write 
$C^p(\R^{d_1};\R^{d_2})$ to denote the space of $p$ times continuously differentiable functions $f:\R^{d_1} \to \R^{d_2}$, and denote the first derivative matrix as $D f(x)_{i,j} = \frac{d}{dx_j} f_i(x)$.
For a function $f:\R\to\R$, we denote its $k$-th derivative as $f^{(k)}$. The space $L_p(\R^d)$ is the collection of all functions $f$ satisfying $\int_{\R^d} \|f(x)\|^p dx<\infty$. We write $a_n \sim b_n$ when there exist constants $c_1,c_2>0$ such that $c_1a_n\leq b_n \leq c_2 a_n$ for all $n\geq 1$. The notation
$a_n \gg b_n$ means that $b_n/a_n \to 0$ as $n\to \infty$. Throughout this paper all positive constants are denoted by $C$, or by $C_p$ if we want to stress their dependence on some external parameter $p$, although they may change from line to line.

\section{Limit theorems for multiscale moving average processes}\label{sec:MA}

In this section, we present a novel limit theorem for functionals of multiscale moving average processes, which will provide the theoretical basis for the statistical procedures investigated in the next section. Our main motivation comes 
from statistics of higher-order increments of $X$, which are defined as follows: For a frequency $\Delta_n\to 0$, 
$\gamma \in \N$ and $k\in \N$, the $k$-th order increments of $X$ at the frequency $\gamma \Delta_n$ are given as 
\begin{align} \label{korder}
	X_{l,n,\gamma} &= \sum_{v=0}^{k} (-1)^v \binom{k}{v} X_{(l-v\gamma)\Delta_n},\quad l=1,\ldots, n. 
\end{align}
The factor $\gamma$ allows us to identify parameters of the lfsm's via their different temporal scaling. 
In particular, the self-similarity of the lfsm yields
\begin{align}
	X_{l,n,\gamma} &\deq \sum_{j=1}^{q} b_j\gamma^{H_j}\Delta_n^{H_j} \int_{-\infty}^l g_{j}(s-l)\, dZ_{s}^{\beta_j},\quad \text{where} \label{eqn:X-multiscale} \\
	g_{j}(s) &= \sum_{v=0}^k \binom{k}{v} (-1)^v (v-s)_+^{H_j - \frac{1}{\beta_j}}. \nonumber
\end{align}
This higher-order differencing is common when working with fractional processes, as higher orders $k$ improve the decay of the autocovariances of $f(X_{l,n,\gamma})$ for suitably bounded functions $f$. 
In particular, the integral kernels decay as $g_j(s) \sim s^{H_j-k - 1/\beta_j}$, making this effect evident.

Motivated by this example, we consider a more general class of discrete models, namely a multiscale array of moving averages $(X_{t,n})_{t\in \N}$ defined as
\begin{align}
	X_{t,n} 
	&= \sum_{i=1}^{q_1} a_{n,i} \int_{-\infty}^t h_i(s-t)\,dB^i_s + \sum_{j=1}^{q_2} b_{n,j} \int_{-\infty}^t g_j(s-t)\, dZ^j_s,  \label{eqn:def-multiscale}
\end{align}
for non-negative sequences $a_{n,i}$ and $b_{n,j}$.
Here, for $i=1,\ldots, q_1$, the $B^i$ are independent standard Brownian motions, and for $j=1,\ldots, q_2$, the $Z^j$ are independent symmetric $\beta_j$-stable motions, standardized such that $\E \exp(i\lambda Z_1^j)=\exp\left(-|\lambda|^{\beta_j}\right)$, $\beta_j\in(0,2)$. We allow the kernels $g,h$ to be $d$-dimensional, i.e. $g_i,h_i:\R \to \R^d$. This allows us to account for different $\gamma$'s in \eqref{korder} by deriving a multivariate limit theory. 
Throughout this section
the kernels are assumed to satisfy the following conditions:
\begin{itemize}
	\item[(i)] $h_i(x)=g_j(x)=0$ for $x\geq 0$, (causality)
	\item[(ii)] $h_i\in L_2(\R)$ with $\|h_i(x)\| \leq C |x|^{\delta_0}$, and $\|h_i^l\|_{L_2} \leq \sigma<\infty$,
	\item[(iii)] $g_j\in L_{\beta_j}(\R)$ and $\|g_j(x)\| \leq C  |x|^{\delta_j}$ for some $\delta_j<-1/\beta_j$.
\end{itemize}
\noindent
Definition \eqref{eqn:def-multiscale} explicitly distinguishes between the Gaussian and the non-Gaussian components. 
As demonstrated in the sequel, this distinction is motivated by a qualitatively different effect of the Gaussian component on the limiting behavior.

Our focus is on statistics of the form
\begin{align}
S_n(f) = \frac{1}{n}\sum_{t=1}^n \left[f(X_{t,n})-\E f(X_{t,n})\right]
\end{align}
where $f:\R^d \to \R$ is a nonlinear function. We will assume 
that it belongs to the following class of functions.
\begin{definition}
	For $\eta\geq 0$, define the class $\mathfrak{F}_\eta$ of all functions $f:\R^d\to\R$ such that
	\begin{itemize}
		\item[(F1)] $f\in C^5(\R^d;\R)$ and $\|D^jf\|_\infty \leq 1$ for $j=0,\ldots, 5$,
		\item[(F2)] $f$ is even, and $f(0)=0$,
		\item[(F3)] $D^2f(x)=D^2f(0)$ for $\|x\| < \eta$.
	\end{itemize}
	Moreover, define the class $\mathfrak{F}_\eta^0\subset \mathfrak{F}_\eta$, which additionally satisfies
	\begin{itemize}
		\item[(F4)] $f(x)=0$ for $\|x\|<\eta$.
	\end{itemize}
\end{definition}
\noindent
The class $\mathfrak{F}_\eta$ contains functions which are quadratic on the interval $(-\eta,\eta)$, whereas functions in $\mathfrak{F}_\eta^0$ are smooth thresholds.
Note that the class $\mathfrak{F}_0 = \mathfrak{F}_0^0$ imposes the least regularity.

The array $X_{t,n}$ is a superposition of multiple processes with long memory and heavy tails, of different severity.
Naively, one would expect that the limiting behavior of $S_n(f)$ is determined by the component with the largest scaling factor $a_{n,i}$ resp.\ $b_{n,j}$.  
However, we observe two interesting phenomena: 
First, for the stable components, it is not the scaling factor $b_{n,j}$ itself which distinguishes the dominant component, but rather the power $b_{n,j}^{\beta_j}$, revealing an interesting interplay between the scale and the tail decay.
Secondly, if $f$ vanishes near zero, then the effect of the Gaussian component is asymptotically negligible compared to the stable components, even if $\max_i a_{n,i} \gg \max_j b_{n,j}$.
This is made precise by the following theorem. 

\begin{theorem}[Variance bound]\label{thm:variance-bound}
	Suppose that the exponent of the tail decay satisfies $\delta^{\star} = \max(2\delta_0, \delta_1\beta_1,\ldots, \delta_{q_2}\beta_{q_2})<-2$, and that  $a_{n,i}$ and $b_{n,j}$ are bounded.
	Then  for all $n$, and for all $f \in \mathfrak{F}_{0}$:
	\begin{align*}
		\Var (S_n(f))
		&\leq \frac{C}{n} \left[\sum_{i=1}^{q_1} a_{n,i}^4 + \sum_{j=1}^{q_2} b_{n,j}^{\beta_j}\right].
	\end{align*}
	If $\eta>0$, then for any $\lambda>0$, there exists a constant $C\in (0,\infty)$ such that for all $n$, and for all $f \in \mathfrak{F}_{\eta}^0$,
	\begin{align*}
		\Var (S_n(f)) 
		&\leq \frac{C}{n} \left[ \left(\sum_{i=1}^{q_1} a_{n,i}^2\right) \exp\left( -\frac{\eta^2\lambda^2}{2d \sigma^2 \sum_{i=1}^{q_1} a_{n,i}^2 }\right) + \sum_{j=1}^{q_2} b_{n,j}^{\beta_j} \right].
	\end{align*}
\end{theorem}
\noindent
We remark that the exponential term in the variance bound for $\eta>0$ vanishes rapidly if $a_{n,i} \to 0 $ at a polynomial speed, such that the contribution of the Gaussian component is negligible in this case. 
That is, the smooth threshold $f\in \mathfrak{F}_\eta^0$ effectively filters out the Gaussian component.

As a consequence of Theorem \ref{thm:variance-bound}, we obtain a law of large numbers with rate of convergence. In particular, 
\begin{align}
    \frac{1}{n} \sum_{t=1}^n f(X_{t,n}) 
    = \E f(X_{t,n}) + \mathcal{O}_{\mathbb{P}}\left( \frac{1}{\sqrt{n}}\sqrt{\sum_{i=1}^{q_1} a_{n,i}^4 + \sum_{j=1}^{q_2} b_{n,j}^{\beta_j}} \right). \label{eqn:LLN}
\end{align}
We highlight that the centering term $\E f(X_{t,n})$ is not explicit as a function of $\theta$, and also depends on the index $n$.

The next theorem is the main result of this section. It demonstrates a central limit theorem for the statistic $S_n(f)$ in various settings.   

\begin{theorem}[Central limit theorem]\label{thm:clt}
	Let $i^{\star}\in\{1,\ldots, q_1\}$ such that $a_{n,i^{\star}}^2 \gg \max_{i\neq i^{\star}} a_{n,i}^2$, and let $j^{\star}\in\{1,\ldots, q_2\}$ such that $b_{n,j^{\star}}^{\beta_{j^{\star}}} \gg \max_{j\neq j^{\star}} b_{n,j}^{\beta_j}$. 
	Assume that $\max(2\delta_0, \beta_1\delta_1,\ldots, \beta_{q_2}\delta_{q_2}) < -2$, and suppose furthermore that 
	\begin{align*}
		a_{n,i^{\star}} \to a \in [0,\infty),\qquad b_{n,j^{\star}}\to b\in[0,\infty).
	\end{align*}	
	Fix some $f\in\mathfrak{F}_0$, and let one of the following conditions hold:
	\begin{enumerate}[(i)]
		\item $a>0$, $b=0$;
		\item $a=0$, $b>0$;
		\item $a=b=0$, and $a_{n,i^{\star}}^4 \ll b_{n,j^{\star}}^{\beta_{j^{\star}}}$, $nb_{n,j^{\star}}^{\beta_{j^{\star}}}\to \infty$;
		\item $a=b=0$, and $a_{n,i^{\star}}^4 \gg b_{n,j^{\star}}^{\beta_j^{\star}} $, $n a_{n,i^{\star}}^4 \to\infty$, and $D^2f(0)\neq 0$.
	\end{enumerate}
	Then
	\begin{align*}
		\frac{\sqrt{n}}{\sqrt{\max(a_{n,i^{\star}}^4, b_{n,j^{\star}}^{\beta_j^{\star}}) }} S_n(f) \wconv \mathcal{N}(0,\xi^2), \\
	\end{align*}
	where the asymptotic variance $\xi^2$ is given in \eqref{eqn:asymp-var}.
	
	Moreover, if (case (v))
	\begin{itemize}
		\item $a=b=0$, $na_{n,i^{\star}}^4\to\infty$, $nb_{n,j^{\star}}^{\beta_{j^{\star}}}\to\infty$,
		\item $f_1\in\mathfrak{F}_0$ such that $D^2f(0)\neq 0$, and $f_2\in\mathfrak{F}_\eta^0$ for some $\eta>0$
		\item for some $\lambda\in(0,1)$,
		\begin{align*}
			\exp\left( -\frac{\eta^2\lambda^2}{2d \sigma^2 \sum_{i=1}^{q_1} a_{n,i}^2 } \right) \ll b_{n,j^{\star}}^{\beta_{j^{\star}}} \ll a_{n,i^{\star}}^4,
		\end{align*}
	\end{itemize}
	then
	\begin{align*}
		\sqrt{n}
		\begin{pmatrix} \sqrt{a_{n,i^{\star}}^4} & 0 \\ 0 & \sqrt{ b_{n,j^{\star}}^{\beta_{j^{\star}}}} \end{pmatrix}^{-1} 
		 \begin{pmatrix}
			S_n(f_1) \\  S_n(f_2)
		\end{pmatrix}
		\;\wconv\; \mathcal{N}\left( 0, \begin{pmatrix}
			\gamma^2_{f_1,0} & 0 \\ 0 & \zeta^2_{f_2,0}
		\end{pmatrix} \right),
	\end{align*}
	and $\gamma^2_{f_1,0}$ and $\zeta^2_{f_2,0}$ are defined in \eqref{eqn:asymp-var}.
\end{theorem}

\noindent
We remark that limit theorems for the sum of Gaussian and non-Gaussian fractional processes are rather rare in the literature as the mathematical tools are different in the two cases. Indeed,  Theorem \ref{thm:clt} seems to be the first result in this setting. 

There exist numerous central and non-central limit theorems for statistics of lfsm's or more general L\'evy moving average processes, all of them focusing on the case $q=1$.  \cite{Pipiras2003} considered a bounded function $f$ and investigated some extensions to non-bounded functions in 
\cite{pipiras2007}. The asymptotic theory for power variation statistics of lfsm has been studied in \cite{Basse-OConnor2017, basse-oconnor2017critical} and later extended to more general functionals in \cite{Basse-OConnor2019}. Further results on high frequency statistics of lfsm's and related models can be found in \cite{Mazur2018}, \cite{basse-oconnor2018}  and \cite{Azmoodeh2020}. 

Notably, all available results for the high-frequency regime scale the increments by $\Delta_n^{-H}$, where $H$ is the Hurst exponent of the lfsm (cf.  Theorem \ref{thm:clt}(ii)). 
This would correspond to the scaling $\Delta_n^{-H_1}$ in our mixed setting.
In contrast, we also allow for different scalings, such that the argument of the nonlinearity vanishes (cases (iii) and (iv)), and we also allow for a dominant Gaussian component (case (i)).
The special case of smooth thresholds in Theorem \ref{thm:clt}(v), has been studied for Lévy processes in \cite{Mies2019}, but not for processes with dependent increments.

The proof of the central limit theorem is performed by approximating $X_{t,n}$ via an $m$-dependent sequence.  
This is similar in spirit to the methodology of \cite{Pipiras2003} and \cite{Basse-OConnor2019}. 
However, in the setting of 
Theorem \ref{thm:clt}, we obtain more refined bounds on the autocovariances, which explicitly account for the scaling terms $a_{n,i}$ and $b_{n,j}$; see Lemma \ref{lem:autocov-decay} in the appendix.

\begin{remark} \rm 
One of the key conditions for the validity of the central limit theorem is the assumption $\max(2\delta_0, \beta_1\delta_1,\ldots, \beta_{q_2}\delta_{q_2}) < -2$. Here we would like to compare this condition with classical conditions in central limit theorems for statistics of Gaussian and non-Gaussian fractional processes. 

We consider the non-Gaussian case first, i.e. $q_1=0$, $q_2=1$,
and let the kernel $g$ be given as in \eqref{eqn:X-multiscale}.
Then the decay rate of the kernel is given as $\delta=H-k-1/\beta$ and the condition of Theorem \ref{thm:clt}
becomes $H\beta-k\beta-1<-2$. Since the function $f$ is assumed to be even, its \textit{Appell rank} is 2 or larger, and the latter condition coincides with the one from \cite{Basse-OConnor2019}. 

On the other hand, when considering the Gaussian case $q_1=1$ and $q_2=0$, our condition seems to be suboptimal (although sufficient for statistical applications). Indeed, when $k=1$ 
the condition translates to $H<1/2$. However, when dealing with functions of \textit{Hermite rank} 2, the optimal condition is known to be $H<3/4$. \qed
\end{remark}

\section{Estimation of the mixed fractional stable motion}\label{sec:estimation}

Observing the representation \eqref{eqn:X-multiscale}, the characteristic function of the $k$-th order increments takes the form
\begin{align}
	\E \cos\left( \lambda X_{l,n,\gamma} \right)
	&= \exp(-\psi_n(\lambda,\gamma)),
	\qquad \psi_n(\lambda,\gamma) 
	= \sum_{j=1}^q \widetilde{b}_j\,  \lambda^{\beta_j} \,  \gamma^{\beta_j H_j} \, \Delta_n^{\beta_j H_j},\label{eqn:characteristic-exponent} \\
	\widetilde{b}_j &= b_j^{\beta_j} \int g_j(s)^{\beta_j}\, ds. 
	\nonumber
\end{align}
Since there is a one-to-one correspondence between  $(b_j,H_j,\beta_j)$  and $(\widetilde{b}_j, H_j,\beta_j)$, 
we focus on the estimation of the latter. 
We summarize the unknown parameters as
\begin{align*}
	\theta &= \left(\widetilde{b}_1,H_1, \beta_1,\ldots, \widetilde{b}_{q}, H_{q},\beta_q\right) \in \Theta,\\
	\Theta &= \left\{\theta \in ((0,\infty) \times (0,1)\times (0,2])^{q} : H_1(\theta)<\ldots<H_q(\theta)\right\}\subset \R^{3q}.
\end{align*}
In the sequel, $\theta_0$ will denote the true but unknown parameter vector.

Formula \eqref{eqn:characteristic-exponent} for the characteristic exponents suggests that one can identify the stability indices $\beta_j$ by varying the scaling factor $\lambda$. 
This property can also be exploited for parameter estimation for Lévy processes and related semimartingales, see \cite{ait2012identifying,reiss2013testing,bull2016near,Mies2019}.
Moreover, formula \eqref{eqn:characteristic-exponent} reveals that one can identify the Hurst exponents $H_j$ by varying the lag $\gamma$. 
This leads to the intuitive idea of using empirical characteristic functions to estimate the parameters of the model, which has been first proposed
in \cite{Mazur2018} in the setting $q=1$ (see also \cite{Ljungdahl2020}). However, their method has several drawbacks, one of which is the need of pre-estimation of the smallest Hurst exponent $H_1$. It leads to singular limit distributions (cf. \cite[Theorem 3.1]{Mazur2018}) and its performance is far from obvious in the mixed case $q>1$.  

Instead we follow a different strategy, which relies on a system
of estimating equations, and does not require prior information about the parameters. Section \ref{sec:GMM} presents the asymptotic theory for the new approach. Section \ref{sec:threshold} demonstrates a further refinement of the estimation method, the smooth threshold, which specifically accounts for the presence of the Gaussian component in the model.

\subsection{Estimation via adaptive moment equations}\label{sec:GMM}

We consider an even  Schwartz function $f:\R\to\R$,
i.e.\ $f\in C^\infty(\R)$ such that 
$$\sup_{x\in \R} |x|^p |f^{(j)}(x)|<\infty$$ 
for all $p>0$ and all $j\geq 0$. For $\lambda_r \in \R$ and $\gamma_r \in \N$ we define the statistic
\begin{align*}
    S_n\left(f;u_n(\theta); (\lambda_r, \gamma_r)_{r=1}^d\right) =\left[ \frac{1}{n} \sum_{l=1}^n f(u_n(\theta) \lambda_{r} X_{l,n,\gamma_r})\right]_{r=1}^{d}, \qquad u_{n}(\theta)= \Delta_n^{-H_1(\theta)}.
\end{align*}
We introduce the random vector
\begin{align}
	\mathcal{G}_n(\theta) = S_n\left(f;u_n(\theta);(\lambda_r, \gamma_r)_{r=1}^{3q}\right) - \E S_n\left(f;u_n(\theta); (\lambda_r, \gamma_r)_{r=1}^{3q} \right),
	 \label{eqn:ee-schwartz}
\end{align}
for some $\lambda_r, \gamma_r, r=1,\ldots, 3q$. Note that the expectations in \eqref{eqn:ee-schwartz} can be determined numerically for any $\theta \in \Theta$, and hence the quantity $\mathcal{G}_n(\theta)$ can be computed from data. We now define
an estimator $\widehat{\theta}_n$ of the unknown parameter $\theta_0$ as a solution of the equation
\begin{align} \label{hattheta}
    \mathcal{G}_n(\widehat{\theta}_n) = 0.
\end{align}

\begin{remark} \rm
To determine statistical properties of $\widehat{\theta}_n$ it is convenient to solve the equation $\mathcal{G}_n(\widehat{\theta}_n) = 0$ over an open set rather than over $\Theta$. This can be achieved by extending the parameter set $(0,2]$ for $\beta$ to an arbitrary open set containing $(0,2]$. While a $\beta$-stable distribution does not exist for $\beta>2$, the expectation  $\E_\theta f(\lambda X_{l,n,\gamma})$ can be extended to all values $\beta \in (0,\infty)$. Indeed, we obtain via Fourier transform:
\begin{align}
	\E_\theta f(\lambda X_{l,n,\gamma}) &= \int \widehat{f}(v)\, \exp(-\psi_n(v\lambda, \gamma,\theta))\, dv, \label{eqn:expect-fourier}\\
	\widehat{f}(v) &= \frac{1}{2\pi}\int \cos(v x) f(x)\, dx. \nonumber
\end{align}
Since $f$ is a Schwartz function, its Fourier transform $\widehat{f}(v)$ is a Schwartz function as well, and in particular decays rapidly as $v\to\infty$. 
On the other hand, definition \eqref{eqn:characteristic-exponent} of the function $\psi_n$ is also sensible for $\beta_j>2$, such that the integral \eqref{eqn:expect-fourier} is finite. 
In the sequel, we use \eqref{eqn:expect-fourier} as definition of the expectation for the case $\beta_j>2$. In practice we extend the original parameter set $\Theta$ to an arbitrary open set that contains $\Theta$. 
\qed
\end{remark}

\noindent
To formulate our result on the asymptotic behavior of the estimator $\widehat{\theta}_n$, we introduce the following rate matrix:
\begin{align*}
	\overline{C}^{j}_n(\theta) &= \frac{\Delta_n^{\beta_j(\theta) (H_1(\theta)-H_j(\theta))}}{\sqrt{n}} \begin{pmatrix}
		1 & -\widetilde{b}_j\beta_j \log(\Delta_n)\mathds{1}_{\{j\neq 1\}} & \widetilde{b}_j  H_j \log(\Delta_n)\mathds{1}_{\{j\neq 1\}}  \\
		0 & 1            & -H_j/\beta_j\\
		0 & 0 			 &  1
	\end{pmatrix}  \in \R^{3\times 3} \\
	\overline{C}_n(\theta)
	&= \diag(\overline{C}^{1}_n,\ldots, \overline{C}^{q}_n) \in\R^{3q \times 3q}.
\end{align*}
We also define the matrix $\overline{W}(\theta)\in\R^{3q \times 3q}$ by
\begin{align*}
	\overline{W}(\theta)_{i,r} = \int \widehat{f}(v) \exp\left(-\widetilde{b}_1|\lambda_r v|^{\beta_1} \gamma_r^{\beta_1H_1} \right)\partial_{\theta_i} \left(\sum_{j=1}^q \widetilde{b}_j \gamma_r^{\beta_j H_j} |\lambda_r v|^{\beta_j}\right) \, dv.
\end{align*}
The main result of this section is the following central limit theorem.

\begin{theorem}\label{thm:clt-GMM}
	Suppose that the order of differencing $k$ is large enough, such that
	\begin{align*}
		k > H_j + \frac{1}{\beta_j},\quad j=1,\ldots, q.
	\end{align*}
	Choose $\lambda_r\in\R$, $\gamma_r \in\N$, $r=1,\ldots, 3q$, such that the matrix $\overline{W}(\theta_0)$ is regular.
	Assume the identifiability condition
	\begin{align}
		\frac{\Delta_n^{\beta_j(\theta_0)[H_1(\theta_0)-H_j(\theta_0)]}}{\sqrt{n}} \ll \frac{1}{|\log \Delta_n|^2}, \quad j=1,\ldots, q. \label{eqn:GMM-identifiability}
	\end{align}
	Then there exists a sequence of random vectors $\widehat{\theta}_n$ such that $P(\mathcal{G}_n(\widehat{\theta}_n)=0) \to 1$ which satisfies 
	\begin{align*}
		\overline{C}_n^{-1} (\widehat{\theta}_n-\theta_0) 
		\;\wconv \; \mathcal{N}\left(0, \Sigma \right),
	\end{align*}
    with asymptotic covariance matrix $\Sigma = \overline{W}(\theta_0)^{-1} \widetilde{\Sigma}\,(\overline{W}(\theta_0)^{-1})^T \in \R^{3q\times 3q}$, and $\widetilde{\Sigma} = \widetilde{\Sigma}(\theta_0)\in\R^{3q\times 3q}$ is given by formula \eqref{eqn:asymp-cov-G} in the appendix.
\end{theorem}

\noindent
The proof of Theorem \ref{thm:clt-GMM} is based upon the asymptotic theory for the random vector $\mathcal{G}_n(\theta_0)$, which is obtained via an application of Theorem \ref{thm:clt} from the previous section, and a general theory for solutions of estimating equations, see Theorem \ref{thm:ee-consistency}.  

From the rate matrix $\overline{C}_n(\theta_0)$, we may  derive constraints on the parameter $\theta_0$ to ensure that all components of the process may be estimated consistently.
If $\Delta_n = n^{-\rho}$ for some $\rho\geq 0$, then $\|\widehat{\theta}_n-\theta_0\|\to 0$ if and only if
\begin{align}
	H_j < H_1 + \frac{1}{2\rho \beta_j}. \label{eqn:identifiability}
\end{align}
For the classical high-frequency regime $\rho=1$, i.e.\ $\Delta_n = 1/n$, this identifiability condition is sharp for two prominent special cases:
\begin{itemize}
	\item[(i)] If $\beta_j=2$ for all $j$, then each $Y_t^{H_j, 2}$ is a fractional Brownian motion, and \eqref{eqn:identifiability} recovers the identifiability condition of \cite{VanZanten2007}.
	\item[(ii)] If $H_j = 1/\beta_j$ and for all $j$, then each $Y_t^{1/\beta_j, \beta_j}$ is a $\beta_j$-stable Lévy process. If additionally $\beta_j<2$ for all $j$, then \eqref{eqn:identifiability} recovers the identifiability condition of \cite{ait2012identifying}.
\end{itemize}
However, condition \eqref{eqn:identifiability} is not sharp when the process contains both, Gaussian and non-Gaussian components. 
If we consider the special case $q=2$ and $\beta_1=2$, $H_1 =
1/2$ , $\beta_2<2$, $H_2 = 1/\beta_2$, it is well known 
that all parameters are identifiable as $\Delta_n\to 0$, see e.g.\ \cite{ait2009estimating}.
However, the moment estimator based on $\mathcal{G}_n$ requires the restriction $\beta_2>1$. We will address this issue in the next subsection.
Another disadvantage of the estimator based on $\mathcal{G}_n$ is that regularity of $\overline{W}(\theta_0)$ is non-trivial to verify.
This issue will also be addressed by the alternative estimator presented in the next subsection.

\subsection{Estimation via smooth thresholds}\label{sec:threshold}

In this section, we improve upon the estimator of Section \ref{sec:GMM} by exploiting the Gaussianity of some components, in order to obtain weaker conditions for identifiability and faster rates of convergence. 
Here, to demonstrate the main ideas, we restrict our attention to a simplified model consisting of two Gaussian components and two non-Gaussian components, and we assume that the Gaussian component is dominant. To be specific, we study the model 
\begin{align}
	X_t &= a_1 Y_t^{H_1, 2} + a_2 Y_t^{H_2,2} + b_1 Y_t^{\overline{H}_1, \beta_1} + b_2 Y_t^{\overline{H}_2, \beta_2}, \label{eqn:model-gauss} \\[1.5 ex]
	H_1 &< \min(H_1, \overline{H}_1,\overline{H}_2), \qquad
	(\overline{H}_1-H_1) \beta_1 < (\overline{H}_2-H_1) \beta_2. \nonumber
\end{align}
Note that the fractional Brownian motion $Y^{H_1, 2}$
is the dominant component on small scales, in the sense that $\gamma^{-H_1} X_{\gamma t}\wconv a_1 Y_t^{H_1,2}$ as $\gamma\downarrow 0$. 
The model has ten parameters and we denote the corresponding parameter set with the constraints as above by $\overline{\Theta}\subset \R^{10}$.
As before (see the transform in \eqref{eqn:characteristic-exponent}) we switch to an equivalent representation of scale parameters and obtain 
\begin{align*}
	\theta=(\widetilde{a}_1, H_1, \widetilde{a}_2, H_2, \widetilde{b}_1, \overline{H}_1, \beta_1, \widetilde{b}_2, \overline{H}_2, \beta_2)\in\R^{10}.
\end{align*}
As the Gaussian part $Y^{H_1, 2}$ is dominant, it is particularly hard to estimate the non-Gaussian components of the model. 
To improve the results from the previous section, we employ the idea of a smooth threshold, which aims at filtering out the continuous part at small scales.
To this end, we consider two Schwartz functions $f_1,f_2:\R\to\R$ such that, for some $\eta>0$,
\begin{align*}
	f_1''(0)\neq 0, \qquad f_2(x)=0 \text{ for } |x|\leq \eta.
\end{align*}
Moreover, let $\lambda_{r,n}(\theta) = \lambda_r u_n(\theta)$ for $\lambda_r\in\R$, and $\gamma_r\in\N$, $r=1,\ldots, 10$. 
We suggest to estimate $\theta\in\R^{10}$ as a solution $\widehat{\theta}_n$ of the estimating equation
\begin{align}
	\mathcal{H}_n(\widehat{\theta}_n) &= 0, \qquad \text{where} \label{eqn:ee-threshold}\\
	\mathcal{H}_n(\theta) &= 
	 \begin{pmatrix}S_n\left(f_1; u_n(\theta); (\lambda_r, \gamma_r)_{r=1}^4\right) - \E S_n\left(f_1; u_n(\theta); (\lambda_r, \gamma_r)_{r=1}^4\right)
		\\[1.5 ex]
		S_n\left(f_2; u_n(\theta); (\lambda_r, \gamma_r)_{r=5}^{10}\right) - \E S_n\left(f_2; u_n(\theta); (\lambda_r, \gamma_r)_{r=5}^{10}\right)
	\end{pmatrix} \in \R^{10}. \nonumber
\end{align}
The first set of moments based on $f_1$ serve to identify the Gaussian components, just as in Section \ref{sec:GMM}.
The function $f_2$ represents a smooth threshold, as $f_2(u_n  X_{l,n,\gamma})\neq 0$ if and only if $X_{l,n,\gamma}>1/u_n$. 
The scaling factor $u_n$ may thus be seen as a reciprocal threshold value, and the moments based on the smooth threshold $f_2$ identify the stable components.
In contrast to Section \ref{sec:GMM}, we choose
\begin{align}
    u_n(\theta) &= w_n \Delta_n^{-H_1(\theta)}, \nonumber \\
    \text{where}\quad
    \Delta_n^{\epsilon} &\ll \; w_n \; \leq \frac{\eta}{9\sigma} \frac{1}{\sqrt{|\log \Delta_n|}}, \qquad \forall \epsilon>0. \label{eqn:wn}
\end{align}
The crucial observation is that, for this choice of $u_n$, the variance of the smooth thresholds is of a smaller order compared to the empirical characteristic function, see Theorem \ref{thm:variance-bound}.
This occurs because the threshold is asymptotically unaffected by the dominant Gaussian component, and the asymptotic distribution is driven by the smaller stable component.
The same idea has been employed in \cite{Mies2019} to construct an estimator for Lévy processes, i.e.\ $H_j = 1/\beta_j$.

To formulate our asymptotic result about the estimating equation \eqref{eqn:ee-threshold}, we define the rate matrix $R_n(\theta)$ as
\begin{align*}
	R_n^j(\theta) &= \frac{\Delta_n^{2 (H_1-H_j)}}{\sqrt{n}} \begin{pmatrix}
		1 & -2\widetilde{a}_j \log(\Delta_n)\mathds{1}_{\{j\neq 1\}}  \\
		0 & 1 
	\end{pmatrix}  \in \R^{2\times 2}, \qquad j=1,2, \\
	\overline{R}^{j}_n(\theta) &= \frac{w_n^{\frac{\beta_1}{2}-\beta_j} \Delta_n^{\frac{\beta_1 (\overline{H}_1- H_1)}{2} - \beta_j(\overline{H}_j-H_1)} }{\sqrt{n}}  \begin{pmatrix}
		1 & 0 & -\widetilde{b}_j \log|w_n|   \\
		0 & 1 & -\overline{H}_j/\beta_j\\
		0 & 0 &  1
	\end{pmatrix}  \in \R^{3\times 3},\qquad j=1,2, \\
	R_n(\theta)&=  \diag(R^{1}_n, R^{2}_n, \overline{R}^{1}_n, \overline{R}^{2}_n) \in\R^{10 \times 10}.
\end{align*}
Furthermore, recalling the notation $\widehat{f}(v) = \frac{1}{2\pi} \int_{-\infty}^\infty f(x) e^{-ivx}\, dx$ of $f$, we define
\begin{align*}
	\underline{W}(\theta) &= \begin{pmatrix}
		\underline{W}_1(\theta) & 0 \\ 0 & \underline{W}_2(\theta)
	\end{pmatrix} \in \R^{10\times 10},\\
	\underline{W}_1(\theta)_{r,i} &= \int \widehat{f}_1(v) \partial_{\theta_i} \left(\sum_{j=1}^2 \widetilde{a}_j \gamma_r^{2 H_j} |\lambda_r v|^{2}\right) \, dv \\
	&= f_1''(0) |\lambda_r|^2 \partial_{\theta_i} \left(\sum_{j=1}^2 \widetilde{a}_j \gamma_r^{2 H_j}\right) , \qquad r,i=1,\ldots, 4, \\ 	
	\underline{W}_2(\theta)_{r-4,i-4} &= \int \widehat{f}_2(v) \partial_{\theta_i} \left(\sum_{j=1}^2 \widetilde{b}_j \gamma_r^{\beta_j \overline{H}_j} |\lambda_r v|^{\beta_j} \right) \, dv, \qquad r,i=5,\ldots, 10.
\end{align*}
The main result of this section is the following theorem.

\begin{theorem}\label{thm:clt-threshold-estimator}
	Let $\Delta_n \sim n^{-\rho}$ for some $\rho>0$.
	Assume that the order of differencing is large enough, such that
	\begin{align*}
		k > \max\left(H_j(\theta_0) + \frac{1}{2}, \quad \overline{H}_j(\theta_0)+\frac{1}{\beta_j(\theta_0)}\right),\quad j=1,2.
	\end{align*}
	Assume furthermore that the following identifiability condition holds:
	\begin{align}
		\begin{split}
			H_2 &< H_1 + \frac{1}{4\rho}, \\
			\overline{H}_1 &< H_1 + \frac{1}{\rho \beta_1}, \\
			\overline{H}_2 &< H_1 + \frac{1}{2\rho \beta_2} + \frac{\beta_1}{2\beta_2} (\overline{H}_1-H_1), 
		\end{split}\label{eqn:identifiability-threshold}
	\end{align}
	and suppose that $\underline{W}(\theta_0)$ is a regular matrix.
    Then there exists a sequence of random vectors $\widehat{\theta}_n$ such that $\mathbb P(\mathcal{H}_n(\widehat{\theta}_n)=0) \to 1$ which satisfies
	\begin{align*}
		R_n(\theta_0)^{-1} \left(\widehat{\theta}_n-\theta_0\right) 
		\;\wconv \; \mathcal{N}\left(0, \Sigma \right),
	\end{align*}
	with asymptotic covariance matrix $\Sigma = \underline{W}(\theta_0)^{-1} \diag(\Sigma_1,\Sigma_2)\,(\underline{W}(\theta_0)^{-1})^T \in \R^{10\times 10}$, and $\Sigma_1\in\R^{4\times 4}, \Sigma_2\in\R^{10\times 10}$ are given by formula \eqref{eqn:asymp-cov-1} resp.\ \eqref{eqn:asymp-cov-2} in the appendix.
\end{theorem}
\noindent
Regularity of the matrix $\underline{W}(\theta_0)$ implicitly imposes assumptions on both, the parameters $\theta_0$ and the statistical design in terms of $\lambda_r$ and $\gamma_r$. 
If we choose the latter hyperparameters carefully, the regularity holds for almost all parameters $\theta$.

\begin{proposition}\label{prop:regularity}
    Let $f_2\geq 0$, $f_2\not\equiv 0$.
    \begin{enumerate}[(i)]
        \item Choose $(\gamma_1,\ldots, \gamma_4)=(1,2,4,8)$, and $\lambda_r\neq 0$ for $r=1,\ldots, 4$. 
        Then $\underline{W}_1(\theta)$ is regular for all $\theta\in\Theta$.
        \item Choose $(\gamma_5,\ldots,\gamma_{10}) = (1,2,2,4,4,8)$, and $(\lambda_5,\ldots, \lambda_{10}) = (1,2,4,8,16,32)$. 
        Then $\underline{W}_2(\theta)$ is regular if
        \begin{align*}
            \beta_1 (1+H_1) \neq \beta_2 (1+H_2) \qquad \text{and}\qquad \beta_1(2+H_1) \neq \beta_2(2+H_2).
        \end{align*}
        \item Choose $(\gamma_5,\ldots,\gamma_{10}) = (1,2,2,4,8,8)$, and $(\lambda_5,\ldots, \lambda_{10}) = (1,1,2,1,1,2)$. 
        Then $\underline{W}_2(\theta)$ is regular if
        \begin{align*}
            H_1\beta_1 \neq H_2\beta_2.
        \end{align*}
    \end{enumerate}
\end{proposition}
\noindent
To circumvent the singular edge cases unveiled in Proposition \ref{prop:regularity}, a potential solution would be to use both sets of moments, and solve a nonlinear least squares problem instead of the system of estimating equations.
However, this extension is beyond the scope of the present paper.

\begin{remark}
    By virtue of Proposition \ref{prop:regularity}, the regularity of $\underline{W}(\theta_0)$ is rather simple to verify. 
    In contrast, the asymptotic analysis of the estimator presented in Section \ref{sec:GMM} requires regularity of the matrix $\overline{W}(\theta_0)$, which is rather unwieldy.
    This difference can be traced back to the different rescaling factors $u_n(\theta)$: 
    in Section \ref{sec:GMM}, for $u_n(\theta) = \Delta_n^{-H_1}$, the rescaled increments $u_n(\theta_0) X_{l,n,\gamma}$ converge weakly towards a $\beta_1$-stable random variable; in Section \ref{sec:threshold}, for $u_n(\theta) \ll \Delta_n^{-H_1}$, we have $u_n(\theta_0) X_{l,n,\gamma} \pconv 0$. 
    Following the same strategy as in the proof of Theorem \ref{thm:clt-threshold-estimator}, it is also possible to derive a central limit theorem for the estimator of Section \ref{sec:GMM} with scaling factor $u_n(\theta) = w_n \Delta_n^{-H_1}$ for $w_n\to 0$, which will no longer require regularity of $\overline{W}(\theta_0)$, but rather of a matrix similar to $\underline{W}_2$.
    On the other hand, the rate will be slightly worse by a factor $w_n^\rho$ for some $\rho>0$. 
    Hence, we do not pursue this direction any further.
\end{remark}

\subsection{Discussion}
\noindent
For $\rho=1$, and for estimation of the smoother Gaussian component, i.e.\ $H_2$, the new smooth thresholding estimator presented in Section \ref{sec:threshold} still matches the identifiability restrictions of  \cite{VanZanten2007} and  \cite{ait2012identifying}. 
In contrast to Section \ref{sec:GMM}, we may now also identify the stable Lévy processes if the dominant component is Gaussian, without the restriction $\beta_j>1$.
This is possible because the Gaussian component is effectively filtered by the smooth threshold. A direct comparison with the identifiability condition \eqref{eqn:identifiability} from Section \ref{sec:GMM} is presented in Table \ref{tab:identifiability}.
Since $\overline{H}_1 > H_1$, the conditions for the smooth threshold estimator, i.e.\ based on the estimating equations $\mathcal{H}_n$, are strictly weaker.

\begin{table}[tb]
	\centering
	\begin{tabular}{c|l|l}
		& Without threshold ($\mathcal{G}_n$) & With threshold ($\mathcal{H}_n$) \\ \hline\hline
		$H_1$ & $(0,1)$ & $(0,1)$ \\
		$H_2$ & $< H_1+\tfrac{1}{4}$ & $< H_1+\tfrac{1}{4}$ \\
		$\overline{H}_1$ & $<H_1+\tfrac{1}{2\beta_1}$ & $<H_1+\tfrac{1}{\beta_1}$  \\
		$\overline{H}_2$ & $<H_1+\tfrac{1}{2\beta_2}$ & $<H_1+\tfrac{1}{2\beta_2} + \frac{\beta_1}{2\beta_2}(\overline{H}_1-H_1)$
	\end{tabular}
	\caption{Identifiability conditions for the various components, in the sampling regime $\Delta_n = 1/n$, i.e.\ $\rho=1$.}
	\label{tab:identifiability}
\end{table}

\begin{table}[tb]
	\centering
	\begin{tabular}{c|l|l}
		& Without threshold ($\mathcal{G}_n$) & With threshold ($\mathcal{H}_n$) \\ \hline\hline
		$H_1$ & $n^{-\frac{1}{2}}$ & $n^{-\frac{1}{2}}$ \\
		$H_2$ & $n^{2(H_2-H_1)-\frac{1}{2}}$ & $n^{2(H_2-H_1)-\frac{1}{2}}$ \\
		$\overline{H}_1$ &   $n^{\beta_1(\overline{H}_1-H_1) - \frac{1}{2} }$ &  $n^{\frac{\beta_1}{2}(\overline{H}_1-H_1) - \frac{1}{2}  } \, (\log n)^{\frac{\beta_1}{4}}$   \\
		$\overline{H}_2$ & $n^{\beta_2(\overline{H}_2-H_1) - \frac{1}{2} }$   & $n^{ \beta_2 (\overline{H}_2-H_1) -\frac{\beta_1}{2}(\overline{H}_1-H_1) - \frac{1}{2}  } \, (\log n)^{\frac{\beta_2}{2} - \frac{\beta_1}{4}}$
	\end{tabular}
	\caption{Rates of convergence in the sampling regime $\Delta_n = 1/n$, i.e.\ $\rho=1$.}
	\label{tab:rates}
\end{table}

The rates of convergence of both estimators are presented in Table \ref{tab:rates}.
Again, the thresholding estimator based on $\mathcal{H}_n$ is strictly better than the estimator based on $\mathcal{G}_n$.
Furthermore, we can assess the rates for the special Lévy case $H_1=1/2$ and $\overline{H}_1 = 1/\beta_2$, $\overline{H}_2 = 1/\beta_2$, with $\beta_1>\beta_2$, which has been studied by \cite{Mies2019} and \cite{ait2012identifying}. 
In this case, Theorem \ref{thm:clt-threshold-estimator} yields, for $\Delta_n = 1/n$,
\begin{align*}
	 \widehat{\beta}_1 - \beta_1    = \mathcal{O}_{\mathbb{P}}\left( (n/\log n)^{-\frac{\beta_1}{4}}  \right), \qquad
	 \widehat{\beta}_2 - \beta_2    = \mathcal{O}_{\mathbb{P}}\left( (n/\log n)^{\frac{\beta_1}{4} - \frac{\beta_2}{2}}  \right).
\end{align*}
These are the same rates of convergence which have been achieved for the Lévy case, and which are conjectured to be optimal for this setting; see the discussion in \cite{Mies2019}.\footnote{The formulas (3.2) and (3.3) in \cite{Mies2019} contain an error where the term $n\log n$ should correctly be $n/\log n$. The latter rate is obtained by substituting the value for $u_n\asymp \sqrt{n/\log n}$ therein.}

As another benchmark, we may investigate the regime $H_2=\frac{1}{2}$, $H_1 \in (1/4,\, 1/2)$, which corresponds to a sum of a classical Brownian motion and a rougher fractional Brownian motion.
This setting has been studied by \cite{chong2021}, in a generalized model allowing for additional nonstationarity. 
Both of our estimators, either based on $\mathcal{G}_n$ or on $\mathcal{H}_n$, yield
\begin{align*}
	\widehat{H}_1 - H_1 = \mathcal{O}_{\mathbb{P}}\left( n^{-\frac{1}{2}} \right), \qquad
	\widehat{\widetilde{a}}_1 - \widetilde{a}_1 &= \mathcal{O}_{\mathbb{P}}\left( n^{-\frac{1}{2}} \right), \qquad
	\widehat{\widetilde{a}}_2 - \widetilde{a}_2 &= \mathcal{O}_{\mathbb{P}}\left( n^{\frac{1}{2} - 2 H_1} \log n \right).
\end{align*}
The rate for $H_1$ is identical to the rate of Chong et al., see Theorem 3.11 therein, and our rate for $\widetilde{a}_1$ is faster by a factor $\log n$.
On the other hand, our rate for $\widetilde{a}_2$ is slower by a factor $\log n$, which is due to the fact that in our setting, $H_2$ is unknown and needs to be estimated as well.

Except for the special cases discussed above, there are currently no benchmarks for the mixed fractional stable motion \eqref{eqn:def-mixed-lfsm}. 
Hence, we do not know whether the conditions presented in Table \ref{tab:identifiability} and the rates presented in Table \ref{tab:rates} are sharp for all parameter regimes.
In fact, we conjecture the results to be not sharp in at least the following two cases.\\

\noindent
\textbf{Case (i):} 
If $\rho\in(0,1)$, we have $n\Delta_n\to\infty$, such that we effectively observe the process $X_t$ on the increasing interval $[0, T_n]$, $T_n\to\infty$. Since the linear fractional stable motion is ergodic, the same holds for $X_t$, and we should expect that all parameters can be estimated consistently in this regime.\\

\noindent
\textbf{Case (ii):}
Suppose that $H_j>1/\beta_j$ for all $j$, and that the process contains no Gaussian component, i.e.\ $\beta_j\in(1,2)$ for all $j$.
In this regime, $X_t$ admits a continuous version.
Interestingly, for different parameters $\theta\neq \theta'$ the measures $\mathbb{P}^\theta$ and $\mathbb{P}^{\theta'}$ induced by $(X_t)_{t\in[0,1]}$ on the path space $C[0,1]$ are singular. 
This is a consequence of the following identifiability result, which is new and might be of independent interest. The proof is presented in the appendix.

\begin{theorem}\label{thm:identifiability-continuous}
    Let $X_t = \sum_{j=1}^q b_j Y_t^{H_j,\beta_j}$ be a mixed fractional stable process, with $\beta_j\in(1,2)$ for all $j=1,\ldots, q$, and $H_j > 1/\beta_j$.
    Assume that the parameters $(\beta_j, H_j)$ are pairwise distinct.\footnote{The claim of the theorem is also valid if $H_j = H_{j'}$ for some $j\neq j'$, as long as $\beta_j\neq \beta_{j'}$. In this case, we sort the components of $\theta$ in lexicographical order in $H_j$ and $\beta_j$, such that $H_1 \leq H_2 \ldots$, and $\beta_j<\beta_{j+1}$ if $H_j=H_{j+1}$.}
    For any two parameters $\theta_1,\theta_2\in\Theta$, $\theta_1\neq \theta_2$, satisfying these requirements, with potentially different sizes $q(\theta_1)\neq q(\theta_2)$, the measures $P^{\theta_1}$ and $P^{\theta_2}$ induced on $C[0,1]$ are mutually singular.
\end{theorem}
\noindent
This result suggests that there should exist a sequence of consistent estimators in the high-frequency regime $\rho=1$, which would imply that our identifiability condition \eqref{eqn:identifiability} is too restrictive.
However, pairwise singularity of the measures is not sufficient for the existence of a consistent sequence of estimators, see \cite[5.1.1]{ait2014high}.
Hence Theorem \ref{thm:identifiability-continuous} does not yield a complete answer about identifiability of the parameters of the mixed lfsm.
\\


\noindent
Besides these theoretical questions, future research also needs to address various practical aspects regarding the estimation of the mixed stable motion. 
In particular, for our estimators, we need to choose the functions $f_1$ and $f_2$, the scaling parameter $u_n$, and the values for $\lambda_r$ and $\gamma_r$. 
All these hyperparameters may affect the asymptotic variance of the estimator in practice. 
Moreover, the number $q$ of components needs be determined in a data-driven way, raising questions of model selection. 
Nevertheless, the asymptotic results presented in this paper demonstrate the various intricacies of the mixed lfsm model, and they show that many models studied separately in the literature may in fact be treated by a unified statistical theory.

\appendix

\section{Estimating equations}\label{sec:ee}

The estimators proposed in Section \ref{sec:estimation} fall into the broader framework of estimating equations. 
In this section, we present some asymptotic results for solutions of general estimating equations.
Let $\Theta\subset\R^d$ be an open parameter and let $F_n(\theta)$, $\theta \in \Theta$ be a $d$-variate random vector.
Typically, $F_n(\theta)$ is a set of moment equations, and a parameter $\theta$ can be estimated by solving $F_n(\widehat{\theta}_n)=0$. A survey of the asymptotic theory of estimating equations is given by \cite{jacod2017review}.
However, for the purpose of this paper, we need to  extend their results.
It should be noted that the theory presented in this section has already been applied implicitly in the proofs in \cite{Mies2019}. 

In order to derive the asymptotic distribution of $\widehat{\theta}_n$, we impose the following conditions:
\begin{itemize}
	\item[(E.1)] There exists a sequence of regular matrices $A_n\in\R^{d\times d}$ such that $A_n F_n(\theta_0)\wconv Z$ for some random vector $Z$.
	\item[(E.2)] The mapping $\theta\mapsto F_n(\theta)$ is $C^1$.
	There exists a sequence $r_n$ of real numbers, and for each $\theta\in\Theta$ there exist sequences of regular (random) matrices $B_n(\theta)$, $C_n(\theta)$, and a regular matrix $W(\theta)$, such that
	\begin{align*}
		\sup_{\theta\in \mathbf{B}_{r_n}(\theta_0)} \|B_n(\theta) DF_n(\theta) C_n(\theta) - W(\theta)\| \pconv 0,\\
		\sup_{\theta\in \mathbf{B}_{r_n}(\theta_0)} \frac{\|C_n(\theta)\| \|B_n(\theta)A_n^{-1}\|}{r_n} \pconv 0.
	\end{align*}
	\item[(E.3)] The mapping $\theta\mapsto (B_n(\theta), C_n(\theta), W(\theta))$ is continuous in the sense that
	\begin{align*}
		\sup_{\theta\in \mathbf{B}_{r_n}(\theta_0)} \|B_n(\theta)B_n(\theta_0)^{-1} - I\| + \|C_n(\theta)C_n(\theta_0)^{-1}-I\| + \|W(\theta)-W(\theta_0)\| \pconv 0.
	\end{align*}
\end{itemize}

If we are only interested in consistency, then (E.1) could be weakened:
\begin{itemize}
	\item[(E.1)'] There exists a sequence of regular matrices $A_n\in\R^{d\times d}$ such that $A_n F_n(\theta_0) = \mathcal{O}_P(1)$.
\end{itemize}
Note also that (E.2) and (E.3) imply the following, upon setting $B_n = B_n(\theta_0)$, $C_n = C_n(\theta_0)$, and $W = W(\theta_0)$.
\begin{itemize}
	\item[(E.2)'] The mapping $\theta\mapsto F_n(\theta)$ is $C^1$.
	There exists a sequence $r_n$ of real numbers, and for each $\theta\in\Theta$ there exist sequences of regular matrices $B_n$, $C_n$, and a regular matrix $W$, such that
	\begin{align*}
		\sup_{\theta\in \mathbf{B}_{r_n}(\theta_0)} \|B_n DF_n(\theta) C_n - W\| \pconv 0, \\
		\frac{\|C_n\| \|B_n A_n^{-1}\|}{r_n} \pconv 0..
	\end{align*}
\end{itemize}
We only need (E.2)' for our theory, but conditions (E.2) and (E.3) might be easier to verify.

\begin{proposition}\label{prop:condition-estimating}
	Conditions (E.2) and (E.3) imply condition (E.2)'.
\end{proposition}

\noindent
\cite{jacod2017review} only consider the special case $A_n = B_n = C_n^{-1}$, see Condition 2.10 therein.
In contrast, the asymptotic theory presented here allows for additional flexibility.
The proof of Section \ref{sec:GMM} uses the case $A_n\neq C_n^{-1}$, and $B_n = I$ as identity matrix, whereas the proof of Section \ref{sec:threshold} also requires a nontrivial $B_n\neq I$.

The following theorem is the main result of this section. 
Its proof uses central ideas of Lemma 6.2 in \cite{jacod2017review}.

\begin{theorem}\label{thm:ee-consistency}
	Let conditions (E.1)' and (E.2)' hold. 
	Then there exists a sequence of random vectors $\widehat{\theta}_n\in\Theta$ such that $\mathbb{P}(F_n(\widehat{\theta}_n)=0)\to 1$ and 
	\begin{align*}
		A_n B_n^{-1} W C_n^{-1} [\widehat{\theta}_n-\theta_0] = \mathcal{O}_{\mathbb{P}}(1).
	\end{align*}
	The sequence is locally unique in the sense that for any other sequence of random variables $\widetilde{\theta}_n$ such that $\mathbb{P}(F_n(\widetilde{\theta}_n)=0)\to 1$ and $\mathbb{P}(\|\widetilde{\theta}_n-\theta_0\|\leq r_n)\to 1$, we have $\mathbb{P}(\widehat{\theta}_n=\widetilde{\theta}_n) \to 1$.
	If additionally $(E.1)$ holds, then, as $n\to\infty$,
	\begin{align*}
		A_n B_n^{-1} W C_n^{-1} (\widehat{\theta}_n-\theta_0) \quad\wconv\quad -Z.
	\end{align*}
\end{theorem}

\noindent
Note that we obtain asymptotic uniqueness among all estimators which converge at least with rate $r_n$. 
In finite samples, it might still occur that the solution of $F_n(\theta)=0$ is not unique, and one would need to pick one of those solutions as an estimator. 
It might also even happen that for large $n$, the estimating equations has two solutions, only one of which yields a consistent estimator.
To ensure that these inconvenient scenarios do not occur, one would need additional global properties of the function $F_n$. 
However, uniqueness of the solution of nonlinear systems of equations is a non-trivial mathematical issue in general.
Hence, results about estimating equations typically only yield the existence of a suitable sequence of solutions, as formulated in Theorem \ref{thm:ee-consistency} above.
Global uniqueness may then be derived on a case-by-case basis.
If this is not possible, the solution of $F_n(\theta)=0$ needs to be determined numerically in practice, and one may use the numerical solution as an estimator.

\begin{proof}[Proof of Proposition \ref{prop:condition-estimating}]
	Set $B_n=B_n(\theta_0)$, $C_n=C_n(\theta_0)$, $W=W(\theta_0)$.
	We have
	\begin{align*}
		\|B_n DF_n(\theta) C_n - W\| 
		&\leq \|B_n(\theta) DF_n(\theta) C_n(\theta) - W(\theta)\| + \|W(\theta)-W(\theta_0)\| \\
		&\quad + \|B_n(\theta) DF_n(\theta) C_n(\theta)-B_n DF_n(\theta) C_n(\theta) \| \\
		&\quad + \|B_n DF_n(\theta) C_n(\theta)-B_n DF_n(\theta) C_n \|.
	\end{align*}
	The first two terms vanish uniformly for $\theta\in \mathbf{B}_{r_n}(\theta_0)$ by (E.2) and (E.3).
	Regarding the last two terms, we observe that
	\begin{align*}
		\|B_n DF_n(\theta) C_n(\theta)-B_n DF_n(\theta) C_n \| 
		&\leq \|B_n DF_n(\theta) C_n\| \|C_n(\theta)C_n^{-1} -I\| \pconv 0.
	\end{align*}
	The same holds for the remaining term.
\end{proof}

\begin{proof}[Proof of Theorem \ref{thm:ee-consistency}]
	\underline{Consistency with rate $r_n$:}
	The equation $F_n(\widehat{\theta}_n)=0$ holds if and only if $\widehat{\theta}_n$ is a fixed point of the function $\phi(\theta)=\theta+ C_n W^{-1}B_n F_n(\theta)$.
	We show that $\phi$ is a contraction for $n$ sufficiently large.
	We use the fact that for two matrices $A,B\in\R^{d\times d}$, it holds that $\|AB\|_F = \|BA\|_F$ for the Frobenius norm.
	In particular, $\|A\|_F = \|C_n^{-1} A C_n\|_F$.
	Thus,
	\begin{align*}
		\sup_{\theta \in \mathbf{B}_{r_n}(\theta_0)} \|D\phi(\theta)\|_F
		&= \sup_{\theta \in \mathbf{B}_{r_n}(\theta_0)}  \|I-C_n W^{-1} B_n DF_n(\theta)  \|_F \\
		&= \sup_{\theta \in \mathbf{B}_{r_n}(\theta_0)}  \|I-W^{-1} B_n DF_n(\theta)C_n  \|_F \\
		&\leq \sup_{\theta\in \mathbf{B}_{r_n}(\theta_0)} \|W^{-1}\|_{op} \|B_n DF_n(\theta) C_n - W\|_F 
		\quad \pconv 0.
	\end{align*}
	In the last step, $\|\cdot\|_{op}$ denotes the Euclidean operator norm.
	Denote by $\Omega_n$ the event that 
	\begin{align*}
		\Omega_n = \left\{ \sup_{\theta \in \mathbf{B}_{r_n}(\theta_0)} \|D\phi(\theta)\|_F \leq \frac{1}{2},\quad \|C_n W^{-1} B_n F_n(\theta_0)\|\leq \frac{r_n}{3} \right\}.
	\end{align*}
	Because $\|C_n W^{-1}B_n F_n(\theta_0)\| = \mathcal{O}_P (\|C_n W^{-1}B_n A_n^{-1}\|_F) \ll r_n$, we have $\mathbb{P}(\Omega_n)\to 1$.
	Now define the sequence $\theta_k$ recursively by $\theta_{k}=\phi(\theta_{k-1})$, and $\theta_0$ as above.
	On the event $\Omega_n$, $\phi$ is a contraction on $\mathbf{B}_{r_n}(\theta_0)$, and we need to show that the sequence $\theta_k$ satisfies $\theta_k\in \mathbf{B}_{r_n}(\theta_0)$.
	We may show by induction that
	\begin{align*}
		\|\theta_k-\theta_0\| \leq \sum_{r=1}^k \|\theta_r-\theta_{r-1}\| 
		&\leq \sum_{r=1}^k 2^{-r} \|\theta_1 - \theta_0\| \\
		&\leq  \|\theta_1-\theta_0\| \\
		&=  \|C_n W^{-1}B_n F_n(\theta_0)\| \quad \leq \frac{r_n}{3}.
	\end{align*}
	In particular, $\theta_k$ is a Cauchy sequence and converges to a limit value $\theta_\infty\in \mathbf{B}_{r_n}(\theta_0)$, which satisfies such that $\phi(\theta_\infty)=\theta_\infty$, i.e.\ $F_n(\theta_\infty)=0$.
	Moreover, $\theta_\infty$ is measurable since each $\theta_k$ is a measurable random variable.
	
	Define $\widehat{\theta}_n = \theta_\infty$ on the event $\Omega_n$, and $\widehat{\theta}_n=\underline{\theta}\in\Theta$ otherwise, where $\underline{\theta}$ is some arbitrary but fixed parameter value.
	Then $\mathbb{P}(F_n(\widehat{\theta}_n)=0)\geq \mathbb{P}(\Omega_n)\to 1$, and $\|\widehat{\theta}_n - \theta_0\| = \mathcal{O}_{\mathbb{P}}(r_n)$.	
	On the event $\Omega_n$, $\phi$ is a contraction on $\mathbf{B}_{r_n}(\theta_0)$, such that $\widehat{\theta}_n$ is also the unique solution on this set. 
	This yields uniqueness result claimed in the Theorem.
	
	\underline{Asymptotic distribution:}
	On the event $\Omega_n$, we apply the mean value theorem to obtain
	\begin{align*}
		0 = A_n F_n(\widehat{\theta}_n) 
		&= A_n F_n(\theta_0) + A_n\widetilde{F}_n [\widehat{\theta}_n-\theta_0] \\
		&= A_n F_n(\theta_0) + A_n B_n^{-1} [B_n\widetilde{F}_nC_n] C_n^{-1} [\widehat{\theta}_n-\theta_0],
	\end{align*}
	where $(\widetilde{F}_n)_{l,r} = \partial_{\theta_r} F_n(\widetilde{\theta}^l)$, for $l,r=1,\ldots, d$, and for some $\widetilde{\theta}^l$ on the line segment between $\theta_0$ and $\widehat{\theta}_n$. 
	Hence, $\widetilde{\theta}^l \in \mathbf{B}_{r_n}(\theta_0)$.
	By (E.1), we obtain the weak limit
	\begin{align*}
		A_n B_n^{-1} [B_n\widetilde{F}_nC_n] C_n^{-1} [\widehat{\theta}_n-\theta_0] \quad\wconv\quad -Z.
	\end{align*}
	By (E.2)', $B_n \widetilde{F}_n C_n \pconv W$ as $n\to\infty$.
	This also yields
	\begin{align*}
		&\quad \left[A_n B_n^{-1} W C_n^{-1}\right]^{-1} \left[A_n B_n^{-1} [B_n\widetilde{F}_nC_n] C_n^{-1}\right] \\
		&= C_n \left[ W^{-1} (B_n \widetilde{F}_n C_n)  \right] C_n^{-1} \quad \pconv I_{d\times d}.
	\end{align*}
	Here, we use that $\|C_n M C_n^{-1}-I\|_F = \|M-I\|_F$.
	In particular, Slutsky's theorem yields
	\begin{align*}
		A_n B_n^{-1} W C_n^{-1} [\widehat{\theta}_n-\theta_0] \quad\wconv\quad -Z.
	\end{align*}	
	If not (E.1), but only (E.1)' holds, then 
	\begin{align*}
		A_n B_n^{-1} [B_n\widetilde{F}_nC_n] C_n^{-1} [\widehat{\theta}_n-\theta_0] =\mathcal{O}_{\mathbb{P}}(1),
	\end{align*}
	and we may proceed analogously as for the central limit theorem.
\end{proof}

\section{Proofs of the main results} \label{sec5}

Here we present the proofs of the main results. 
Section \ref{sec5.1} is devoted to some classical statements about L\'evy processes and moment estimates for mixed stable distributions. In Section \ref{sec5.2} we derive bounds for covariance kernels, which will be key for proving central limit theorems. Section \ref{sec5.3} presents the proof of Theorem \ref{thm:clt}, which is based upon an approximation of the original statistic by a sequence of $m$-dependent random variables. In Section \ref{sec5.5} we prove the theoretical statements for the adaptive estimator, while Section \ref{sec5.6} treats the smooth threshold case.
The singularity of measures induced by the mixed lfsm is addressed in Section \ref{sec:identifiability}.

\subsection{Preliminaries about stable laws} \label{sec5.1}

The Lévy measure $\nu(dz)$ of a $d$-dimensional $\beta$-stable law is characterized by a finite measure $\Lambda$ on the unit sphere $\mathbb{S}^{d-1}$, such that
\begin{align}
	\nu(A) &= \int_{0}^\infty \int_{\mathbb{S}^{d-1}} \mathds{1}(\rho \cdot r \in A)\, d\Lambda(\rho) \, r^{-1-\beta}\, dr. \label{eqn:def-levymeasure}
\end{align}
A multivariate $\beta$-stable distribution is an infinitely divisible distribution with characteristic triplet $(0,0,\nu)$.
That is, the characteristic function of a stable random vector $Z$ is given by 
\begin{align*}
	\E(\exp(i\lambda^T Z)) = \exp\left(-\int_{\R^d} [e^{i\lambda^T z} - 1 - i\lambda^T z\, \mathds{1}_{\{\|z\|\leq 1\}}]\nu(dz)\right), \qquad \lambda\in\R^d,
\end{align*}
for $\nu$ as in \eqref{eqn:def-levymeasure}.
If $\nu$ is symmetric, the truncation may be omitted.
For any $\lambda\in\R^d$, the projection $\lambda^T Z$ is a univariate $\beta$-stable random variable, see \cite[Thm.\ 2.1.5]{samorodnitsky1994stable}.
In the sequel, we will frequently use the polynomial tail bound $\mathbb{P}(|Z|>x)\leq C x^{-\beta}$ for a univariate $\beta$-stable random variable, see e.g.\ \cite[1.2.15]{samorodnitsky1994stable} or \cite[(1.7)]{mies2020}. 
Applying this inequality componentwise to a $d$-variate $\beta$-stable random vector $Z$, we find that
\begin{align}
	\mathbb{P}(\|Z\|>x) \leq C x^{-\beta}. \label{eqn:tailbound}
\end{align}
For any function $f:\R^d\to\R$, we denote
\begin{align*}
	f^{[\nu]}(x) &= \int_{\R^d} \left[ f(x+z)-f(x)-z^T\nabla f(x)\mathds{1}_{\{|z|\leq 1\}}\right] \, \nu(dz).
\end{align*}
Note that for any $f\in C^2(\R^d)$,
\begin{align}
	\|f^{[\nu]}\|_\infty 
	&\leq C \|f\|_\infty \nu(\{|z|>1\}) + C\|D^2f\|_\infty  \int_{|z|\leq 1} |z|^2 \, \nu(dz) \nonumber  \\
	&\leq C \|f\|_\infty \nu(\{|z|>1\}) + C\|D^2f\|_\infty \Lambda(\mathbb{S}^{d-1})  \int_{-1}^1  |r|^{1-\beta}  \, dr \nonumber \\
	&\leq C_\nu \left[ \|f\|_\infty + \|D^2f\|_\infty \right], \label{eqn:jump-c2}
\end{align}
Moreover, we highlight that $\frac{d}{dx_i}f^{[\nu]}(x) = (\frac{d}{dx_i}f)^{[\nu]}$, such that we may also use \eqref{eqn:jump-c2} to obtain uniform bounds for higher order derivatives of $f^{[\nu]}$.
In particular,
\begin{align*}
	\|(f^{[\nu]})^{(j)}\|_\infty \leq C_\nu \left[ \|D^jf\|_\infty + \|D^{j+2}f\|_\infty \right].
\end{align*}
We may also derive an alternative representation of $f^{[\nu]}$. 
To this end, let $(Z_h)_{h\in [0,1]}$ be a Lévy process with $Z_1 \deq Z$. 
Then $Z_h \deq h^{\frac{1}{\beta}} Z$, and Itô's formula yields
\begin{align}
    f^{[\nu]}(x)
    &= \left[\frac{d}{dh} f(x+Z_h)\right]_{h=0} 
    = \left[\frac{d}{dh} f(x+h^{\frac{1}{\beta}}Z)\right]_{h=0}. \label{eqn:fnu-deriv}
\end{align}

Our results make use of the class of functions $\mathfrak{F}_\eta$ introduced in Section \ref{sec:threshold}.
We note that condition (F1) implies for any $f\in\mathfrak{F}_\eta$ that $\|f^{[\nu]}\|_\infty,\|Df^{[\nu]}\|_\infty,\|D^2 f^{[\nu]}\|_\infty<\infty$ for a $\beta$-stable Lévy measure $\nu$.

\begin{lemma}\label{lem:truncexpect}
	Let $Z$ be a $\beta$-stable random vector.
	For any $0\leq q<p$ such that $q<\beta$, there exists a constant $C_{p,q,Z}$ such that for all $a>0$,
	\begin{align*}
		\E \left[\|a Z\|^q \wedge \|a Z\|^p \right] \leq C_{p,q,Z} (a^p+a^\beta).
	\end{align*}
\end{lemma}
\begin{proof}[Proof of Lemma \ref{lem:truncexpect}]
	For the $\beta$-stable random vector $Z$, we use the polynomial tail bound \eqref{eqn:tailbound}. 
	Then, for $q>0$.
	\begin{align*}
		\E \left[\|a Z\|^q \wedge \|a Z\|^p \right]
		&= \int_1^\infty \mathbb{P}(\|aZ\|^q>\epsilon)\, d\epsilon + \int_0^1 \mathbb{P}\left( \|a Z\|^p > \epsilon \right)\, d\epsilon \\
		&= \int_1^\infty \mathbb{P}\left( \|Z\| > \epsilon^\frac{1}{q}/a \right)\, d\epsilon  + \int_{0}^1 \mathbb{P}\left( \|Z\| > \epsilon^\frac{1}{p}/a \right)\, d\epsilon \\ 
		&\leq C_Z a^\beta \int_1^\infty \epsilon^{-\frac{\beta}{q}}\, d\epsilon + a^p + C_Z a^\beta\int_{a^p}^1 \epsilon^{-\frac{\beta}{p}}\, d\epsilon \\
		&\leq C_{p,q,Z} (a^\beta + a^p).
	\end{align*}
	For $q=0$, we obtain the same bound since $\mathbb{P}(\|aZ\|^0 > \epsilon)=0$ for $\epsilon<1$.
\end{proof}

\noindent
We will also use the following technical result.

\begin{lemma}\label{lem:bounded-taylor}
	Let $f\in C^2(\R^d;\R)$ such that $f$, $D f$, and $D^2f$ are bounded.
	Then
	\begin{align*}
		|f(y+x)-f(y) - D f(y)^T x| & \leq  \left(\|x\|^2\sup_{r\in[0,1]}\|D^2f(y+rx)\|_\infty\right) \\
		&\qquad \wedge \left(2\|f\|_\infty + \|x\| \|D f\|_\infty\right) \\
		&\leq (1\wedge\|x\|^2 ) (\|f\|_\infty + \|Df\|_\infty + \|D^2f\|_\infty).
	\end{align*}
\end{lemma}
\begin{proof}[Proof of Lemma \ref{lem:bounded-taylor}]
	A Taylor expansion yields that $|f(y+x)-f(y) - Df(y)x| \leq \sup_{r\in[0,1]}\|D^2f(y+rx)\|_\infty \|x\|^2$.
	On the other hand, $|f(y+x)-f(y) - Df(y)x|  \leq 2 \|f\|_\infty + \|D f\|_\infty \|x\|$.
\end{proof}

\noindent
The following Lemma presents an essential technical result, which will turn out to be essential to compute autocovariances of mixed moving average processes.

\begin{lemma}\label{lem:mom}
	Let $B^i \sim \mathcal{N}(0,\Sigma^i)$ be independent Gaussian random vectors for $\Sigma^i\in\R^{d\times d}$, $i=1,\ldots, q_1$, and let $Z^j$ be independent $d$-dimensional $\beta_j$-stable random vector with Lévy measure $\nu_j$, for $j=1,\ldots, q_2$.
	Assume that the $Z^j$ are symmetric, i.e.\ $Z^j\deq -Z^j$.
	Let the $B^i$ and $Z^j$ be independent.
	Consider non-negative real sequences $a_{n,i},b_{n,j}\to 0$.
	
	\begin{enumerate}[(i)]
		\item If $f\in\mathfrak{F}_\eta$ for some $\eta>0$, then for any $\lambda\in(0,1)$,
		\begin{align*}
			\E f\left( \sum_{i=1}^{q_1} a_{n,i} B^i + \sum_{j=1}^{q_2} b_{n,j} Z^j \right) 
			&= \frac{1}{2} \sum_{r,r'=1}^d D^2_{rr'}f(0) \left[\sum_{i=1}^{q_1} a_{n,i}^2 \Sigma^i_{rr'}\right]   + \sum_{j=1}^{q_2} b_{n,j}^{\beta_j} f^{[\nu_j]}(0)  \\
			& \quad + \mathcal{O}\left( \max_j b_{n,j}^{\beta_j} \max_i a_{n,i}^2 \right) + \mathcal{O}\left( \max_{j} b_{n,j}^{2\beta_j} \right)  \\
			&\quad + \mathcal{O}\left( \exp\left( -\frac{ \eta^2 (1-\lambda)^2 }{2 d\, \sum_{i=1}^{q_1} a_{n,i}^2 \, \max_l |\Sigma^i_{ll}|} \right) \right).
		\end{align*}
		\item If $f\in\mathfrak{F}_0$ and $D^2f(0)=0$, then
		\begin{align*}
			\E f\left( \sum_{i=1}^{q_1} a_{n,i} B^i + \sum_{j=1}^{q_2} b_{n,j} Z^j \right) 
			&=  \frac{1}{8}\sum_{i,i'=1}^{q_1}a_{n,i}^2 a_{n,i'}^2\sum_{l,l',r,r'=1}^d D^4_{l,l',r,r'}f(0) \Sigma^i_{ll'}\Sigma^{i'}_{rr'} \\
			&\quad + \sum_{j=1}^{q_2} b_{n,j}^{\beta_j} f^{[\nu_j]}(0) + o\left(\max_i a_{n,i}^4 \right) \\
			&\quad + \mathcal{O}\left( \max_j b_{n,j}^{\beta_j} \max_i a_{n,i}^2 \right) + \mathcal{O}\left( \max_{j} b_{n,j}^{2\beta_j} \right) .
		\end{align*}
		\item If $f\in\mathfrak{F}_0$, then
		\begin{align*}
			\E f\left( \sum_{i=1}^{q_1} a_{n,i} B^i + \sum_{j=1}^{q_2} b_{n,j} Z^j \right)
			&= \frac{1}{2} \sum_{r,r'=1}^d D^2_{rr'}f(0) \left[\sum_{i=1}^{q_1} a_{n,i}^2 \Sigma^i_{rr'}\right]   + \sum_{j=1}^{q_2} b_{n,j}^{\beta_j} f^{[\nu_j]}(0) \\
			&\quad + \mathcal{O}\left(\max_j b_{n,j}^{2\beta_j} \right) +  \mathcal{O}\left(\max_i a_{n,i}^4 \right).
		\end{align*}
	\end{enumerate}
	All $\mathcal{O}(\ldots)$ terms are bounded uniformly in $f\in\mathfrak{F}_\eta$ resp.\ $f\in \mathfrak{F}_0$, and all $o(\ldots)$ terms vanish uniformly in $f\in\mathfrak{F}_\eta$ resp.\ $f\in \mathfrak{F}_0$.
\end{lemma}
\begin{proof}[Proof of Lemma \ref{lem:mom}]
	For a function $f\in C^2(\R^d,\R)$, we introduce the differential operator
	\begin{align*}
		\mathcal{D} f = \sum_{i=1}^{q_1}\frac{a_{n,i}^2}{2}\sum_{l,l'=1}^d \Sigma^i_{ll'} D^2_{ll'}f
	\end{align*}
	Let $(B^i_t)_{t\in [0,1]}$ be independent Brownian motions  with covariance $\Sigma^i$ for $i=1,\ldots, q_1$, and let $Z_t^j$ be a symmetric $\beta_j$-stable Lévy process with Lévy measure $\nu_j$, such that $B^i_1=B^i$ and $Z^j_1=Z^j$.
	We denote
	\begin{align*}
		G_{t,n}&= \sum_{i=1}^{q_1} a_{n,i} B^i_t ,\\
		J_{t,n}&= \sum_{j=1}^{q_2} b_{n,j} Z^j_t 
	\end{align*}
	Then Itô's formula and the scaling of $\nu^j(dz)$ yields
	\begin{align}
		\E f\left(G_{1,n}+J_{1,n} \right) = \int_0^1 \E \left\{\left( \mathcal{D}f + \sum_{j=1}^{q_2} b_{n,j}^{\beta_j} f^{[\nu_j]} \right) (G_{t,n}+J_{t,n} )\right\} \, dt. \label{eqn:mom-1}
	\end{align}
	To obtain the scaling $b_{n,j}^{\beta_j}$ of the jump-measure integral, we use that
	\begin{align*}
		&\quad \int \left[f(x+b_{n,j} z) - f(x) - Df(x)b_{n,j} z\mathds{1}_{\{|z|\leq 1\}}\right]\, \nu^j(dz) \\
		&= \int_0^\infty \int_{S^{d-1}} \frac{f(x+b_{n,j} \rho r) - f(x) - Df(x)b_{n,j} \rho r\mathds{1}_{\{r\leq 1\}}}{r^{1+\beta_j}}\, \Lambda^j(d\rho)\,dr \\
		&= b_{n,j}^{\beta_j} \int_0^\infty \int_{S^{d-1}} \frac{f(x+ \rho r) - f(x) - Df(x)\rho r\mathds{1}_{\{r/b_{n,j}\leq 1\}}}{r^{1+\beta_j}}\,\Lambda^j(d\rho) \, dr \\
		&=  b_{n,j}^{\beta_j} \int_0^\infty \int_{S^{d-1}} \frac{f(x+ \rho r) - f(x) - Df(x)\rho r\mathds{1}_{\{r\leq 1\}}}{r^{1+\beta_j}}\,\Lambda^j(d\rho)\, dr,
	\end{align*}
	using the symmetry in the last step.
	
	\noindent
	\underline{Controlling the stable part:}
	Since $f^{[\nu_j]}$ and its first two derivatives are bounded, we find that
	\begin{align*}
		&\quad \int_0^1 \left|\E f^{[\nu_j]}\left(G_{t,n} + J_{t,n}\right)-f^{[\nu_j]}\left( J_{t,n} \right)\right|\, dt \\
		&\leq \int_0^1 \left| \underbrace{\E\left[ D f^{[\nu_j]}(J_{t,n}) G_{t,n} \right]}_{=0} \right|\, dt + C \|D^2f^{[\nu_j]}\|_\infty \E \|G_{1,n}\|^2 
		\quad = \mathcal{O}\left(\sum_{i=1}^{q_1} a_{n,i}^2\right).
	\end{align*}
	Now let $q_2'$ be such that $\beta_i\leq 1$ for $1\leq i \leq q_2'$, and $\beta_i>1$ for $q_2' < i \leq q_2$.
	Then Lemma \ref{lem:truncexpect} yields,
	\begin{align*}
		\int_0^1 \left|\E f^{[\nu_j]}\left(\sum_{i=1}^{q_2} b_{n,i} Z_t^i\right) - f^{[\nu_j]}\left(\sum_{i=q_2'+1}^{q_2} b_{n,i} Z_t^i\right) \right|\, dt
		&\leq C \|Df^{[\nu_j]}\|_\infty  \sum_{i=1}^{q_2'}\int_0^1\E\left[ 1\wedge\left\| b_{n,i} Z_t^i\right \| \right] \\
		&\leq C \sum_{i=1}^{q_2'} \int_0^1 (t^{\frac{1}{\beta_i}}b_{n,i} + t b_{n,i}^{\beta_i}) \\
		&\leq C \sum_{i=1}^{q_2'} b_{n,i}^{\beta_i}.
	\end{align*}	
	In the last step, we used $b_{n,i}=\mathcal{O}(1)$, and $Z_{t}^i \deq t^{1/\beta_i} Z_1^i$.
	For the indices $i>q_2'$, i.e.\ $\beta_i> 1$, we may tighten this bound by using $\E\|Z^i_t\|<\infty$ and $\E Z^i_t =0$.
	Via Lemma \ref{lem:bounded-taylor}, since the first two derivatives of $f^{[\nu_j]}$ are bounded, we obtain
	\begin{align*}
		&\quad \int_0^1 \left|\E f^{[\nu_j]}\left(\sum_{i=q_2'+1}^{q_2} b_{n,i} Z_t^i\right) - f^{[\nu_j]}(0) \right|\, dt \\
		&\leq \sum_{i=q_2'+1}^{q_2} \int_0^1 \left|\underbrace{\E\left[Df^{[\nu_j]}(0)b_{n,i} Z^i_t  \right]}_{=0}\right| + C\,\E \left[ 1\wedge \|b_{n,i} Z_t^i\|^2 \right]\, dt\\
		&\leq C \sum_{i=q_2'+1}^{q_2} b_{n,i}^{\beta_i} ,
	\end{align*}
	using Lemma \ref{lem:truncexpect} and $b_{n,i}^2 = \mathcal{O}(b_{n,i}^{\beta_i})$ in the last step.
	Thus, returning to \eqref{eqn:mom-1}, we have shown that
	\begin{align}
		\int_0^1 \E f^{[\nu_j]}(G_{t,n} + J_{t,n})\, dt 
		= f^{[\nu_j]}(0) + \mathcal{O}\left( \sum_{i=1}^{q_1}a_{n,i}^2 + \sum_{i=1}^{q_2} b_{n,i}^{\beta_i} \right), \label{eqn:mom-2}
	\end{align}
	which yields
	\begin{align*}
		\E f(G_{1,n}+J_{1,n}) 
		&=   \int_0^1 \E Df( G_{t,n} + J_{t,n} ) \, dt + \sum_{j=1}^{q_2}b_{n,j}^{\beta_j} f^{[\nu_j]}(0) \\
		&\quad +  \mathcal{O} \left( \left(\sum_{j=1}^{q_2} b_{n,j}^{\beta_j}\right)\left( \sum_{i=1}^{q_1}a_{n,i}^2+\sum_{j=1}^{q_2} b_{n,j}^{\beta_j} \right) \right).
	\end{align*}
	
	\noindent
	\underline{Controlling the Gaussian part, claim (i):}
	Since $D^2f(x)=D^2f(0)$ for $\|x\|\leq \eta$, we have
	\begin{align*}
		&\quad \left|  \int_0^1 \E\left[ Df(G_{t,n} + J_{t,n})-Df(0)\right]\, dt \right| \\
		&\leq  C \sum_{i=1}^{q_1} a_{n,i}^2 \int_0^1 \mathbb{P}(\|G_{t,n} + J_{t,n}\| > \eta)\, dt \\
		&\leq  C \sum_{i=1}^{q_1} a_{n,i}^2 \left[\mathbb{P}\left(\|G_{1,n}\| > \eta\lambda\right) + \sum_{j=1}^{q_2} \mathbb{P}\left( \|Z_1^j\| > \frac{\eta(1-\lambda)}{b_{n,j} q_2} \right)\right] \\
		&\leq  C \sum_{i=1}^{q_1} a_{n,i}^2 \left[ \mathbb{P}\left(\|G_{1,n}\| > \eta\lambda\right) + \sum_{j=1}^{q_2} b_{n,j}^{\beta_j} \right], 
	\end{align*}
	using the polynomial tail bound \eqref{eqn:tailbound}.
	For the last term, we may exploit the fact that $G_{1,n}\sim\mathcal{N}(0,\overline{\Sigma})$ for $\overline{\Sigma} = \sum_{i=1}^{q_1} a_{n,i}^2 \Sigma^i$.
	Using the tail bound $\mathbb{P}(|N|>x) \leq 2 \exp(-x^2/2)$ for a standard Gaussian random variable $N$, we find that
	\begin{align*}
		\mathbb{P}(\|G_{1,n}\| >x) \leq \sum_{l=1}^d \mathbb{P}(|G_{1,n}^l|^2>x^2/d) \leq 2 d \exp\left(-\frac{x^2}{2d\max_l \overline{\Sigma}_{ll}}\right).
	\end{align*}
	Hence,
	\begin{align*}
		\mathbb{P}\left(\|G_{1,n}\| > \eta\lambda\right)
		&=\mathcal{O}\left( \exp\left( -\frac{ \eta^2 \lambda^2 }{2 d\, \sum_{i=1}^{q_1} a_{n,i}^2 \, \max_l |\Sigma^i_{ll}|} \right)   \right).
	\end{align*}
	This completes the proof of the first claim.
	
	\noindent
	\underline{Claim (ii):} For the second claim, we again use \eqref{eqn:mom-1} and obtain
	\begin{align}
		\E \left[ \mathcal{D}f(G_{t,n} + J_{t,n}) \right]
		&= \mathcal{D}f(0) + \int_0^t \E \left\{ \left( \mathcal{D}^2 f + \sum_{j=1}^{q_2} b_{n,j}^{\beta_j}  \mathcal{D}f^{[\nu_j]}\right) (G_{s,n} + J_{s,n})\, ds  \right\} \label{eqn:mom-3}
	\end{align}
	It holds that $\|\mathcal{D}f^{[\nu_j]}\|_\infty \leq C$, since $f\in C^4$ with bounded derivatives up to order four.
	Since $\|D^4 f\|_\infty \leq 1$ and $\|D^5 f\|_\infty\leq 1$, we further have 
	\begin{align*}
		\left| \E \left[ D^4f(G_{s,n} + J_{s,n})\right] - D^4f(0) \right| \leq \E \left[ 1 \wedge \| G_{s,n} + J_{s,n} \| \right] = o(1),
	\end{align*}
	uniformly in $s\in[0,1]$, since $a_{n,i}\to 0$, $b_{n,j}\to 0$ as $n\to\infty$. 
	That is, 
	\begin{align*}
		\E \left[ \mathcal{D}^2f(G_{s,n} + J_{s,n})\right] = \mathcal{D}^2f(0) + o\left( \sum_{i=1}^{q_1} a_{n,i}^2 \right)^2.
	\end{align*}
	Plugging this into \eqref{eqn:mom-3}, we find that, uniformly in $t\in[0,1]$,
	\begin{align}
		\E \left[ \mathcal{D}f(G_{t,n}+J_{t,n}) \right]
		&= \mathcal{D}f(0) + t \mathcal{D}^2 f(0) + o\left( \sum_{i=1}^{q_1} a_{n,i}^2 \right)^2 + \mathcal{O}\left(\sum_{j=1}^{q_2}b_{n,j}^{\beta_j} \right). \label{eqn:expect-1}
	\end{align}
	By assumption, in case (ii), we have $Df(0)=0$.
	Plugging this into \eqref{eqn:mom-1}, and using \eqref{eqn:mom-2}, we obtain
	\begin{align*}
		\E f(G_{1,n} + J_{1,n} )
		&=\int_0^1 \E \left\{\left( \mathcal{D}f + \sum_{j=1}^{q_2} b_{n,j}^{\beta_j} f^{[\nu_j]} \right) (G_{t,n}+J_{t,n} )\right\} \, dt \\
		&= \mathcal{D}^2f(0) \int_0^1  t\, dt + o\left( \sum_{i=1}^{q_1} a_{n,i}^2 \right)^2 + \mathcal{O}\left( \sum_{i=1}^{q_1} a_{n,i}^2 \sum_{j=1}^{q_2} b_{n,j}^{\beta_j} \right) \\
		&\quad + \sum_{j=1}^{q_2} b_{n,j}^{\beta_j} f^{[\nu_j]}(0) + \sum_{j=1}^{q_2} b_{n,j}^{\beta_j}\mathcal{O}\left( \sum_{i=1}^{q_1}a_{n,i}^2 + \sum_{i=1}^{q_2} b_{n,i}^{\beta_i} \right) \\
		&= \frac{1}{2}\mathcal{D}^2f(0) + \sum_{j=1}^{q_2} b_{n,j}^{\beta_j} f^{[\nu_j]}(0) + o\left( \sum_{i=1}^{q_1} a_{n,i}^2 \right)^2\\
		&\quad + \mathcal{O}\left( \max_j b_{n,j}^{\beta_j} \max_i a_{n,i}^2 \right) + \mathcal{O}\left( \max_{j} b_{n,j}^{2\beta_j} \right) . 
	\end{align*}
	To complete the proof of (ii), observe that
	\begin{align*}
		\mathcal{D}^2f(0) 
		&= \frac{1}{4}\sum_{i,i'=1}^{q_1}a_{n,i}^2 a_{n,i'}^2\sum_{l,l',r,r'=1}^d D^4_{l,l',r,r'}f(0) \Sigma^i_{ll'}\Sigma^{i'}_{rr'}.
	\end{align*}
	
	\noindent
	\underline{Claim (iii):} Regarding the third claim, note that the leading term is identical to case (i), but the bound on the remainder is less tight, since we impose fewer assumptions on $f$.
	In particular, \eqref{eqn:mom-3} yields that
	\begin{align*}
		\E \left[ \mathcal{D}f(G_{t,n}+J_{t,n}) \right]
		&= \mathcal{D}f(0) + \mathcal{O}\left( \max_i a_{n,i}^4\right) + \mathcal{O}\left( \max_j b_{n,j}^{\beta_j} \max_i a_{n,i}^2 \right).
	\end{align*}
	In combination with \eqref{eqn:mom-1} and \eqref{eqn:mom-2}, we obtain
	\begin{align*}
		\E f(G_{1,n} + J_{1,n} )
		&= \mathcal{D}f(0)  + \sum_{j=1}^{q_2} b_{n,j}^{\beta_j} f^{[\nu_j]}(0) \\
		&\qquad + \mathcal{O}\left( \max_i a_{n,i}^4\right) + \mathcal{O}\left( \max_j b_{n,j}^{\beta_j} \max_i a_{n,i}^2 \right) + \mathcal{O}\left( \max_j b_{n,j}^{2\beta_j} \right),
	\end{align*}
	which yields the desired bound for the remainder terms.
\end{proof}

\subsection{Covariance bounds for multiscale moving average processes} \label{sec5.2}
In our subsequent proofs, we will use the following $m$-dependent approximation
\begin{align*}
	X_{t,n,m} 
	&= \sum_{i=1}^{q_1} a_{n,i} \left[\int_{t-m}^t h_i(s-t)\,dB^i_s + \int_{-\infty}^{t-m} h_i(s-t)\,dB^{i,t}_s \right] \\
	&\quad + \sum_{j=1}^{q_2} b_{n,j} \left[ \int_{t-m}^t g_j(s-t)\, dZ^j_s + \int_{-\infty}^{t-m} g_j(s-t)\, dZ^{j,t}_s \right],
\end{align*}
where $B^{i,t}$ is an independent copy of $B^i$ for each $t\in\N$, and $Z^{j,t}$ is an independent copy of $Z^j$ for each $t\in\N$.
Note that $X_{t,n,m}\deq X_{t,n}$ for each fixed $t$, but the dependence structure may differ.
We will also write $X_{t,n,\infty}=X_{t,n}$.

\begin{lemma}\label{lem:autocov-decay}
	Let $f_1\in \mathfrak{F}^0_{\eta_1}, f_2\in\mathfrak{F}^0_{\eta_2}$ for $\eta_1,\eta_2\geq 0$.
	Suppose furthermore that $1+\delta_j\beta_j<0$, $1+2\delta_0<0$.
	
	\begin{enumerate}[(i)]
		\item For any $t>t'\in \N$, and any $m, m'\in \N\cup\{\infty\}$, and all $f_1\in \mathfrak{F}_{\eta_1}^0$, $f_2\in\mathfrak{F}_{\eta_2}^0$,
		\begin{align*}
			\left|\Cov\left[f_1(X_{t,n,m}), f_2(X_{t',n,m'})\right]\right| 
			&\leq C  |t-t'|^{1+2\delta_0}   \left[\sum_{i=1}^{q_1} a_{n,i}^4 + \sum_{i=1}^{q_1} a_{n,i}^2 \sum_{j=1}^{q_2} b_{n,j}^{\beta_j}\right] \\
			&\quad + C \sum_{j=1}^{q_2} |t-t'|^{1+\delta_j \beta_j} b_{n,j}^{\beta_j}.
		\end{align*}
		\item If $\eta_1>0$ or $\eta_2>0$, then for any $\lambda\in(0,1)$, $t> t'\in \N$, and any $m, m'\in \N\cup\{\infty\}$, and all $f_1\in \mathfrak{F}_{\eta_1}^0$, $f_2\in\mathfrak{F}_{\eta_2}^0$,
		\begin{align*}
			&\left|\Cov\left[f_1(X_{t,n,m}), f_2(X_{t',n,m'})\right]\right| \\
			&\leq C  |t-t'|^{1+2\delta_0} \left(  \sum_{i=1}^{q_1} a_{n,i}^2 \right) \left[  \exp\left( -\frac{\eta_1^2\lambda^2}{2d \sigma^2 \sum_{i=1}^{q_1} a_{n,i}^2 }\right) + \sum_{j=1}^{q_2} b_{n,j}^{\beta_j} \right]\\
			&\quad + C\sum_{j=1}^{q_2} |t-t'|^{1+\delta_j\beta_j} b_{n,j}^{\beta_j}.
		\end{align*}
	\end{enumerate}
\end{lemma}
\begin{proof}[Proof of Lemma \ref{lem:autocov-decay}]
	Set $\overline{m}=t-t'$, such that $X_{t,n,\overline{m}}$ and $X_{t',n,m'}$ are independent.
	Then
	\begin{align*}
		\left|\Cov\left[f_1(X_{t,n,m}), f_2(X_{t',n,m'})\right]\right|
		&= \left|\Cov\left[ f_1(X_{t,n,m})-f_1(X_{t,n,\overline{m}}), f_2(X_{t',n,m'})\right] \right|.
	\end{align*}
	Moreover, $\Cov\left[f_1(X_{t,n,m}), f_2(X_{t',n,m'})\right]=0$ for $\overline{m}\geq m$ by construction, hence we only consider the case $\overline{m}<m$.
	Without loss of generality, we assume that the $\beta_j$ are ordered and let $q_2'$ be such that $\beta_j\leq 1$ for $1\leq j \leq q_2'$, and $\beta_j>1$ for $q_2'<j\leq q_2$.
	Denote 
	\begin{align*}
		G_{t,n,m} &= \sum_{i=1}^{q_1} a_{n,i} \left[\int_{t-m}^t h_i(s-t)\,dB^i_s + \int_{-\infty}^{t-m} h_i(s-t)\,dB^{i,t}_s \right], \\
		J_{t,n,m,1} &=  \sum_{j=1}^{q_2'} b_{n,j} \left[ \int_{t-m}^t g_j(s-t)\, dZ^j_s + \int_{-\infty}^{t-m} g_j(s-t)\, dZ^{j,t}_s \right],\\
		J_{t,n,m,2} &=  \sum_{j=q_2'+1}^{q_2} b_{n,j} \left[ \int_{t-m}^t g_j(s-t)\, dZ^j_s + \int_{-\infty}^{t-m} g_j(s-t)\, dZ^{j,t}_s \right], \\
		J_{t,n,m} &= J_{t,n,m,1} + J_{t,n,m,2},
	\end{align*}
	such that
	\begin{align}
		&\quad \left|\Cov\left[f_1(X_{t,n,m}), f_2(X_{t',n,m})\right]\right| \nonumber \\
		&\leq  \left|\Cov\left[ f_1(X_{t,n,m})-f_1(G_{t,n,m} + J_{t,n,\overline{m},1} + J_{t,n,m,2}), f_2(X_{t',n,m'})\right] \right| \label{eqn:cov-1}\\
		&\quad + \left|\Cov\left[ f_1(G_{t,n,m} + J_{t,n,\overline{m},1} + J_{t,n,m,2}) - f_1(X_{t,n,\overline{m}}), f_2(X_{t',n,m'})\right] \right|
		\label{eqn:cov-2}.
	\end{align}
	
	\noindent
	\underline{Ad \eqref{eqn:cov-1}:}
	A Taylor expansion yields
	\begin{align*}
		&\quad \left|\Cov\left[ f_1(G_{t,n,m} + J_{t,n,m,1} + J_{t,n,m,2}) - f_1(G_{t,n,m} + 	J_{t,n,\overline{m},1} + J_{t,n,m,2}), f_2(X_{t',n,m'})\right] \right| \\
		&\leq C\|f_2\|_\infty (\|Df_1\|_\infty + \|f_1\|_\infty) \E \left[ 1\wedge \|J_{t,n,m,1}-J_{t,n,\overline{m},1}\| \right].
	\end{align*}
	Moreover, for each  $l=1,\ldots, d$,
	\begin{align*}
		J^l_{t,n,m,1}-J^l_{t,n,\overline{m},1} 
		&= \sum_{j=1}^{q_2'} b_{n,j}\int_{t-m}^{t-\overline{m}} g^l_j(s)\, dZ_s^j \\
		&\deq \sum_{j=1}^{q_2'} b_{n,j}\left(\int_{t-m}^{t-\widetilde{m}} |g^l_j(s)|^{\beta_j}\,ds\right)^\frac{1}{\beta_j} 	Z_1^j, \\
		\intertext{using the inner clock property of stochastic integrals with respect to $\beta$-stable processes \cite[Thm.\ 3.1]{Rosinski1986}, and}
		\int_{t-m}^{t-\overline{m}} |g^l_j(s)|^{\beta_j}\,ds 
		&\leq C \int_{-\infty}^{-\overline{m}} |s|^{\beta_j \delta_j} \, ds \\
		&\leq C \overline{m}^{1+\beta_j\delta_j}.
	\end{align*}
	Thus, Lemma \ref{lem:truncexpect} yields 
	\begin{align*}
		\E \left[ 1\wedge \|J_{t,n,m,1}-J_{t,n,\overline{m},1}\| \right] 
		&\leq \sum_{l=1}^d \E \left[ 1\wedge \|J^l_{t,n,m,1}-J^l_{t,n,\overline{m},1}\| \right] \\
		&\leq C \sum_{j=1}^{q_2'}\left(b_{n,j} \overline{m}^{\frac{1}{\beta_j}+\delta_j}\right)^{1\wedge\beta_j},
	\end{align*}
	hence
	\begin{align*}
		\left|\Cov\left[ f_1(G_{t,n,\overline{m}} + J_{t,n,m}) - f_1(G_{t,n,\overline{m}} + J_{t,n,\overline{m}}), 	f_2(X_{t',n,m'})\right] \right| 
		\leq C \sum_{j=1}^{q_2'} b_{n,j}^{\beta_j} \overline{m}^{1+\delta_j \beta_j},
	\end{align*}
	because $\beta_j\leq 1$ for $j\leq q_2'$.
	
	\noindent
	\underline{Ad \eqref{eqn:cov-2}:}
	We employ the abbreviation $Y_{t,n,m} = G_{t,n,m} + J_{t,n,m,2}$.
	Lemma \ref{lem:bounded-taylor} yields
	\begin{align}
		&\quad \Cov\left[ f_1(G_{t,n,m} + J_{t,n,\overline{m},1} + J_{t,n,m,2}) - f_1(X_{t,n,\overline{m}}), f_2(X_{t',n,m'})\right] \nonumber \\
		&=\Cov\left[ f_1(Y_{t,n,m} + J_{t,n,\overline{m},1}) - f_1(Y_{t,n,\overline{m}} + J_{t,n,\overline{m},1}), f_2(X_{t',n,m'})\right] \nonumber \\
		&= \E \left[ Df_1(X_{t,n,\overline{m}}) (Y_{t,n,m}-Y_{t,n,\overline{m}}) \left(f_2(X_{t',n,m'}) - \E f_2(X_{t',n,m'})\right) \right] + \Gamma_{t,t',n}, \label{eqn:cov-lin-1}\\[1.5 ex]
		|\Gamma_{t,t',n}|
		&\leq C \, \E \Big[ \Big(\left( 1+ \|Y_{t,n,m}-Y_{t,n,\overline{m}}\| \right)  \nonumber\\
		&\quad \wedge \Big(\sup_{r\in[0,1]} \|D^2f_1(X_{t,n,\overline{m}} + r(Y_{t,n,m}-Y_{t,n,\overline{m}}))\| \;  \|Y_{t,n,m}-Y_{t,n,\overline{m}}\|^2 \Big) \Big) \nonumber  \\
		&\qquad \qquad \cdot \left|f_2(X_{t',n,m'}) - \E f_2(X_{t',n,m'})\right| \Big]. \nonumber
	\end{align}
	Now observe that for any $L\geq 0$, and any $x,y>0$,
	\begin{align*}
		(1+x+y) \wedge L(x+y)^2 
		&\leq 4 (1+(x\vee y)) \wedge L (x\vee y)^2 \\
		& \leq 4 \left[ (1+x)\wedge L x^2 \right] + 4 \left[ (1+y)\wedge L y^2 \right].
	\end{align*}
	Applying this inequality with $x=\|G_{t,n,m}-G_{t,n,\overline{m}}\|$ and $y=\|J_{t,n,m,2}-J_{t,n,\overline{m},2}\|$, such that $\|Y_{t,n,m}-Y_{t,n,\overline{m}}\|\leq x+y$,
	we find that
	\begin{align}
		|\Gamma_{t,t',n}| 
		&\leq C \, \E \left[ \left( 1+ \|J_{t,n,m,2}-J_{t,n,\overline{m},2}\| \right) \wedge  \|J_{t,n,m,2}-J_{t,n,\overline{m},2}\|^2  \right] \nonumber \\
		&\quad + C \, \E \Big[ \Big(\left( 1+ \|G_{t,n,m}-G_{t,n,\overline{m}}\| \right) \nonumber \\
		&\quad \wedge \Big(\sup_{r\in[0,1]} \|D^2f_1(X_{t,n,\overline{m}} + r(G_{t,n,m}-G_{t,n,\overline{m}}))\| \;  \|G_{t,n,m}-G_{t,n,\overline{m}}\|^2 \Big) \Big) \nonumber \\
		&\qquad \qquad \cdot \left|f_2(X_{t',n,m'}) - \E f_2(X_{t',n,m'})\right| \Big] \nonumber \\
		&\leq C\, \E \left[ 1 \wedge \|J_{t,n,m,2}-J_{t,n,\overline{m},2}\|^2 \| \right] \nonumber \\
		\begin{split}
			&\quad + C\, \E \Big[\sup_{r\in[0,1]} \|D^2f_1(X_{t,n,\overline{m}} + r(Y_{t,n,m}-Y_{t,n,\overline{m}}))\| \;  \|G_{t,n,m}-G_{t,n,\overline{m}}\|^2 \\
			&\qquad\qquad \cdot \left|f_2(X_{t',n,m'}) - \E f_2(X_{t',n,m'})\right| \Big]  
		\end{split}\label{eqn:cov-3}
	\end{align}
	Just as for the term \eqref{eqn:cov-1}, we can use Lemma \ref{lem:truncexpect} to obtain 
	\begin{align*}
		\E \left(1\wedge \|J_{t,n,m,2} - J_{t,n,\widetilde{m},2}\|^2\right)\leq \sum_{j=q_2'+1}^{q_2} b_{n,j}^{\beta_j} \overline{m}^{1+\delta_j\beta_j}.
	\end{align*}
	We now treat the leading term in \eqref{eqn:cov-lin-1}. 
	Recall that $t=t'+\overline{m}$ and $m>\overline{m}$.
	Then
	\begin{align*}
		Y_{t,n,m}-Y_{t,n,\overline{m}} 
		&= \sum_{i=1}^{q_1} a_{n,i}\int_{t-m}^{t'} h_i(s-t)\, d( B^{i,t}_s - B^i_s ) \\
		& \quad + \sum_{j=q_2'+1}^{q_2} b_{n,j} \int_{t-m}^{t'} g_j(s-t)\, d( Z^{j,t}_s - Z^j_s ),
	\end{align*}
	so that
	\begin{subequations}
		\begin{align}
			&\E \left[ Df_1(X_{t,n,\overline{m}}) (Y_{t,n,m}-Y_{t,n,\overline{m}}) \left(f_2(X_{t',n,m'}) - \E f_2(X_{t',n,m'})\right) \right] \nonumber \\
			&= \sum_{i=1}^{q_1} a_{n,i} \E \left[ Df_1(X_{t,n,\overline{m}}) \left(\int_{t-m}^{t'} h_i(s-t)\, d( B^{i,t}_s - B^i_s ) \right) \left(f_2(X_{t',n,m'}) - \E f_2(X_{t',n,m'})\right) \right] \label{eqn:cov-zero-1} \\
			&\quad + \sum_{j=q_2'+1}^{q_2} b_{n,j} \E \left[ Df_1(X_{t,n,\overline{m}}) \left(\int_{t-m}^{t'} g_j(s-t)\, d( Z^{j,t}_s - Z^j_s ) \right) \left(f_2(X_{t',n,m'}) - \E f_2(X_{t',n,m'})\right) \right] \label{eqn:cov-zero-2}
		\end{align}
	\end{subequations}
	Since $X_{t,n,\overline{m}}$ and $B^{i,t}$ are independent of $X_{t',n,m'}$, we find that
	\begin{align*}
		&\E \left[ Df_1(X_{t,n,\overline{m}}) \int_{t-m}^{t'} h_i(s-t)\, d B^{i,t}_s \left(f_2(X_{t',n,m'}) - \E f_2(X_{t',n,'m})\right) \right] \\
		&= \E \left[ Df_1(X_{t,n,\overline{m}}) \int_{t-m}^{t'} h_i(s-t)\, d B^{i,t}_s \right] \E\left[f_2(X_{t',n,m'}) - \E f_2(X_{t',n,m'})\right] \qquad =0.
	\end{align*}
	Moreover, $X_{t,n,\overline{m}}$ is independent of $X_{t',n,m'}$ and $B^i_s, s\in[t-m,t']$.
	Thus,
	\begin{align*}
		&\E \left[ Df_1(X_{t,n,\overline{m}}) \int_{t-m}^{t'} h_i(s-t)\, d B^{i}_s \left(f_2(X_{t',n,m'}) - \E f_2(X_{t',n,m'})\right) \right] \\
		&= \E \left[ Df_1(X_{t,n,\overline{m}})\right] \E\left[ \int_{t-m}^{t'} h_i(s-t)\, d B^{i,t}_s \left(f_2(X_{t',n,m'}) - \E f_2(X_{t',n,m'})\right)\right] \qquad =0.
	\end{align*}
	In the last step, we used that $\E [Df_1(X_{t,n,\overline{m}})]=0$ since $Df_1$ is an odd function and $X_{t,n,\overline{m}}$ has a symmetric distribution.
	The same argument applies for the term \eqref{eqn:cov-zero-2} because $\beta_j>1$ for $j>q_2'$ such that $\E \|\int_{t-m}^{t'} h_i(s-t)dZ_s^j\|<\infty$.
	Hence, we have shown that the leading term in \eqref{eqn:cov-lin-1} vanishes: 
	\begin{align*}
		\E \left[ Df_1(X_{t,n,\overline{m}}) (Y_{t,n,m}-Y_{t,n,\overline{m}}) \left(f_2(X_{t',n,m'}) - \E f_2(X_{t',n,m'})\right) \right] = 0.
	\end{align*}
	Hence
	\begin{align}
		&\left| \Cov\left[ f_1(G_{t,n,m} + J_{t,n,\overline{m},1} + J_{t,n,m,2}) - f_1(Y_{t,n}), f_2(X_{t',n,m'})\right] \right| \quad = |\Gamma_{t,t',n}|\nonumber \\
		&\leq C\, \sum_{j=q_2'+1}^{q_2} b_{n,j}^{\beta_j} \overline{m}^{1+\delta_j\beta_j}\nonumber \\
		\begin{split}
			&\quad + C\, \E \Big[ \sup_{r\in[0,1]} \|D^2f_1(X_{t,n,\overline{m}} + r(Y_{t,n,m}-Y_{t,n,\overline{m}}))\| \; \|G_{t,n,m}-G_{t,n,\overline{m}}\|^2 \\
			&\qquad \qquad \cdot \left|f_2(X_{t',n,m'}) - \E f_2(X_{t',n,m'})\right| \Big]. 
		\end{split}
	\end{align}
	It remains to bound \eqref{eqn:cov-3}, where we distinguish the cases (i) and (ii).
	
	\noindent
	\underline{Ad \eqref{eqn:cov-3}, case (i):}
	We exploit the boundedness of $D^2f_1$, $D^2f_2$, $f_2$, and the fact that $Df_2(0)=0$, to find that
	\begin{align*}
		\eqref{eqn:cov-3} 
		&\leq C \E \left[ \|G_{t,n,m}-G_{t,n,\overline{m}}\|^2 |\E f_2(X_{t',n,m'})| \right] \\
		& + C \E \left[ \|G_{t,n,m}-G_{t,n,\overline{m}}\|^2 \left( 1 \wedge \|X_{t',n,m'}\|^2 \right) \right].
	\end{align*}
	Lemma \ref{lem:mom}(i) yields $\E |f_2(X_{t',n,m'})| \leq C \sum_{i=1}^{q_1} a_{n,i}^2 + C \sum_{j=1}^{q_2} b_{n,j}^{\beta_j}$, such that
	\begin{align*}
		\E \left[ \|G_{t,n,m}-G_{t,n,\overline{m}}\|^2 |\E f_2(X_{t',n,m'})| \right] 
		&\leq C \left( \sum_{i=1}^{q_1} a_{n,i}^2 \right)\left( \sum_{i=1}^{q_1} a_{n,i}^2 +  \sum_{j=1}^{q_2} b_{n,j}^{\beta_j}\right).
	\end{align*}
	Moreover, 
	\begin{align}
		&\E \left[ \|G_{t,n,m}-G_{t,n,\overline{m}}\|^2 \left( 1 \wedge \|X_{t',n,m'}\|^2 \right) \right] \nonumber \\
		&\leq  2\E \left[ \|G_{t,n,m}-G_{t,n,\overline{m}}\|^2 \left( 1 \wedge \|J_{t',n,m'}\|^2 \right) \right] \nonumber \\
		&\quad  + 2\E \left[ \|G_{t,n,m}-G_{t,n,\overline{m}}\|^2  \|G_{t',n,m'}\|^2 \right] \nonumber \\
		\begin{split}
			&\leq  2\E \left[ \|G_{t,n,m}-G_{t,n,\overline{m}}\|^2\right] \;  \E \left( 1 \wedge \|J_{t',n,m'}\|^2 \right)  \\
			&\quad  + 2\sqrt{\E \|G_{t,n,m}-G_{t,n,\overline{m}}\|^4}  \sqrt{\E \|G_{t',n,m'}\|^4}.
		\end{split} \label{eqn:case-i}
	\end{align}
	Here, we used that $J_{t', n,m'}$ is independent from $G_{t,n,m}$ and $G_{t',n,m'}$. 
	Note that $\E(1 \wedge \|J_{t',n,m'}\|^2) \leq C \sum_{j=1}^{q_2} b_{n,j}^{\beta_j}$ by Lemma \ref{lem:truncexpect}.
	Moreover, by Itô's isometry, we have
	\begin{align*}
		\E\|G_{t,n,m} - G_{t,n,\overline{m}}\|^2 
		&=  \sum_{i=1}^{q_1} a_{n,i}^2 \int_{t-m}^{t-\overline{m}} \|h_i(s-t)\|^2\, ds \\
		&\leq C \int_{-\infty}^{-\overline{m}} |s|^{2\delta_0}\, ds \sum_{i=1}^{q_1} a_{n,i}^2 \\
		&\leq C \overline{m}^{1+2\delta_0}  \sum_{i=1}^{q_1} a_{n,i}^2.
	\end{align*}
	Since $G_{t,n,m}-G_{t,n,\overline{m}}$ is Gaussian, we obtain $\E\|G_{t,n,m}-G_{t,n,\overline{m}}\|^4 \leq C (\E\|G_{t,n,m}-G_{t,n,\overline{m}}\|^2)^2$, such that \eqref{eqn:case-i} is upper bounded by 
	\begin{align*}
		\eqref{eqn:case-i} 
		& \leq C \left( \overline{m}^{1+2\delta_0} \sum_{i=1}^{q_1} a_{n,i}^2 \right) \left(\sum_{j=1}^{q_2} b_{n,j}^{\beta_j}\right) +  C \left( \overline{m}^{1+2\delta_0} \sum_{i=1}^{q_1} a_{n,i}^2 \right) \left(  \sum_{i=1}^{q_1} a_{n,i}^2 \right),
	\end{align*}
	so that
	\begin{align*}
		\eqref{eqn:cov-3} 
		&\leq C \left( \overline{m}^{1+2\delta_0} \sum_{i=1}^{q_1} a_{n,i}^2 \right) \left(\sum_{i=1}^{q_1} a_{n,i}^2 + \sum_{j=1}^{q_2} b_{n,j}^{\beta_j}\right).
	\end{align*}
	
	\noindent
	\underline{Ad \eqref{eqn:cov-3}, case (ii), $\eta_1>0$:}
	Since $D^2f_1(x)=0$ for $\|x\|\leq \eta_1$, we have for any $\lambda\in(0,1)$,
	\begin{align*}
		\eqref{eqn:cov-3} 
		&\leq C\, \E \left[  \mathbb{I}  \cdot \|G_{t,n,m}-G_{t,n,\overline{m}}\|^2  \cdot \left|f_2(X_{t',n,m'}) - \E f_2(X_{t',n,m'})\right| \right] \\
		&\leq C\, \E \left[  \mathbb{I}  \cdot \|G_{t,n,m}-G_{t,n,\overline{m}}\|^2  \right] \qquad \text{where} \\[1.5 ex]
		\mathbb{I} 
		&=\mathds{1}(\|X_{t,n,\overline{m}}\|>\eta_1) +\mathds{1}(\|G_{t,n,m} + J_{t,n,\overline{m},1} + J_{t,n,m,2}\|>\eta_1) \\
		&\leq \mathds{1}(\|G_{t,n,m}\|>\eta_1\lambda) + \mathds{1}(\|G_{t,n,\overline{m}}\|>\eta_1\lambda) \\
		&\quad + 2\cdot\mathds{1}\left(\|J_{t,n,\overline{m},1}\|>\tfrac{\eta_1(1-\lambda)}{2}\right) + \mathds{1}\left(\|J_{t,n,\overline{m},2}\|>\tfrac{\eta_1(1-\lambda)}{2}\right) + \mathds{1}\left(\|J_{t,n,m,2}\|>\tfrac{\eta_1(1-\lambda)}{2}\right) \\
		&= \mathbb{I}_G + \mathbb{I}_J,
	\end{align*}
	where $\mathbb{I}_G = \mathds{1}(\|G_{t,n,m}\|>\eta_1\lambda) + \mathds{1}(\|G_{t,n,\overline{m}}\|>\eta_1\lambda)$, and $\mathbb{I}_J$ accordingly.
	Since $\mathbb{I}_J$ is independent of $G_{t,n,m}$ and $G_{t,n,\overline{m}}$,
	\begin{align*}
		\E \left[  \mathbb{I}_J  \cdot \|G_{t,n,m}-G_{t,n,\overline{m}}\|^2  \right] 
		&\leq C \E(\mathbb{I}_J) \E \|G_{t,n,m}-G_{t,n,\overline{m}}\|^2 \\
		&\leq C \left( \sum_{j=1}^{q_2} b_{n,j}^{\beta_j} \right) \left( \overline{m}^{1+2\delta_0}  \sum_{i=1}^{q_1} a_{n,i}^2 \right),
	\end{align*}
	using the polynomial tail bound \eqref{eqn:tailbound} for the jump processes.
	Moreover, for any $q>1$ and $p>1$ such that $\frac{1}{p}+\frac{1}{q}=1$,
	\begin{align*}
		\E \left[  \mathbb{I}_G  \cdot \|G_{t,n,m}-G_{t,n,\overline{m}}\|^2  \right]
		&\leq C\E(\mathbb{I}_G)^\frac{1}{q}  \left(\E \|G_{t,n,m}-G_{t,n,\overline{m}}\|^{2p}\right)^\frac{1}{p} \\
		&\leq C \mathbb{P}\left(\|G_{t,n,m}\|>\eta_1\lambda\right)^\frac{1}{q} \left( \overline{m}^{1+2\delta_0}  \sum_{i=1}^{q_1} a_{n,i}^2 \right),
	\end{align*}
	because $G_{t,n,m}\deq G_{t,n,\overline{m}}$.
	Each component $G_{t,n,m}^l$, $l=1,\ldots, d$, is a Gaussian random variable with mean zero and variance $\sigma_{l,n}^2 = \sum_{i=1}^{q_1} a_{n,i}^2 \|h_i^l\|_{L_2}^2 \leq \overline{\sigma}^2 \sum_{i=1}^{q_1} a_{n,i}^2 $.
	We may then employ the Gaussian tail bound $\mathbb{P}(|N|>x) \leq 2 \exp(-x^2/2)$ for a standard Gaussian random variable to obtain
	\begin{align*}
		\mathbb{P}\left(\|G_{t,n,m}\|^2>\eta_1^2\lambda^2\right)
		&\leq \sum_{l=1}^d \mathbb{P}\left(|G_{t,n,m}^l| > \tfrac{\eta_1 \lambda}{\sqrt{d}}\right) \\
		&\leq 2\sum_{l=1}^d \exp\left( -\frac{\eta_1^2\lambda^2}{2d \sigma_{l,n}^2} \right) \\
		& \leq C \exp\left( -\frac{\eta_1^2\lambda^2}{2d \sigma^2 \sum_{i=1}^{q_1} a_{n,i}^2 } \right),
	\end{align*}
	such that
	\begin{align*}
		&\E \left[  \mathbb{I}_G  \cdot \|G_{t,n,m}-G_{t,n,\overline{m}}\|^2  \cdot \left|f_2(X_{t',n,m'}) - \E f_2(X_{t',n,m'})\right| \right] \\
		&\leq C \exp\left( -\frac{\eta_1^2\lambda^2}{2d q \sigma^2 \sum_{i=1}^{q_1} a_{n,i}^2 }\right) \left( \overline{m}^{1+2\delta_0}  \sum_{i=1}^{q_1} a_{n,i}^2 \right).
	\end{align*}
	Since $q>1$ and $\lambda\in(0,1)$ are arbitrary, we obtain that for any $\lambda\in(0,1)$
	\begin{align*}
		\eqref{eqn:cov-3} 
		&\leq C\left( \overline{m}^{1+2\delta_0}  \sum_{i=1}^{q_1} a_{n,i}^2 \right) \left[  \exp\left( -\frac{\eta_1^2\lambda^2}{2d \sigma^2 \sum_{i=1}^{q_1} a_{n,i}^2 }\right) + \sum_{j=1}^{q_2} b_{n,j}^{\beta_j} \right].
	\end{align*}
	\noindent
	\underline{Ad \eqref{eqn:cov-3}, case (ii), $\eta_2>0$:}
	Since $\|D^2f_1\|_\infty<\infty$, we have, for any $\lambda\in(0,1)$,
	\begin{align*}
		\eqref{eqn:cov-3} 
		&\leq C \E \left[\|G_{t,n,m} - G_{t,n,\overline{m}}\|^2 \cdot |f_2(X_{t',n,m'})| \right] 
		+ C \E \left[\|G_{t,n,m} - G_{t,n,\overline{m}}\|^2\right] |\E f_2(X_{t',n,m'})| \\
		&\leq C \E \left[\|G_{t,n,m} - G_{t,n,\overline{m}}\|^2 \cdot \left[\mathds{1}(\|G_{t',n,m'}\|>\eta_2\lambda) + \mathds{1}(\|J_{t',n,m'}\|>\eta_2(1-\lambda))\right] \right] \\
		&\qquad + C \E \left[\|G_{t,n,m} - G_{t,n,\overline{m}}\|^2\right] |\E f_2(X_{t',n,m'})|.
	\end{align*}
	Lemma \ref{lem:mom}(i) yields that 
	\begin{align*}
		|\E f_2(X_{t',n,m'})| 
		&\leq C \left[\exp\left( -\frac{\eta_2^2\lambda^2}{2d \sigma^2 \sum_{i=1}^{q_1} a_{n,i}^2 }\right) + \sum_{j=1}^{q_2} b_{n,j}^{\beta_j}\right].
	\end{align*}
	Note that this upper bound also holds trivially if $a_{n,i}\not\to 0$ or $b_{n,j}\not\to 0$.
	Moreover, since $J_{t',n,m'}$ is independent of $G_{t,n,m}$ and $G_{t,n,\overline{m}}$,
	\begin{align*}
		&\E \left[\|G_{t,n,m} - G_{t,n,\overline{m}}\|^2 \cdot \mathds{1}(\|J_{t',n,m'}\|>\eta_2(1-\lambda))\right] \\
		&= \E \|G_{t,n,m} - G_{t,n,\overline{m}}\|^2 \cdot \mathbb{P}\left(\|J_{t',n,m'}\| > \eta_2(1-\lambda)\right) \\
		&\leq C\,\E \|G_{t,n,m} - G_{t,n,\overline{m}}\|^2 \sum_{j=1}^{q_2} b_{n,j}^{\beta_j},
	\end{align*}
	for a constant $C=C(\lambda)$, by virtue of the polynomial tail bound for the stable component \eqref{eqn:tailbound}.
	Finally, Hölder's inequality yields for any $q>1$ and $p>1$ such that $\frac{1}{p}+\frac{1}{q}=1$,
	\begin{align*}
		&\E \left[\|G_{t,n,m} - G_{t,n,\overline{m}}\|^2 \cdot \mathds{1}(\|G_{t',n,m'}\|>\eta_2\lambda) \right] \\
		&\leq \left(\E \|G_{t,n,m} - G_{t,n,\overline{m}}\|^{2p}\right)^{\frac{1}{p}} \mathbb{P}\left( \|G_{t',n,m'}\|>\eta_2\lambda  \right)^\frac{1}{q} \\
		&\leq C \exp\left( -\frac{\eta_2^2\lambda^2}{2d q \sigma^2 \sum_{i=1}^{q_1} a_{n,i}^2 }\right) \left( \overline{m}^{1+2\delta_0}  \sum_{i=1}^{q_1} a_{n,i}^2 \right),
	\end{align*}
	just as in case (ii).
	Since $q>1$ and $\lambda\in(0,1)$ are arbitrary, we find that for any $\lambda\in(0,1)$
	\begin{align*}
		\eqref{eqn:cov-3}
		&\leq C \left( \overline{m}^{1+2\delta_0}  \sum_{i=1}^{q_1} a_{n,i}^2 \right) \left[  \exp\left( -\frac{\eta_2^2\lambda^2}{2d \sigma^2 \sum_{i=1}^{q_1} a_{n,i}^2 }\right) + \sum_{j=1}^{q_2} b_{n,j}^{\beta_j} \right].
	\end{align*}
	This completes the proof.
\end{proof}

\begin{proof}[Proof of Theorem \ref{thm:variance-bound}]
	Note that
	\begin{align*}
		\left|\Var \left( S_n(f) \right)\right|
		&\leq  \frac{1}{n^2}\sum_{t=1}^n \sum_{z\in\Z} \left|\Cov(f(X_{t,n}), f(X_{t+z,n}))\right| \\
		&\leq \frac{1}{n}\left[C\sum_{z\in\Z} |z|^{1+\delta^{\star}} \right] \left[ \sum_{i=1}^{q_1} a_{n,i}^4 + \left( 1+ \sum_{i=1}^{q_1} a_{n,i}^2 \right) \sum_{j=1}^{q_2} b_{n,j}^{\beta_j}  \right],
	\end{align*}
	using Lemma \ref{lem:autocov-decay}(i) in the last step.
	The first series is finite because $\delta^{\star}<-2$.
	The second term may be simplified because $\sum_{i=1}^{q_1} a_{n,i}^2\leq A$ for some $A$. 
	Hence, the constant $C$ in the Theorem also depends on this upper bound $A$.
	The second inequality may be derived analogously via Lemma \ref{lem:autocov-decay}(ii).
\end{proof}

\subsection{Central limit theorem for multiscale moving average processes} \label{sec5.3}
This section is devoted to the proof of Theorem \ref{thm:clt}. We start with the following result.
\begin{lemma}\label{lem:approximation}
	Let $f\in \mathfrak{F}^0_\eta$, and assume furthermore that $1+\beta_j\delta_j<0$, and $1+2\delta_0<0$.
	\begin{enumerate}[(i)]
		\item If $\eta=0$, then there exists some $C>0$ such that for all $t,m,n\in\N$,
		\begin{align*}
			\E \left| f(X_{t,n}) - f(X_{t,n,m}) \right|^2 
			\leq C m^{1+2\delta_0}  \left[\sum_{i=1}^{q_1} a_{n,i}^4 + \sum_{i=1}^{q_1} a_{n,i}^2 \sum_{j=1}^{q_2} b_{n,j}^{\beta_j}\right]  + C\sum_{j=1}^{q_2} b_{n,j}^{\beta_j} m^{1+\beta_j\delta_j}.
		\end{align*}
		\item If $\eta>0$, then for any $\lambda\in(0,1)$, there exists some $C>0$ such that for all $t,m,n\in\N$,
		\begin{align*}
			\E \left| f(X_{t,n}) - f(X_{t,n,m}) \right|^2 
			&\leq C m^{1+2\delta_0}\sum_{i=1}^{q_1} a_{n,i}^2  \left[  \exp\left( -\frac{\eta^2\lambda^2}{2d \overline{\sigma}^2 \sum_{i=1}^{q_1} a_{n,i}^2 }\right) + \sum_{j=1}^{q_2} b_{n,j}^{\beta_j} \right] \\
			&\quad + C\sum_{j=1}^{q_2} b_{n,j}^{\beta_j} m^{1+\beta_j\delta_j}.
		\end{align*}
	\end{enumerate}
\end{lemma}
\begin{proof}[Proof of Lemma \ref{lem:approximation}]
	For $m\in \N\cup\{\infty\}$, recall the definition
	\begin{align*}
		G_{t,n,m} &= \sum_{i=1}^{q_1} a_{n,i} \left[\int_{t-m}^t h_i(s-t)\,dB^i_s + \int_{-\infty}^{t-m} h_i(s-t)\,dB^{i,t}_s \right], \\
		J_{t,n,m} &=  \sum_{j=1}^{q_2} b_{n,j} \left[ \int_{t-m}^t g_j(s-t)\, dZ^j_s + \int_{-\infty}^{t-m} g_j(s-t)\, dZ^{j,t}_s \right].
	\end{align*}
	such that $X_{t,n} = G_{t,n,\infty}+J_{t,n,\infty}$.
	Then
	\begin{align}
		\begin{split}
			\E \left| f(X_{t,n}) - f(X_{t,n,m}) \right|^2 
			&\leq 2 \E \left| f(X_{t,n}) - f(G_{t,n,\infty} + J_{t,n,m}) \right|^2 \\
			&\quad + 2 \E \left| f(G_{t,n,\infty}+J_{t,n,m}) - f(X_{t,n,m}) \right|^2, 
		\end{split}\label{eqn:approx-1} \\
		\E \left| f(X_{t,n}) - f(G_{t,n,\infty} + J_{t,n,m}) \right|^2 
		&\leq \E \left[ 1 \wedge \|J_{t,n,\infty} - J_{t,n,m}\|^2 \right] \nonumber \\
		&\leq C\sum_{j=1}^{q_2} \E \left[ 1 \wedge \left\| b_{n,j} \int_{-\infty}^{t-m} g_j(s-t) dZ^j_s \right\|^2  \right]. \nonumber
	\end{align}
	Since 
	\begin{align*}
		\int_{-\infty}^{t-m} g^l_j(s-t) dZ^j_s 
		&\deq \left(\int_{-\infty}^{-m}|g^l_j(s)|^{\beta_j}\,ds\right)^\frac{1}{\beta_j} Z^j_1, \\
		\int_{-\infty}^{-m}|g^l_j(s)|^{\beta_j}\,ds 
		&\leq C \int_{m}^{\infty} |s|^{\beta_j \delta_j}\, ds \quad \leq C m^{1+\beta_j\delta_j},		
	\end{align*}
	we obtain via Lemma \ref{lem:truncexpect} that $\E \left[ 1 \wedge \|J_{t,n,\infty} - J_{t,n,m}\|^2 \right] \leq \sum_{j=1}^{q_2} b_{n,j}^{\beta_j} m^{1+\delta_j\beta_j}$.
	
	In view of \eqref{eqn:approx-1}, it remains to study
	\begin{align}
		&\E \left| f(G_{t,n,\infty}+J_{t,n,m}) - f(X_{t,n,m}) \right|^2 \nonumber \\
		&\leq \E \left[ \|G_{t,n,\infty} - G_{t,n,m}\|^2 \sup_{r\in [0,1]} \left\|Df\left(J_{t,n,m} + r G_{t,n,m} + (1-r) G_{t,n,\infty}\right)\right\|^2 \right] \label{eqn:approx-2} \\
		&\leq  \E \left[ \|G_{t,n,\infty} - G_{t,n,m}\|^2 \left(1 \wedge (\|J_{t,n,m}\|^2 + \|G_{t,n,m}\|^2 + \|G_{t,n,\infty}\|^2)\right) \right] \nonumber 
		\intertext{using the boundedness of $D^2f$, and the fact that $Df(0)=0$,}
		&\leq \E  \|G_{t,n,\infty} - G_{t,n,m}\|^2\, \E \left( 1 \wedge \|J_{t,n,m}\|^2 \right) + \sqrt{\E \|G_{t,n,\infty} - G_{t,n,m}\|^4} \sqrt{\E \|G_{t,n,m}\|^4}, \nonumber
		\intertext{since $G_{t,n,m}\deq G_{t,n,\infty}$, and since $J_{t,n,m}$ is independent of $G_{t,n,m}$ and $G_{t,n,\infty}$,}
		&\leq C \E  \|G_{t,n,\infty} - G_{t,n,m}\|^2 \sum_{j=1}^{q_2} b_{n,j}^{\beta_j} + C\,\E  \|G_{t,n,\infty} - G_{t,n,m}\|^2 \E \|G_{t,n,m}\|^2. \nonumber
	\end{align}
	In the last step, we used Lemma \ref{lem:truncexpect} to bound the term $J_{t,n,m}$, and the fact that $\E \|N\|^p \leq C (\E \|N\|^2)^\frac{p}{2}$ for a Gaussian random vector $N$.
	Now note that $\E \|G_{t,n,m}\|^2 \leq C \sum_{i=1}^{q_1} a_{n,i}^2$.
	Furthermore, Itô's isometry yields 
	\begin{align*}
		\E  \|G_{t,n,\infty} - G_{t,n,m}\|^2 
		&= 2\sum_{i=1}^{q_1} a_{n,i}^2  \int_{-\infty}^{t-m} \|h_i(s-t)\|^2 \, ds \\
		&\leq C \sum_{i=1}^{q_1} a_{n,i}^2 \int_{m}^\infty |s|^{2\delta_0}\, ds \\
		&\leq C m^{1+2\delta_0} \sum_{i=1}^{q_1} a_{n,i}^2.
	\end{align*}
	This establishes the claim (i).
	
	For case (ii), we bound \eqref{eqn:approx-2} as
	\begin{align*}
		\eqref{eqn:approx-2}
		&\leq C \,\E \Big[ \|G_{t,n,\infty} - G_{t,n,m}\|^2 \\
		&\qquad \cdot \left[\mathds{1}(\|J_{t,n,m}\|> \eta(1-\lambda)) + \mathds{1}(\|G_{t,n,m}\|> \eta\lambda) + \mathds{1}(\|G_{t,n,\infty}\|> \eta\lambda) \right]   \Big] \\
		&\leq C \, \E\|G_{t,n,\infty} - G_{t,n,m}\|^2 \mathbb{P}(\|J_{t,n,m}\|>\eta(1-\lambda)) \\
		&\quad + C (\E \|G_{t,n,\infty} - G_{t,n,m}\|^{2p})^\frac{1}{p} \mathbb{P}(\|G_{t,n,m}\|>\eta\lambda)^\frac{1}{q} \\
		&\leq C\, \E\|G_{t,n,\infty} - G_{t,n,m}\|^2 \left[ \mathbb{P}(\|J_{t,n,m}\|>\eta(1-\lambda)) + \mathbb{P}(\|G_{t,n,m}\|>\eta\lambda)^\frac{1}{q}\right],
	\end{align*}
	for any $\lambda\in(0,1)$ and any $p,q>1$, $\frac{1}{p} + \frac{1}{q} = 1$, and the constant $C$ depends on $\lambda$ and $q$.
	The probabilities may be bounded as in the proof of Lemma \ref{lem:autocov-decay}, such that the previous term is bounded by
	\begin{align*}
		&C\, \E\|G_{t,n,\infty} - G_{t,n,m}\|^2 \left[ \sum_{j=1}^{q_2} b_{n,j}^{\beta_j} +  \exp\left( -\frac{\eta^2\lambda^2}{2d q \sigma^2 \sum_{i=1}^{q_1} a_{n,i}^2 } \right) \right] \\
		&\leq C m^{1+2\delta_0} \sum_{i=1}^{q_1} a_{n,i}^2 \left[ \sum_{j=1}^{q_2} b_{n,j}^{\beta_j} +  \exp\left( -\frac{\eta^2\lambda^2}{2d q \sigma^2 \sum_{i=1}^{q_1} a_{n,i}^2 } \right) \right].
	\end{align*}
	Since $\lambda\in(0,1)$ and $q>1$ are arbitrary, we find that 
	\begin{align*}
		&\E \left| f(G_{t,n,\infty}+J_{t,n,m}) - f(X_{t,n,m}) \right|^2 \\
		&\leq C m^{1+2\delta_0} \sum_{i=1}^{q_1} a_{n,i}^2 \left[ \sum_{j=1}^{q_2} b_{n,j}^{\beta_j} +  \exp\left( -\frac{\eta^2\lambda^2}{2d \sigma^2 \sum_{i=1}^{q_1} a_{n,i}^2 } \right) \right]
	\end{align*}
	Plugging this into \eqref{eqn:approx-1} establishes claim (ii).
\end{proof}

To formulate the next Lemma, we need some more notation.
For any $t\in\N$ and any $m\in\N\cup\{\infty\}$, we denote
\begin{align*}
	\mathcal{G}_{t,i,m} &= \int_{t-m}^t h_i(s-t)\, dB^i_s + \int_{-\infty}^{t-m} h_i(s-t)\, dB^{i,t}_s, \quad i=1,\ldots, q_1, \\
	\mathcal{J}_{t,j,m} &= \int_{t-m}^t g_j(s-t)\, dZ_s^j + \int_{-\infty}^{t-m} g_j(s-t)\, dZ_s^{j,t},\quad j=1,\ldots, q_2,\\
	X_{t,n,m}&= \sum_{i=1}^{q_1} a_{n,i} \mathcal{G}_{t,i,m} + \sum_{j=1}^{q_2} b_{n,j} \mathcal{J}_{t,j,m}.
\end{align*}
Then $\mathcal{G}_{t,i,m}$ is a Gaussian time series, and its autocovariance matrices $\Cov(\mathcal{G}_{t,i,m}, \mathcal{G}_{t',i,m})=\Sigma^{i,m,t-t'}$ are given by
\begin{align*}
	\Sigma_{l,l'}^{i,m,r}  &= \begin{cases}
		\int_{-\infty}^0 h_i^l(s) h_i^{l'}(s)\, ds, & r=0, \\
		\int_{(0\vee r)-m}^{(0\wedge r)} h_i^l(s-r)h_i^{l'}(s)\, ds, & |r| = 1,\ldots, m-1,\\
		0, & |r| \geq m,
	\end{cases}
\end{align*}
for $r\in\Z$ and $l,l'=1,\ldots, d$.
Note that $\Sigma^{i,m,-r}_{l,l'} = \Sigma^{i,m,r}_{l',l}$, i.e.\ $(\Sigma^{i,m,r})^T = \Sigma^{i,m,-r}$.

Moreover, $(\mathcal{J}_{t,j,m},\mathcal{J}_{t',j,m}) \deq (\mathcal{J}_{t-t',j,m},\mathcal{J}_{0,j,m})$ is a $2d$-dimensional $\beta_j$-stable random vector, and we denote its Lévy measure by $\nu_{r, j,m}$ for $r=t-t'\in\Z$.

\begin{lemma}\label{lem:autocov-limit}
	Let $f\in\mathfrak{F}_\eta$ for some $\eta\geq 0$, and $a_{n,i}\to 0$, $b_{n,j}\to 0$ as $n\to\infty$.
	\begin{enumerate}[(i)]
		\item Let $i^{\star}\in\{1,\ldots, q_1\}$ be such that $a_{n,i^{\star}}^2 \gg \max_{i\neq i^{\star}} a_{n,i}^2$.
		For any $t,t'\in \N$ and any $m\in \N\cup\{\infty\}$,
		\begin{align*}
			&\Cov(f(X_{t,n,m}), f(X_{t',n,m})) \\
			&= \frac{ a_{n,i^*}^4}{4} \sum_{l,l',r,r'=1}^d \left[ D^2_{ll'}f(0) D^2_{rr'}f(0) \right]  \Sigma^{i^*,m,t-t'}_{lr}\Sigma^{i^*,m,t-t'}_{l'r'} + \sum_{j=1}^{q_2} b_{n,j}^{\beta_j} \overline{f}^{[\nu_{t-t',j,m}]}(0)
			\\
			&\quad  + o\left(a_{n,i^*}^4\right)
			+ o\left( \sum_{j=1}^{q_2} b_{n,j}^{\beta_j} \right),
		\end{align*}
		where $\overline{f}:\R^{d} \times \R^d\to\R, (x,y)\mapsto f(x)f(y)$.
		\item If $f\in\mathfrak{F}_\eta^0$ and $\eta>0$, then for any fixed $\lambda\in(0,1)$, any $t,t'\in\N$, and any $m\in \N\cup\{\infty\}$,
		\begin{align*}
			\Cov(f(X_{t,n,m}), f(X_{t',n,m}))
			&= \sum_{j=1}^{q_2} b_{n,j}^{\beta_j} \overline{f}^{[\nu_{t-t',j,m}]}(0) + o\left( \sum_{j=1}^{q_2} b_{n,j}^{\beta_j} \right) \\
			&\qquad+ \mathcal{O} \left( \exp\left( -\frac{\eta^2\lambda^2}{2d \sum_{i=1}^{q_1} a_{n,i}^2 \sigma^2} \right) \right).
		\end{align*}	
	\end{enumerate}
\end{lemma}
Recall from \eqref{eqn:fnu-deriv} that
\begin{align*}
    \overline{f}^{[\nu_{z,j,m}]} = \left[ \frac{d}{dh} \E \left( f\left(h^\frac{1}{\beta} \mathcal{J}_{z,j,m}\right)\cdot f\left(h^\frac{1}{\beta} \mathcal{J}_{0,j,m}\right) \right) \right]_{h=0}.
\end{align*}

\begin{proof}[Proof of Lemma \ref{lem:autocov-limit}]
	By virtue of Lemma \ref{lem:mom}(iii), we have for $f\in\mathfrak{F}^0_\eta$,
	\begin{align}
		\begin{split}
			&\E f\left( X_{t,n,m} \right) = \E f(X_{t',n,m}) = \E f(X_{t,n}) \\
			&= \frac{1}{2}\sum_{i=1}^{q_1} a_{n,i}^2 \sum_{l,l'=1}^d \Sigma^{i,m,0}_{l,l'} D^2_{ll'}f(0) + \mathcal{O}\left(\sum_{j=1}^{q_2} b_{n,j}^{\beta_j}\right) + \mathcal{O}\left(\sum_{i=1}^{q_1} a_{n,i}^2\right)^2.
		\end{split} \label{eqn:cov-mean-1}
	\end{align}
	For $f\in\mathfrak{F}^0_\eta$ with $\eta>0$, Lemma \ref{lem:mom}(i) yields
	\begin{align}
		\begin{split}
			&\E f\left( X_{t,n,m} \right) = \E f(X_{t',n,m}) = \E f(X_{t,n}) \\
			&= \mathcal{O}\left(\sum_{j=1}^{q_2} b_{n,j}^{\beta_j}\right) + \mathcal{O}\left( \exp\left( -\frac{\eta^2\lambda^2}{2d \sum_{i=1}^{q_1} a_{n,i}^2 \sigma^2} \right)\right).
		\end{split} \label{eqn:cov-mean-2}
	\end{align}
	Furthermore, we write $\E \left[ f\left( X_{t,n,m} \right) f\left(X_{t',n,m}) \right) \right] 
	= \E \left[ \overline{f}\left(X_{t,n,m} , X_{t',n,m} \right) \right]$ for $\overline{f}:\R^{d}\times \R^d\to \R, \overline{f}((x,y)) = f(x)f(y)$.
	We may thus apply Lemma \ref{lem:mom} for the $2d$-dimensional random vector $\left(X_{t,n,m} , X_{t',n,m} \right)$ and the function $\overline{f}$.
	
	\underline{Case (i):}
	Note that $f(0)=0$ and $Df(0)=0$ imply $D^2 \overline{f}(0)=0$. 
	Therefore, Lemma \ref{lem:mom}(ii) yields
	\begin{align*}
		\E \overline{f}(X_{t,n,m} , X_{t',n,m}) 
		&= \frac{1}{8}\sum_{i,i'=1}^{q_1} a_{i,n}^2 a_{i',n}^2 \sum_{l,l',r,r'=1}^{2d} D^4_{l,l',r,r'}\overline{f}(0)\overline{\Sigma}^{i,m}_{ll'}\overline{\Sigma}^{i',m}_{rr'} \\
		&\quad + \sum_{j=1}^{q_2} b_{n,j}^{\beta_j} \overline{f}^{[\nu_{t-t',j,m}]}(0)+o\left( \sum_{j=1}^{q_2} b_{n,j}^{\beta_j} \right) + o\left( \sum_{i=1}^{q_1} a_{n,i}^2 \right)^2.
	\end{align*}
	Here, $\overline{\Sigma}^{i,m}$ is the covariance matrix of the $2d$-dimensional Gaussian random vector $(\mathcal{G}_{t,i,m}, \mathcal{G}_{t',i,m})$.
	It takes the form
	\begin{align*}
		\overline{\Sigma}^{i,m}_{lr} 
		&= \begin{cases}
			\Sigma_{l,r}^{i,m,0},& l,r\leq d, \\
			\Sigma_{(l-d),(r-d)}^{i,m,0}, & l,r\geq d+1, \\
			\Sigma_{l, (r-d)}^{i,m,t-t'}, & l\leq d, \; r\geq d+1, \\
			\Sigma_{l-d, r}^{i,m,t'-t} = \Sigma_{r, l-d}^{i,m,t-t'}, & r\leq d, \; l\geq d+1.
		\end{cases}
	\end{align*}
	Moreover, the derivative $D^4_{l,l',r,r'}\overline{f}(0)$ is non-zero only if precisely two of the indices $l,l',r,r'$, are smaller or equal than $d$. 
	This is a consequence of the properties $Df(0)=0$ and $f(0)=0$.
	More formally, if we fix the ordering $l\leq l' \leq r \leq r'$, then
	\begin{align*}
		D^4_{l,l',r,r'} \overline{f}(0) 
		&= \begin{cases}
			D^2_{l,l'}f(0) \, D^2_{r,r'} f(0), & l'\leq d,\; r\geq d+1, \\
			0, & \text{otherwise}.
		\end{cases}
	\end{align*}  
	Hence,
	\begin{align*}
		& \sum_{l,l',r,r'=1}^{2d} D^4_{l,l',r,r'}\overline{f}(0)\overline{\Sigma}^{i,m}_{ll'}\overline{\Sigma}^{i,m}_{rr'} \\
		&= \sum_{l,l'=1}^{d} \sum_{r,r'=d+1}^{2d} \left[D^2_{ll'}f(0) D^2_{(r-d), (r'-d))} f(0) \right] \left[2 \overline{\Sigma}^{i,m}_{ll'}\overline{\Sigma}^{i,m}_{rr'} + 2\overline{\Sigma}^{i,m}_{lr}\overline{\Sigma}^{i,m}_{l'r'}+ 2\overline{\Sigma}^{i,m}_{lr'}\overline{\Sigma}^{i,m}_{l'r}\right] \\
		&= 2\sum_{l,l',r,r'=1}^d \left[ D^2_{ll'}f(0) D^2_{rr'}f(0) \right] \left[ \Sigma^{i,m,0}_{ll'} \Sigma^{i,m,0}_{rr'} + \Sigma^{i,m,t-t'}_{lr}\Sigma^{i,m,t-t'}_{l'r'} + \Sigma^{i,m,t-t'}_{lr'}\Sigma^{i,m,t-t'}_{l'r} \right] \\
		&= 2\sum_{l,l',r,r'=1}^d \left[ D^2_{ll'}f(0) D^2_{rr'}f(0) \right] \left[ \Sigma^{i,m,0}_{ll'} \Sigma^{i,m,0}_{rr'} + 2\Sigma^{i,m,t-t'}_{lr}\Sigma^{i,m,t-t'}_{l'r'}\right].
	\end{align*}
	In combination with \eqref{eqn:cov-mean-1}, we find that
	\begin{align*}
		&\Cov(f(X_{t,n,m}), f(X_{t',n,m})) \\
		&= \frac{ a_{n,i^{\star}}^4}{4} \sum_{l,l',r,r'=1}^d \left[ D^2_{ll'}f(0) D^2_{rr'}f(0) \right]  \Sigma^{i^{\star},m,t-t'}_{lr}\Sigma^{i^{\star},m,t-t'}_{l'r'}  + \sum_{j=1}^{q_2} b_{n,j}^{\beta_j} \overline{f}^{[\nu_{t-t',j,m}]}(0)
		\\
		&\quad + \mathcal{O}\left(a_{n,i^{\star}}^2\max_{i\neq i^{\star}} a_{n,i}^2\right)
		+ o\left( \sum_{j=1}^{q_2} b_{n,j}^{\beta_j} \right) + o \left( \sum_{i=1}^{q_1} a_{n,i}^2 \right)^2 .
	\end{align*}
	Since $a_{n,i^*}^2 \gg \max_{i\neq i^*} a_{n,i}^2$, this establishes claim (i).
	
	\underline{Case (ii):}
	If $f(x)=0$ for $\|x\|\leq \eta$, then $\overline{f}(x,y)=0$ for $\|(x,y)\|\leq \eta/\sqrt{2}$.
	In particular, Lemma \ref{lem:mom}(i) yields
	\begin{align*}
		\E \overline{f}(X_{t,n,m} , X_{t',n,m}) 
		&= \sum_{j=1}^{q_2} b_{n,j}^{\beta_j} \overline{f}^{[\nu_{t-t',j,m}]}(0) + o\left( \sum_{j=1}^{q_2} b_{n,j}^{\beta_j} \right) + \mathcal{O} \left( \exp\left( -\frac{\eta^2\lambda^2}{2d \sum_{i=1}^{q_1} a_{n,i}^2 \sigma^2} \right) \right).
	\end{align*}
	Together with \eqref{eqn:cov-mean-2} and $D^2 f(0)=0$, we find that claim (ii) holds.
\end{proof}

We are now ready to provide the proof of the main result of this section, Theorem \ref{thm:clt}.
The asymptotic variance depends on the limiting regime and on the function $f$, but all cases may be expressed in terms of the quantities
\begin{align}
	\xi^2 = 
	\begin{cases}
		\gamma^2_{f,1} = a^{-2}\sum_{z\in\Z} \Cov[ f(a\mathcal{G}_{z,i^{\star},\infty}), f(a\mathcal{G}_{0,i^{\star},\infty})], & \text{case (i)}; \\
		\zeta^2_{f,1} = b^{-\beta_{j^{\star}}/2} \sum_{z\in\Z} \Cov[ f(b\mathcal{J}_{z,j^{\star},\infty}), f(b\mathcal{J}_{0,j^{\star},\infty})]
		,& \text{case (ii)}; \\
		\zeta^2_{f,0} 
		= \sum_{z\in\Z} \overline{f}^{[\nu_{z,\infty,j^{\star}}]}(0),& \text{case (iii)}; \\
		\gamma^2_{f,0} 
		= \frac{1}{4}\sum_{l,l',r,r'=1}^d \left[D^2_{ll'}f(0) D^2_{rr'}f(0)\right] \left[\sum_{z\in\Z} \Sigma^{i^{\star},\infty,z}_{lr}\Sigma^{i^{\star},\infty,z}_{l'r'}\right],& \text{case (iv)}. \\
	\end{cases} \label{eqn:asymp-var}
\end{align}
and we recall that
\begin{align*}
    \overline{f}^{[\nu_{z,\infty,j^{\star}}]}(0)
    &= \left[ \frac{d}{dh} \E \left( f\left( h^\frac{1}{\beta} \mathcal{J}_{z,j^\star,\infty} \right) \cdot f\left( h^\frac{1}{\beta} \mathcal{J}_{0,j^\star,\infty} \right) \right) \right]_{h=0}, \\
    \Sigma^{i,\infty,z}_{lr}
    &= \int_{-\infty}^\infty h_i^l(s-z) h_i^{r}(s)\, ds.
\end{align*}

\begin{proof}[Proof of Theorem \ref{thm:clt}]
	We define the statistic 
	\begin{align}
	S_{n,m}(f) = \frac{1}{n}\sum_{t=1}^n \left[f(X_{t,n,m})-\E f(X_{t,n})\right]
	\end{align}
	Consider the most difficult case (v) first.
	To derive the weak convergence, we will establish the following steps:
	\begin{align}
		\begin{split}
			\lim_{m\to\infty} \limsup_{n\to\infty} \frac{n}{a_{n,i^{\star}}^4} \E \left\| S_{n,m}(f_1)-
			S_{n}(f_1)\right\|^2 &= 0, \\
			\lim_{m\to\infty} \limsup_{n\to\infty} \frac{n}{b_{n,j^{\star}}^{\beta_{j^{\star}}}} \E \left\| S_{n,m}(f_2)-
			S_{n}(f_2) \right\|^2 &= 0, 
		\end{split} \label{eqn:m-limit}
	\end{align}
	and, for all $m\in\N\cup\{\infty\}$,
	\begin{align}
		\lim_{n\to\infty}\Cov \left\{\sqrt{n}\begin{pmatrix} \sqrt{a_{n,i^{\star}}^4} & 0 \\ 0 & \sqrt{ b_{n,j^{\star}}^{\beta_{j^{\star}}}} \end{pmatrix}^{-1} 
		 \begin{pmatrix}
			S_{n,m}(f_1) \\  S_{n,m}(f_2)
		\end{pmatrix} \right\}
		&=\begin{pmatrix}
			\gamma_m^2 & 0 \\ 0 & \zeta_m^2
		\end{pmatrix} . \label{eqn:m-cov-limit}
	\end{align}
	for some $\gamma_m^2, \zeta_m^2,$ where $\gamma_\infty^2=\gamma^2$ and $\zeta_\infty^2=\zeta_{f_2}^2$.
	Once \eqref{eqn:m-limit} and \eqref{eqn:m-cov-limit} are established, we may proceed as follows:
	Since $f_1$ and $f_2$ are bounded and since $\sqrt{n a_{n,i^{\star}}^4}\to \infty$, $n b_{n,j^{\star}}^{\beta_{j^{\star}}}\to\infty$, all summands vanish uniformly such that Lindeberg's condition holds.
	We employ a central limit theorem for triangular arrays of $m$-dependent random variables \citep{neumann2013} to obtain, for any finite $m\in\N$,
	\begin{align*}
		\sqrt{n}\begin{pmatrix} \sqrt{a_{n,i^{\star}}^4} & 0 \\ 0 & \sqrt{ b_{n,j^{\star}}^{\beta_{j^{\star}}}} \end{pmatrix}^{-1} 
		\begin{pmatrix}
			S_{n,m}(f_1) \\  S_{n,m}(f_2)
		\end{pmatrix}
		&\wconv \mathcal{N}\left(0,\begin{pmatrix}
			\gamma_m^2 & 0 \\ 0 & \zeta_m^2
		\end{pmatrix} \right),
	\end{align*}
	by virtue of \eqref{eqn:m-cov-limit}.
	Furthermore, \eqref{eqn:m-limit} implies that $\gamma_m^2\to \gamma^2$ and $\zeta_m^2\to \zeta^2$ as $m\to\infty$. 
	Hence, standard arguments \citep[Thm.\ 3.2]{Billingsley1999} yield the weak convergence claimed in the theorem.
	
	\underline{Ad \eqref{eqn:m-limit}:}
	For any $f\in\mathfrak{F}_0$, we have
	\begin{align*}
		&\E \left\| S_{n,m}(f)-
			S_{n}(f) \right\|^2 \\
		&\leq \frac{1}{n^2}\sum_{t=1}^n \sum_{z\in\Z} \left| \Cov \left( f(X_{t,n,m}) - f(X_{t,n}),f(X_{t+z,n,m}) - f(X_{t+z,n}) \right) \right| \\
		&\leq \frac{1}{n} \sum_{z=-L}^L \E \left| f(X_{1,n,m}) - f(X_{1,n}) \right|^2 \\
		&\quad + \frac{1}{n} \sum_{\substack{z\in\Z \\ |z|>L}} \left|\Cov  \left( f(X_{z,n,m}) - f(X_{z,n}),f(X_{0,n,m}) - f(X_{0,n}) \right) \right|,
	\end{align*}
	for any $L\in\N$.
	Using Lemma \ref{lem:autocov-decay} and Lemma \ref{lem:approximation}, we find that
	\begin{align}
		& \E \left\| S_{n,m}(f)-
			S_{n}(f) \right\|^2 \nonumber \\
		\begin{split}
			&\leq C L \frac{1}{n}a_{n,i^{\star}}^4 m^{1+2\delta_0} + C L \frac{1}{n} \sum_{j=1}^{q_2} m^{1+\delta_j\beta_j} b_{n,j}^{\beta_j} \\
			&\quad + C \frac{1}{n}a_{n,i^{\star}}^4\sum_{z=L+1}^\infty |z|^{1+2\delta_0}  + \sum_{j=1}^{q_2} C \frac{1}{n} \sum_{z=L+1}^\infty |z|^{1+\delta_j \beta_j}b_{n,j}^{\beta_j}.
		\end{split} \label{eqn:m-approx}
	\end{align}
	Since $\sum_{j=1}^{q_2} b_{n,j}^{\beta_j} \ll a_{n,i^{\star}}^4$ by assumption, and since $2+\delta_j\beta_j<0$ for all $j=1,\ldots,q_2$, we find that
	\begin{align*}
		&\limsup_{n\to\infty}\frac{n}{a_{n,i^{\star}}^4} \E \left\| S_{n,m}(f_1)-
			S_{n}(f_1) \right\|^2\\
		&\leq C L m^{1+2\delta_0} + C L^{2+2\delta_0}.
	\end{align*}
	The latter bound holds for arbitrary $L$, i.e.\ $C$ does not depend on $L$. 
	Since $2+2\delta_0<0$, we may choose $L=L_m\to\infty$ suitably to ensure that the latter term vanishes as $m\to\infty$.
	
	To handle $f_2$, note that our assumptions on $a_{n,i}$ and $b_{n,j}$ imply that 
	\begin{align*}
		\exp\left( -\frac{\eta^2\lambda^2}{2d \sigma^2 \sum_{i=1}^{q_1} a_{n,i}^2 }\right) \ll \sum_{j=1}^{q_2} b_{n,j}^{\beta_j}, \\
		\lim_{n\to\infty}\frac{1}{a_{n,i^{\star}}^2}\sum_{j=1}^{q_2} b_{n,j}^{\beta_j}&= 0 .
	\end{align*}
	Using Lemma \ref{lem:autocov-decay} and Lemma \ref{lem:approximation} again, we thus find that
	\begin{align*}
		&\frac{n}{b_{n,j^{\star}}^{\beta_{j^{\star}}}} \E \left\| S_{n,m}(f_2)-
			S_{n}(f_2) \right\|^2 \\
		&\leq C L m^{1+2\delta_0}\sum_{i=1}^{q_1}a_{n,i}^2 + C L \sum_{j=1}^{q_2} m^{1+\beta_j\delta_j} 
		+ C \sum_{z=L+1}^\infty |z|^{1+2\delta_0} \sum_{i=1}^{q_1} a_{n,i}^2 + C\sum_{j=1}^{q_2} \sum_{z=L+1}^\infty |z|^{1+\delta_j\beta_j},
	\end{align*}
	so that
	\begin{align*}
		\limsup_{n\to\infty} \frac{n}{b_{n,j^{\star}}^{\beta_{j^{\star}}}} \E \left\| S_{n,m}(f_2)-
			S_{n}(f_2) \right\|^2
		\leq C \sum_{j=1}^{q_2} \left( L m^{1+\beta_j\delta_j} + L^{2+\delta_j\beta_j}  \right).
	\end{align*}
	for any $L\in\N$.
	Upon choosing $L=L_m\to\infty$ suitably, we find that the latter term vanishes as $m\to\infty$.
	We have thus established \eqref{eqn:m-limit}.
	
	\underline{Ad \eqref{eqn:m-cov-limit}:}
	By virtue of Lemma \ref{lem:autocov-decay}, we have for $\delta^{\star}=\max(2\delta_0, \delta_1\beta_1,\ldots, \delta_{q_2} \beta_{q_2}) < -2$,
	\begin{align}
		\frac{1}{a_{n,i^{\star}}^4}\Cov\left( f_1(X_{t,n,m}), f_1(X_{t',n,m}) \right) 
		&\leq C |t-t'|^{1+\delta^{\star}},\label{eqn:majorant-1}
		\\
		\frac{1}{b_{n,j^{\star}}^{\beta_{j^{\star}}}}\Cov\left( f_2(X_{t,n,m}), f_2(X_{t',n,m}) \right) 
		&  \leq C |t-t'|^{1+\delta^{\star}}, \label{eqn:majorant-2}
		\\
		\frac{1}{a_{n,i^{\star}}^2 \sqrt{b_{n,j^{\star}}^{\beta_{j^{\star}}}}}\Cov\left( f_1(X_{t,n,m}), f_2(X_{t',n,m}) \right)
		&\leq C |t-t'|^{1+\delta^{\star}} \frac{ \sqrt{b_{n,j^{\star}}^{\beta_{j^{\star}}}} }{a_{n,i^{\star}}^2}. \nonumber
	\end{align}
	Hence,
	\begin{align*}
		\lim_{n\to\infty}  \Cov\left( \frac{\sqrt{n}}{ a_{n,i^{\star}}^2} S_{n,m}(f_1), \frac{\sqrt{n}}{ \sqrt{b_{n,j^{\star}}^{\beta_{j^{\star}}}}} S_{n,m}(f_2) \right) = 0.
	\end{align*}
	Moreover, Lemma \ref{lem:autocov-limit} yields
	\begin{align*}
		\frac{1}{a_{n,i^{\star}}^4}\Cov\left( f_1(X_{t,n,m}), f_1(X_{t',n,m}) \right) 
		&= \frac{1}{a_{n,i^{\star}}^4}\Cov\left( f_1(X_{t-t',n,m}), f_1(X_{0,n,m}) \right) \\
		&= \frac{1}{4} \sum_{l,l',r,r'=1}^d \left[ D^2_{ll'}f_1(0) D^2_{rr'}f_1(0) \right]  \Sigma^{i,m,t-t'}_{lr}\Sigma^{i,m,t-t'}_{l'r'} + o(1),
	\end{align*}
	and
	\begin{align*}
		\frac{1}{b_{n,j^{\star}}^{\beta_{j^{\star}}}}\Cov\left( f_2(X_{t,n,m}), f_2(X_{t',n,m}) \right)
		&= \frac{1}{b_{n,j^{\star}}^{\beta_{j^*}}}\Cov\left( f_2(X_{t-t',n,m}), f_2(X_{0,n,m}) \right) \\
		&=  \overline{f}_2^{[\nu_{t-t',j^{\star},m}]}(0) + o(1).
	\end{align*}
	We may utilize the upper bounds \eqref{eqn:majorant-1} and \eqref{eqn:majorant-2} to apply the dominated convergence theorem, and obtain
	\begin{align}
		\begin{split}
			\lim_{n\to\infty} \frac{n}{ a_{n,i^{\star}}^4} \Var\left( S_{n,m}(f_1)\right) 
			&=\gamma_m^2 = \sum_{z\in \Z} \sum_{l,l',r,r'=1}^d \left[ D^2_{ll'}f_1(0) D^2_{rr'}f_1(0) \right]  \Sigma^{i,m,z}_{lr}\Sigma^{i,m,t-t'}_{l'r'}, \\
			\lim_{n\to\infty} \frac{n}{b_{n,j^{\star}}^{\beta_{j^{\star}}} } \Var\left( S_{n,m}(f_2)\right) 
			&=\zeta_m^2 = \sum_{z\in \Z} \overline{f}_2^{[\nu_{z,j^{\star},m}]}(0).
		\end{split} \label{eqn:limit-variance}
	\end{align}
	The upper bounds \eqref{eqn:majorant-1} and \eqref{eqn:majorant-2} in particular ensure that $\gamma_m^2$ and $\zeta_m^2$ are finite.
	Note that \eqref{eqn:limit-variance} also holds for $m=\infty$, and thus establishes \eqref{eqn:m-cov-limit}, completing the proof of claim (i).
	
	\underline{Cases (i)-(iv):}
	The proof is largely analogous to case (v).
	Leveraging \eqref{eqn:m-approx}, we find that
	\begin{align}
		\limsup_{n\to\infty} \frac{n}{(a_{n,i^{\star}}^4 \vee b_{n,j^{\star}}^{\beta_{j^{\star}}})} \E \left\| S_{n,m}(f) - S_{n}(f) \right\| \leq C L m^{1+\delta_{j^{\star}}\beta_{j^{\star}}} + C L^{2+\delta_{j^{\star}}\beta_{j^{\star}}}. \label{eqn:m-approx-2}
	\end{align}
	By choosing $L=L_m\to\infty$ suitably, this term tends to zero as $m\to\infty$.
	Regarding the variances, Lemma \ref{lem:autocov-decay} yields
	\begin{align*}
		\frac{1}{(a_{n,i^{\star}}^4 \vee b_{n,j^{\star}}^{\beta_{j^{\star}}})}\Cov\left( f(X_{t,n,m}), f(X_{t',n,m}) \right) 
		&  \leq C |t-t'|^{1+\delta^{\star}}
	\end{align*}
	for $\delta^{\star}=\max(2\delta_0, \delta_1\beta_1,\ldots, \delta_{q_2} \beta_{q_2}) < -2$.
	
	For cases (i) and (ii), we have	$X_{t,n,m}\pconv a\mathcal{G}_{t,i^{\star},m}$, and $X_{t,n,m}\pconv b \mathcal{J}_{t,j^{\star},m}$, respectively.
	Since $f$ is bounded an continuous, we find that
	\begin{align*}
		\lim_{n\to\infty} \Cov(f(X_{t+z,n,m}), f(X_{t,n,m})) = 
		\begin{cases}
			\Cov[ f(a\mathcal{G}_{z,i^{\star},m}), f(a\mathcal{G}_{0,i^{\star},m})], & \text{case (i)}, \\
			\Cov[ f(b\mathcal{J}_{z,j^{\star},m}), f(b\mathcal{J}_{0,j^{\star},m})], & \text{case (ii)}.
		\end{cases}
	\end{align*}
	This convergence holds for all $m\in\N\cup\{\infty\}$, and all $z\in\Z$.
	
	For case (iii), we may apply Lemma \ref{lem:autocov-limit} to obtain, for any $m\in\N\cup\{\infty\}$,
	\begin{align*}
		\frac{1}{b_{n,j^{\star}}^{\beta_{j^*}}}\Cov\left( f(X_{t,n,m}), f(X_{t',n,m}) \right)
		&= \frac{1}{b_{n,j^{\star}}^{\beta_{j^{\star}}}}\Cov\left( f(X_{t-t',n,m}), f(X_{0,n,m}) \right) \\
		&=  \overline{f}^{[\nu_{t-t',j^{\star},m}]}(0) + o(1).
	\end{align*}
	
	For case (iv), Lemma \ref{lem:autocov-limit} yields
	\begin{align*}
		&\quad \frac{1}{a_{n,i^{\star}}^{4}}\Cov\left( f(X_{t,n,m}), f(X_{t',n,m}) \right) \\
		& = \frac{1}{4}\sum_{l,l',r,r'=1}^d \left[D^2_{ll'}f(0) D^2_{rr'}f(0)\right] \left[ \Sigma^{i^{\star},\infty,t-t'}_{lr}\Sigma^{i^{\star},\infty,t-t'}_{l'r'}\right] + o(1).
	\end{align*}
	
	Thus, in all cases (i)-(iv), the dominated convergence theorem yields
	\begin{align*}
		\lim_{n\to\infty} \frac{n}{a_{n,i^{\star}}^4 \vee b_{n,j^{\star}}^{\beta_{j^{\star}}} } \Var\left( S_{n,m}(f_1) \right) = \xi^2_m,
	\end{align*}
	for some $\xi^2_m$ with $\lim_{m\to\infty} \xi^2_m = \xi^2$ as claimed in the Theorem.
	Since $n (a_{n,i^{\star}}^4 \vee b_{n,j^{\star}}^{\beta_{j^{\star}}})\to\infty$, we obtain
	\begin{align*}
		\frac{\sqrt{n}}{\sqrt{a_{n,i^{\star}}^4 \vee b_{n,j^{\star}}^{\beta_{j^{\star}}} }}S_{n,m}(f_1) \quad \wconv\quad \mathcal{N}(0,\xi^2_m).
	\end{align*}
	Together with \eqref{eqn:m-approx-2}, this allows us to conclude, just as in case (v), that 
	\begin{align*}
		\frac{\sqrt{n}}{\sqrt{a_{n,i^{\star}}^4 \vee b_{n,j^{\star}}^{\beta_{j^*}} }}S_{n}(f_1) \quad \wconv\quad \mathcal{N}(0,\xi^2),
	\end{align*}	
	completing the proof.
\end{proof}

\subsection{Adaptive estimation equation} \label{sec5.5}

\subsubsection{Gradients}

Define the matrix
\begin{align*}
	C^{j}_n(\theta) &= \begin{pmatrix}
		1 & -\widetilde{b}_j\beta_j \log(\Delta_n)\mathds{1}_{j\neq 1} & \widetilde{b}_j [ -\log|w_n| + H_1 \log(\Delta_n)\mathds{1}_{j\neq 1}]  \\
		0 & 1            & -H_j/\beta_j\\
		0 & 0 			 &  1
	\end{pmatrix} \\
	&\qquad \cdot  w_n^{-\beta_j(\theta)} \Delta_n^{\beta_j(\theta) (H_1(\theta)-H_j(\theta))} \in \R^{3\times 3} \\
	C_n(\theta)
	&= \diag(C^{1}_n,\ldots, C^{q}_n) \in\R^{(3q) \times (3q)}
\end{align*}
We also introduce the vectors $W=W(f,\theta,\lambda,\gamma,w)\in\R^{3q}$ as
\begin{align}
	W(f,\theta,\lambda,\gamma,w)_i  
	&= \int \widehat{f}(v) \exp\left(-\widetilde{b}_1|w \lambda v|^{\beta_1} \gamma^{\beta_1H_1} \right)\left[\partial_{\theta_i} \sum_{j=1}^q \widetilde{b}_j \gamma^{\beta_j H_j} |\lambda v|^{\beta_j}\right] \, dv. \label{eqn:def-W}
\end{align}
Recall that we may formally define $\E_{\theta} f(\lambda u_n(\theta)X_{1,n,\gamma})$ via \eqref{eqn:expect-fourier} for all $\theta\in\widetilde\Theta \supset \Theta$, where
\begin{align*}
    \widetilde{\Theta} = \left(\,(0,\infty) \times (0,1) \times (0,\infty) \,\right)^q
\end{align*}
Using this convention, we first obtain the following result on the non-random part.

\begin{lemma}\label{lem:moment-gradient-adaptive}
	Let $f:\R\to\R$ be a Schwartz function.
	Let $\mathbf{K}\subset\widetilde{\Theta}$ be a compact set, and let $u_n(\theta) = \Delta_n^{-H_1(\theta)} w_n$ for some sequence $w_n\to w\in[0,\infty)$.
	Then
	\begin{align*}
		\sup_{\theta\in \mathbf{K}} \left\| D \E_\theta f(\lambda u_n(\theta) X_{1,n,\gamma})C_n(\theta)   - W(f,\theta,\lambda,\gamma,w)\right\| \to 0, 
	\end{align*}
\end{lemma}
\begin{proof}[Proof of Lemma \ref{lem:moment-gradient-adaptive}]
	Since $f$ is a Schwartz function, the Fourier transform $\widehat{f}(v) = \frac{1}{2\pi}\int \exp(-ivx)f(x)\, dx$ is a Schwartz function as well, and $f(x) = \int \exp(ivx)\, \widehat{f}(v)\, dv$.
	Hence,
	\begin{align*}
		\E_\theta f(\lambda X_{1,n,\gamma}) 
		&= \int \E_\theta \exp(-iv \lambda X_{l,n,\gamma}) \widehat{f}(v)\, dv \\
		&= \int \widehat{f}(v) \exp(-\psi_n(\lambda v,\theta,\gamma))\, dv.
	\end{align*}
	For $j\neq 1$, this yields,
	\begin{align*}
		&\begin{pmatrix}
			\partial_{\widetilde{b}_j} & \partial_{H_j} & \partial_{\beta_j}
		\end{pmatrix}
		\E_\theta f(\lambda u_n(\theta) X_{1,n,\gamma}) \\
		&= \int \widehat{f}(v) \exp(-\psi_n(v\,\lambda u_n(\theta),\theta,\gamma)) |w_n \lambda v|^{\beta_j} \Delta_n^{\beta_j(H_j-H_1)} \gamma^{\beta_j H_j} \widetilde{b}_j \\
		&\qquad \quad 
		\begin{pmatrix}
			1/\widetilde{b}_j \\ 
			\beta_j [\log(\gamma) + \log(\Delta_n)] \\
			H_j \log(\gamma) + (H_j-H_1)\log(\Delta_n) + \log|w_n v\lambda|
		\end{pmatrix}^T\, dv,
	\end{align*}
	so that
	\begin{align}
		&\begin{pmatrix}
			\partial_{\widetilde{b}_j} & \partial_{H_j} & \partial_{\beta_j}
		\end{pmatrix}
		\E_\theta f(\lambda u_n(\theta) X_{1,n,\gamma}) C_n^j(\theta) \nonumber \\
		&= \widetilde{b}_j |\lambda|^{\beta_j} \gamma^{\beta_j H_j}  \int \widehat{f}(v) \exp(-\psi_n(v\,\lambda u_n(\theta),\theta,\gamma)) |v|^{\beta_j} 
		\begin{pmatrix}
			1/\widetilde{b}_j \\ 
			\beta_j \log(\gamma) \\
			\log|\lambda v|
		\end{pmatrix}^T\, dv. \label{eqn:grad-2}
	\end{align}
	For $j=1$, the partial derivatives $\partial_{\widetilde{b}_1}$ and $\partial_{\beta_1}$ take the same form, but
	\begin{align*}
		&\partial_{H_1}\E_\theta f(\lambda u_n(\theta) X_{1,n,\gamma}) \\
		&=  \int \widehat{f}(v) \exp(-\psi_n(v\,\lambda u_n(\theta),\theta,\gamma)) |w_n \lambda v|^{\beta_1}  \gamma^{\beta_1 H_1} \widetilde{b}_1 [\beta_1 \log \gamma] \\
		&-  \int \widehat{f}(v) \exp(-\psi_n(v\,\lambda u_n(\theta),\theta,\gamma)) \sum_{i=2}^n |w_n \lambda v|^{\beta_i}  \gamma^{\beta_i H_i} \widetilde{b}_i  \Delta_n^{\beta_i(H_i-H_1)} [\beta_i \log \Delta_n].
	\end{align*}
	The second term appears because the parameter $H_1$ affects all components via the scaling factor. 
	However, this second term is asymptotically negligible because $H_i>H_1$ for $i\geq 2$.
	Thus
	\begin{align*}
		&\begin{pmatrix}
			\partial_{\widetilde{b}_1} & \partial_{H_1} & \partial_{\beta_1}
		\end{pmatrix}
		\E_\theta f(\lambda u_n(\theta) X_{1,n,\gamma}) C_n^1(\theta)
		\\
		&= \int \widehat{f}(v) \exp(-\psi_n(v\,\lambda  u_n(\theta),\theta,\gamma)) |\lambda v|^{\beta_1} \gamma^{\beta_1 H_1} \widetilde{b}_1 \\
		&\qquad \quad 
		\begin{pmatrix}
			1/\widetilde{b}_1 \\ 
			\beta_1 \log(\gamma) \\
			\log|\lambda v|
		\end{pmatrix}^T\, dv + \mathcal{O}\left( |\log \Delta_n| \max_{i\geq 2} \Delta_n^{\beta_i (H_i-H_1)} |w_n|^{\beta_i-\beta_1} \right)
	\end{align*}
	where the implicit constant in the $\mathcal{O}(\ldots)$ term is bounded on compacts in $\Theta$. 
	Since $H_i>H_1$ for all parameters in the open parameter set $\widetilde{\Theta}$, the latter term vanishes uniformly on compacts $\mathbf{K}\subset\widetilde{\Theta}$.	

	Recalling $w=\lim_n w_n$, we have for any $K>0$ and any compact $\mathbf{K}\subset\widetilde{\Theta}$,
	\begin{align*}
		\sup_{\theta\in \mathbf{K}} \sup_{|v|\leq K} \left|\psi_n(\lambda u_n(\theta)v, \theta,\gamma) - (\widetilde{b}_1 |wv\lambda|^{\beta_1} \gamma^{\beta_1 H_1})\right|\to 0.
	\end{align*}
	Thus, \eqref{eqn:grad-2} converges, locally uniformly in $\widetilde{\Theta}$, towards the limit 
	\begin{align}
		&\begin{pmatrix}
			W_{3j-2}(f,\theta,\lambda,\gamma,w) \\ 
			W_{3j-1}(f,\theta,\lambda,\gamma,w) \\ 
			W_{3j}(f,\theta,\lambda,\gamma,w)
		\end{pmatrix} \nonumber \\
		&=
		\widetilde{b}_j |\lambda|^{\beta_j} \gamma^{\beta_j H_j}  \int \widehat{f}(v) |v|^{\beta_j} \exp\left(-\widetilde{b}_1 |w v\lambda|^{\beta_1} \gamma^{\beta_1 H_1}\right)
		\begin{pmatrix}
			1/\widetilde{b}_j \\ 
			\beta_j \log(\gamma) \\
			\log|\lambda v|
		\end{pmatrix}^T\, dv.  \label{eqn:gradlim-2}
	\end{align}
	This completes the proof.
\end{proof}

\begin{lemma}\label{lem:lipschitz}
	Let $f\in C^1(\R;\R)$ be such that $\sup_{x\in\R} |x f'(x)|<\infty$. 
	Then for any $a>0$ there exists some $L(a)>0$ such that for any $x\in\R$, the mapping $\widetilde{f}:(a,\infty)\to \R, c\mapsto f(cx)$ is $L(a)$-Lipschitz continuous. 
\end{lemma}
\begin{proof}[Proof of Lemma \ref{lem:lipschitz}]
	The mapping $\widetilde{f}$ is differentiable with $\widetilde{f}'(c) = xf'(cx)$, such that
	\begin{align*}
	    \sup_{c\geq a} |\widetilde{f}'(c)| 
	    &= \sup_{c\geq a} \frac{1}{c} \sup_{x\in\R} |cx f'(cx)| \\
	    &\leq \frac{1}{a} \sup_{x\in\R} |x f'(x)| \quad =L(a).
	\end{align*}
\end{proof}

\begin{lemma}\label{lem:lipschitz-k}
	Let $f:\R\to\R$ be a Schwartz function, and denote $\psi_{c_1,c_2}(x) = \frac{f(c_1 x)-f(c_2 x)}{|c_1-c_2|}$ for any $c_1,c_2>0$.
	Then for any $k=0,1,2,\ldots$, and for any $0<a < b <\infty$, there exists some $C_k=C_k(f,a,b)<\infty$ such that
	\begin{align*}
		\sup_{c_1, c_2 \in [a,b]} \| \psi^{(k)}_{c_1,c_2}\|_\infty \quad \leq C_k(f,a,b)).
	\end{align*}
\end{lemma}
\begin{proof}[Proof of Lemma \ref{lem:lipschitz-k}]
	Since $f$ is a Schwartz function, we have $\sup_x |x f'(x)|<\infty$.
	Hence, Lemma \ref{lem:lipschitz} yields the claim for $k=0$.
	We now proceed by induction. 
	Obvserve that
	\begin{align*}
		\psi'_{c_1,c_2}(x) = c_1\frac{f'(c_1 x)-f'(c_2 x)}{|c_1-c_2|} + \frac{c_1-c_2}{|c_1-c_2|} f'(c_2x).
	\end{align*}
	Note that $f'$ is a Schwartz function as well, so that we obtain by induction
	\begin{align*}
		\left|\psi^{(k+1)}_{c_1,c_2}(x)\right| 
		&\leq 
		c_1 \left| D^k \frac{f'(c_1 x)-f'(c_2 x)}{|c_1-c_2|} \right| 
		\quad + c_2^k |(\tfrac{d}{dx})^k f(c_2x)|  \\
		&\leq c_1 C_k(f', a, b) + c_2^k \| f^{(k+1)}\|_\infty \quad =: C_{k+1}(f, a,b).
	\end{align*}
\end{proof}

To handle the gradient w.r.t.\ $\theta$ of the stochastic term $\frac{1}{n} \sum_{l=1}^n f(\lambda_{r,n}(\theta)X_{l,n,\gamma_r})$, we will use the following technical lemma.

\begin{lemma}\label{lem:moment-gradient-adaptive-stochastic}
	Let $f:\R\to\R$ be an even Schwartz function, and suppose that the order of differencing $k$ satisfies.
	\begin{align*}
		k > H_j +\frac{1}{\beta_j}, \qquad j=1,\ldots, q.
	\end{align*}
	Let $u_n(\theta) = \Delta_n^{-H_1(\theta)} w_n$ for some sequence $w_n\to w\in[0,\infty)$.
	Then for any sequence $r_n\ll 1/|\log \Delta_n|$, as $n\to\infty$, $\Delta_n\to 0$,
	\begin{align*}
		\sup_{\theta \in \mathbf{B}_{r_n}(\theta_0)} \left\| D \left[ \frac{1}{n} \sum_{l=1}^n f(\lambda u_n(\theta)X_{l,n,\gamma})\right] C_n(\theta) \right\| = \mathcal{O}_{\mathbb{P}}\left(\frac{|\log \Delta_n|}{\sqrt{n}}  \sqrt{ \sum_{j=1}^q w_n^{\beta_j} \Delta_n^{\beta_j(H_j-H_1)} } \right),
	\end{align*}
	where convergence in probability holds under the probability measure induced by $\theta_0$.
\end{lemma}
\begin{proof}[Proof of Lemma \ref{lem:moment-gradient-adaptive-stochastic}]
	As $u_n(\theta) = w_n \Delta_n^{-H_1(\theta)}$, we only need to consider the derivative w.r.t.\ $H_1$, because all other derivatives are zero.
	Hence,
	\begin{align}
		&\left\| D \left[ \frac{1}{n} \sum_{l=1}^n f(\lambda u_n(\theta)X_{l,n,\gamma})\right] C_n(\theta) \right\| \nonumber \\
		&= \left|\log(\Delta_n) \frac{1}{n} \sum_{l=1}^n \lambda u_n(\theta) X_{l,n,\gamma} f'(\lambda u_n(\theta) X_{l,n,\gamma}) \right| \nonumber \\
		&= |\log(\Delta_n)|\; \left| \frac{1}{n} \sum_{l=1}^n \varphi(\lambda u_n(\theta) X_{l,n,\gamma}) \right|, \label{eqn:gradient-adaptive-1}
	\end{align} 
	for $\varphi(x) = x f'(x)$.  
	Moreover, if $\|\theta-\theta_0\|\leq r_n \leq A/|\log \Delta_n| $, then  $|\Delta_n^{-H_1(\theta)}/\Delta_n^{-H_1(\theta_0)}-1| \leq \exp(A)\, r_n |\log \Delta_n| $.
	Since $r_n \ll 1/|\log \Delta_n|$, we find that for some large $C$,
	\begin{align*}
		\sup_{\theta \in \mathbf{B}_{r_n}(\theta_0)} \left| \frac{1}{n} \sum_{l=1}^n \varphi(\lambda u_n(\theta) X_{l,n,\gamma}) \right| 
		&\leq  \sup_{c: |c-1|\leq C r_n |\log \Delta_n|} \left| \frac{1}{n} \sum_{l=1}^n \varphi(c \lambda u_n(\theta_0) X_{l,n,\gamma}) \right| \\
		&\leq \sup_{c \in [\frac{1}{2}, \frac{3}{2}]} \left| \frac{1}{n} \sum_{l=1}^n \varphi(c \lambda u_n(\theta_0) X_{l,n,\gamma}) \right|,
	\end{align*}
	where the latter inequality holds for $n$ large enough.
	
	Since $f$ is a Schwartz function, so is $\varphi$.
	By virtue of Lemma \ref{lem:lipschitz-k}, there exists some $M>0$ such that the function $\psi(x) = \frac{\varphi(c_1 x) - \varphi(c_2x)}{M\,|c_1-c_2|}$ belongs to the class $\mathfrak{F}_0$ for all $c_1,c_2 \in [\frac{1}{2}, \frac{3}{2}]$. 
	Hence, Theorem \ref{thm:variance-bound} yields
	\begin{align*}
		\Var \left[ \frac{1}{n} \sum_{l=1}^n \left[ \varphi(c_1 \lambda u_n(\theta_0) X_{l,n,\gamma}) -  \varphi(c_2 \lambda u_n(\theta_0) X_{l,n,\gamma}) \right] \right] 
		&\leq C M^2 |c_1-c_2|^2 \frac{1}{n} \left( \sum_{j=1}^q w_n^{\beta_j} \Delta_n^{\beta_j(H_j-H_1)}\right), \\
		\Var \left[ \frac{1}{n} \sum_{l=1}^n \varphi(\lambda u_n(\theta_0) X_{l,n,\gamma})  \right]
		&\leq \frac{C}{n} \left( \sum_{j=1}^q w_n^{\beta_j} \Delta_n^{\beta_j(H_j-H_1)}\right).
	\end{align*}
	Moreover, $\E \varphi(c \lambda u_n(\theta_0) X_{l,n,\gamma})=0$ because $\varphi$ is a bounded odd function, and $X_{l,n,\gamma}$ is a symmetric random variable.
	This allows to bound
	\begin{align*}
		&\left\| \sup_{c \in [\frac{1}{2}, \frac{3}{2}]} \left| \frac{1}{n} \sum_{l=1}^n \varphi(c\lambda \Delta_n^{-H_1(\theta_0)} X_{l,n,\gamma}) \right| \right\|_{L_2} \\
		&\leq \left\| \frac{1}{n} \sum_{l=1}^n \varphi(\lambda \Delta_n^{-H_1(\theta_0)} X_{l,n,\gamma}) \right\|_{L_2} \\
		&\quad + \sup_{\frac{1}{2} = c_0 < c_1 < \ldots < c_{w+1}=\frac{3}{2};\, w\in\N} \sum_{i=0}^w \left\| \frac{1}{n} \sum_{l=1}^n \varphi(c_{i+1}\lambda \Delta_n^{-H_1(\theta_0)} X_{l,n,\gamma}) - \varphi(c_{i}\lambda \Delta_n^{-H_1(\theta_0)} X_{l,n,\gamma})  \right\|_{L_2} \\
		&\leq \frac{C}{\sqrt{n}} \sqrt{ \sum_{j=1}^q w_n^{\beta_j} \Delta_n^{\beta_j(H_j-H_1)} }
		\quad = \frac{C}{\sqrt{n}} \sqrt{ \sum_{j=1}^q w_n^{\beta_j(\theta_0)} \Delta_n^{\beta_j(\theta_0)\cdot(H_j(\theta_0)-H_1(\theta_0))} }.
	\end{align*}
	Note that the latter steps may also be regarded as a chaining argument w.r.t.\ the $L_2$-norm, using the metric total variation as an upper bound on the covering numbers.	
\end{proof}

\begin{cor}\label{cor:gradient}
	Let $f$ be an even Schwartz function, and denote 
	\begin{align*}
		F_n(\theta) = \frac{1}{n} \sum_{l=1}^n \left[f(\lambda u_n(\theta)X_{l,n,\gamma}) - \E_\theta f(\lambda u_n(\theta) X_{l,n,\gamma}) \right].
	\end{align*}
	Suppose that the order of differencing $k$ satisfies
	\begin{align*}
		k > H_j +\frac{1}{\beta_j}, \qquad j=1,\ldots, q.
	\end{align*}
	If $w_n\to w \in [0,\infty)$, then for any sequence $r_n\ll 1/|\log \Delta_n|$, 
	\begin{align*}
		\sup_{\theta \in \mathbf{B}_{r_n}(\theta_0)} \left\| D F_n(\theta) C_n(\theta) - W(f,\theta,\lambda,\gamma,w) \right\| = o_p(1),
	\end{align*}
	where convergence in probability holds under the probability measure induced by $\theta_0$.
\end{cor}

\subsubsection{Limit theory}

\begin{lemma}\label{lem:clt-G}
    Let the conditions of Theorem \ref{thm:clt-GMM} hold. 
    Then
    \begin{align*}
        \sqrt{n} \mathcal{G}_n(\theta_0) \wconv \mathcal{N}(0,\widetilde{\Sigma}),
    \end{align*}
    where $\widetilde{\Sigma}\in\R^{3q\times 3q}$ is given by \eqref{eqn:asymp-cov-G} below.
\end{lemma}
\begin{proof}[Proof of Lemma \ref{lem:clt-G}]
    We first show how the statistic $\mathcal{G}_n(\theta_0)$ fits into the framework of Theorem \ref{thm:clt}.
    Let $d=3q$, and define the kernel
    \begin{align*}
        \widetilde{g}_{j}(s) 
        &= \left( \sum_{v=0}^k \binom{k}{v} (-1)^v (\gamma_r v-s)_+^{H_j - \frac{1}{\beta_j}} \right)_{r=1}^d,
    \end{align*}
    such that
    \begin{align*}
        \widetilde{X}_{l,n} 
        = \left(\Delta_n^{-H_1} X_{l,n,\gamma_r}\right)_{r=1}^d 
        &\deq \sum_{j=1}^q \Delta_n^{H_j-H_1} b_j \int_{-\infty}^l \widetilde{g}_j(s-l)\, dZ_s^{\beta_j}.
    \end{align*}
    If $\beta_1=2$, this is the regime of Theorem \ref{thm:clt}(i), and if $\beta_1<2$, this is the regime of Theorem \ref{thm:clt}(ii).
    In both cases, we find for any function $\varphi\in\mathfrak{F}_0$ that 
    \begin{align}
        \frac{1}{\sqrt{n}} S_n(\varphi) 
        = \frac{1}{\sqrt{n}} \sum_{l=1}^n \left[ \varphi(\widetilde{X}_{l,n}) - \E\varphi(\widetilde{X}_{l,n}) \right]
        \wconv \mathcal{N}(0,\xi^2), \label{eqn:clt-moments-G-univariate}
    \end{align}
    for 
    \begin{align}
        \xi^2 
        &= \sum_{z\in \Z} \Cov\left[ \varphi( (b_1 Y_{z,\gamma_r}^{H_1,\beta_1})_{r=1}^d),\; \varphi( (b_1Y_{0,\gamma_r}^{H_1,\beta_1})_{r=1}^d) \right], \nonumber \\
        Y_{l,\gamma}^{H,\beta}
        &= \sum_{v=0}^k (-1)^v \binom{k}{v} Y_{l-v\gamma}^{H,\beta}. \label{eqn:lfsm-kdiff}
    \end{align}
    Note that the factor $b_1^{-\frac{\beta_1}{2}}$ of formula \eqref{eqn:asymp-var}, the factor $b_1^{-2}$ for the case $\beta_1=2$, are not present, because here we standardize by $\sqrt{n}$ only.
    Now let 
    \begin{align*}
        \widetilde{f}:\R^d\to\R^d, x=(x_1,\ldots, x_d)\mapsto \left[ \, f(\lambda_r x_r)\, \right]_{r=1}^d,
    \end{align*}
    such that 
    \begin{align*}
        S_n(f; (\lambda_r,\gamma_r)_{r=1}^d) 
        &= \frac{1}{n} \sum_{l=1}^n \left[ \widetilde{f}\left( \widetilde{X}_{l,n} \right) - \E \widetilde{f}\left( \widetilde{X}_{l,n} \right) \right].
    \end{align*}
    Applying the Cramer-Wold device with \eqref{eqn:clt-moments-G-univariate}, we find that
    \begin{align*}
        \sqrt{n} \mathcal{G}_n(\theta_0) 
        = \sqrt{n} S_n(f; (\lambda_r,\gamma_r)_{r=1}^d) 
        \;\wconv\; \mathcal{N}(0,\widetilde{\Sigma}),
    \end{align*}
    for some asymptotic covariance matrix $\widetilde{\Sigma}\in\R^{d\times d}$, given by
    \begin{align}
        \widetilde{\Sigma}_{r, r'} 
        &= \sum_{z\in\Z} \Cov\left[ f(\lambda_r Y_{z,\gamma_r}^{H_1,\beta_1}),\; f(\lambda_{r'} Y_{0,\gamma_{r'}}^{H_1,\beta_1}) \right], \qquad r,r'=1,\ldots, d. \label{eqn:asymp-cov-G}
    \end{align}
\end{proof}

\begin{proof}[Proof of Theorem \ref{thm:clt-GMM}]
	By virtue of Lemma \ref{lem:clt-G}, the estimating equation \eqref{eqn:ee-schwartz} satisfies condition (E.1) with $A_n = \sqrt{n}$.
	Now denote 
	\begin{align*}
		s_n = \frac{\Delta_n^{\beta_j(\theta_0)[H_1(\theta_0)-H_j(\theta_0)]}}{\sqrt{n}}.
	\end{align*}
	and define $r_n = \sqrt{\frac{s_n}{|\log \Delta_n|^2}} \ll \frac{1}{|\log \Delta_n|^2}$.
	We show that conditions (E.2) and (E.3) hold as well, with $B_n(\theta)=1$, $C_n(\theta)$, and $\overline{W}(\theta)$, and $r_n$ as above.
	Regarding (E.3), we have
	\begin{align*}
		&\sup_{\theta\in\mathbf{B}_{r_n}(\theta_0)}  \|C_n(\theta)C_n(\theta_0)^{-1}-I\| \\
		&\leq \sum_{j=1}^q \sup_{\theta\in\mathbf{B}_{r_n}(\theta_0)}  \|C^j_n(\theta)-C^j_n(\theta_0)\|\; \|C_n(\theta_0)^{-1}\| 
		= \mathcal{O}\left(r_n |\log \Delta_n|^2\right),
	\end{align*}
	which tends to zero by our choice of $r_n$.
	Note that Corollary \ref{cor:gradient} only requires $r_n\ll 1/|\log \Delta_n|$, but here we need the stronger condition on $r_n$.
	
	Regarding (E.2), the first convergence is established by Corollary \ref{cor:gradient}, and the second convergence holds because
	\begin{align*}
		\sup_{\theta \in \mathbf{B}_{r_n}(\theta_0)}  \|C_n(\theta)\| \, \|A_n(\theta_0)^{-1}\| 
		&\leq C \sup_{\theta \in \mathbf{B}_{r_n}(\theta_0)} \max_j \frac{\Delta_n^{\beta_j(\theta)(H_1(\theta)-H_j(\theta))}}{\sqrt{n}} \\
		&\leq C \sup_{\theta \in \mathbf{B}_{r_n}(\theta_0)} \max_j \frac{\Delta_n^{\beta_j(\theta_0)(H_1(\theta_0)-H_j(\theta_0))}}{\sqrt{n}},
	\end{align*}
	 using $r_n \ll 1/|\log \Delta_n|$ in the last step.
	Note that the latter term tends to zero by assumption.
	Thus, we may apply Theorem \ref{thm:ee-consistency} to obtain the desired asymptotic properties of the estimating equation, with asymptotic variance
	\begin{align*}
		\Sigma = \overline{W}(\theta_0)^{-1} \widetilde{\Sigma}\,(\overline{W}(\theta_0)^{-1})^T.
	\end{align*}
\end{proof}

\subsection{Smooth thresholds} \label{sec5.6}

Let $u_n(\theta) = \Delta_n^{-H_1(\theta)} w_n$ for $w_n \to 0$ satisfying \eqref{eqn:wn}.
Introduce the scaling terms $d_n = d_n(\theta) = \max_{j=1,2} w_n^{\frac{\beta_j}{2}}\Delta_n^{\frac{\beta_j(\overline{H}_j-H_1)}{2}} \leq \Delta_n^{\frac{\beta_1(\overline{H}_1-H_1)}{2}} \to 0$.
Define the matrices
\begin{align*}
	\underline{C}_n^j(\theta) &= \begin{pmatrix}
		1 & -2\widetilde{a}_j \log(\Delta_n)\mathds{1}_{\{j\neq 1\}}  \\
		0 & 1 
	\end{pmatrix} w_n^{-2} \Delta_n^{2 (H_1(\theta)-H_j(\theta))} \in \R^{2\times 2} \\
	\doubleunderline{C}^{j}_n(\theta) &= \begin{pmatrix}
		1 & 0 & -\widetilde{b}_j \log|w_n|   \\
		0 & 1 & -\overline{H}_j/\beta_j\\
		0 & 0 &  1
	\end{pmatrix} w_n^{-\beta_j(\theta)} \Delta_n^{\beta_j(\theta) (H_1(\theta)-\overline{H}_j(\theta))} d_n \in \R^{3\times 3} \\
	\underline{C}_n(\theta)&= \diag(\underline{C}^{1}_n,\underline{C}^{2}_n, \doubleunderline{C}^{1}_n, \doubleunderline{C}^{2}_n) \in\R^{10 \times 10}, \\
	B_n(\theta) &= \diag(1,1,1,1,d_n^{-1},d_n^{-1},d_n^{-1},d_n^{-1},d_n^{-1},d_n^{-1},)\in\R^{10\times 10},\\
	\Lambda_n(\theta) 
		&= \frac{1}{\sqrt{n}}\,\diag \Big((u_n(\theta) \Delta_n^{H_1(\theta)})^2, \ldots, (u_n(\theta) \Delta_n^{H_1(\theta)})^2, \\
		&\qquad\qquad (u_n(\theta) \Delta_n^{\overline{H}_1(\theta)})^{\frac{\beta_1(\theta)}{2}},\ldots, (u_n(\theta) \Delta_n^{\overline{H}_1(\theta)})^{\frac{\beta_1(\theta)}{2}} \Big) \in \R^{10\times 10}.
\end{align*}

\begin{lemma}\label{lem:ABC-rate}
	Suppose that $w_n\to 0$ satisfies \eqref{eqn:wn}, and $r_n \ll 1 / |\log \Delta_n|^2$ such that for the true parameter $\theta=\theta_0$, it holds for some $\epsilon>0$
	\begin{align*}
		\Delta_n^{2(H_1(\theta_0)-H_2(\theta_0))} 
		&\ll \sqrt{n} \, r_n\, \Delta_n^\epsilon, \\
		\Delta_n^{ \frac{\beta_1(\theta_0)}{2}(\overline{H}_1(\theta_0)-H_1(\theta_0))\, -\, \beta_2(\theta_0)(\overline{H}_2(\theta_0)-H_1(\theta_0)) } 
		&\ll \sqrt{n}\, r_n\, \Delta_n^\epsilon.
	\end{align*}
	Then
	\begin{align}
		\sup_{\theta\in \mathbf{B}_{r_n}(\theta_0)} \|\underline{C}_n(\theta)\| \; \| B_n(\theta) \Lambda_n(\theta_0)\| = o(r_n), \label{eqn:ABC-rate-1}
	\end{align}
	and 
	\begin{align}
		\sup_{\theta\in \mathbf{B}_{r_n}(\theta_0)} \|B_n(\theta)B_n(\theta_0)^{-1} - I\| + \|\underline{C}_n(\theta)\underline{C}_n(\theta_0)^{-1}-I\| + \|\underline{W}(\theta)-\underline{W}(\theta_0)\| \pconv 0. \label{eqn:ABC-rate-2}
	\end{align}
\end{lemma}
\begin{proof}[Proof of Lemma \ref{lem:ABC-rate}]
	The second claim is a direct consequence of the definition of $\underline{C}_n$, $B_n$, and $\underline{W}$, and of the requirement $r_n\ll \frac{1}{|\log \Delta_n|^2}$. Regarding the first claim, we observe that
	\begin{align*}
		\sup_{\theta\in \mathbf{B}_{r_n}(\theta_0)} \|\underline{C}_n(\theta)\| 
		&= \mathcal{O} \left( |\log \Delta_n| w_n^{-2} \Delta_n^{2(H_1(\theta_0) - H_2(\theta_0))} \right) \\
		& \qquad + \mathcal{O} \left( |\log \Delta_n| w_n^{-2} d_n \max_{j=1,2} \Delta_n^{\beta_j(\theta_0)(H_1(\theta_0) - \overline{H}_j(\theta_0))} \right) \\
		&\leq \mathcal{O} \left( \Delta_n^{-\epsilon}  \Delta_n^{2(H_1(\theta_0) - H_2(\theta_0))} \right)  
		+ \mathcal{O} \left( |\log \Delta_n| w_n^{-2} d_n \Delta_n^{\beta_2(\theta_0)(H_1(\theta_0) - \overline{H}_2(\theta_0))} \right) \\
		&\leq \mathcal{O} \left( \Delta_n^{-\epsilon}  \Delta_n^{2(H_1(\theta_0) - H_2(\theta_0))} \right)  \\
		&\quad + \mathcal{O} \left( \Delta_n^{-\epsilon} \Delta_n^{\frac{\beta_1(\theta_0)}{2}(\overline{H}_1(\theta_0) - H_1(\theta_0))\, +\, \beta_2(\theta_0)(H_1(\theta_0) - \overline{H}_2(\theta_0))} \right) \\
		&= o(r_n \sqrt{n}). 
	\end{align*}
	Here, we used that $\beta_1(\overline{H}_1-H_1) \leq \beta_2(\overline{H}_2-H_2)$.
	Moreover,
	\begin{align*}
		\sup_{\theta\in \mathbf{B}_{r_n}(\theta_0)} \| B_n(\theta) \Lambda_n(\theta_0) \| 
		&= \mathcal{O}\left( \frac{1}{\sqrt{n}} w_n^2 \right) 
		+ \mathcal{O}\left( \frac{1}{\sqrt{n}} d_n^{-1}(\theta_0) (u_n \Delta_n^{\overline{H}_1(\theta_0)})^{\frac{\beta_1(\theta_0)}{2}} \right) \\
		&=\mathcal{O}\left( \frac{1}{\sqrt{n}} \right).
	\end{align*}
	This allows us to derive the rate \eqref{eqn:ABC-rate-1}.
\end{proof}

\subsubsection{Gradients}

Recall the definition
\begin{align*}
	\underline{W}(\theta) &= \begin{pmatrix}
		\underline{W}_1(\theta) & 0 \\ 0 & \underline{W}_2(\theta)
	\end{pmatrix} \in \R^{10\times 10},\\
	\underline{W}_1(\theta)_{r,j} &= W(f_1, \theta, \lambda_r,\gamma_r, 0)_{j}, \qquad r,j=1,\ldots, 4, \\ 	
	\underline{W}_2(\theta)_{r,j} &= W(f_2, \theta, \lambda_{r+4},\gamma_{r+4}, 0)_{j+4}, \qquad r,j=1,\ldots, 6,
\end{align*}
where $W(f,\theta,\lambda,\gamma,0)$ is defined in \eqref{eqn:def-W}.

\begin{lemma}\label{lem:gradient-threshold-moment}
	Let $f$ be a Schwartz function such that $f(x)=f(0)$ for $|x|<\eta$, for some $\eta>0$.
	Let $u_n(\theta) = \Delta_n^{-H_1(\theta)} w_n$ for $w_n \to 0$ satisfying \eqref{eqn:wn}.
	Then, for $j=1,2$, as $n\to\infty$, and $\Delta_n\to 0$, and for any compact set $\mathbf{K}\subset\overline{\Theta}$,
	\begin{align*}
		\sup_{\theta \in \mathbf{K}} \left\| [D_{(\widetilde{a}_j, H_j)}\,\E_\theta f(\lambda u_n(\theta) X_{1,n,\gamma})] \, \underline{C}_n^j(\theta) \right\| 
		&= \mathcal{O}\left( \sup_{\theta\in\mathbf{K}} \max_{i=1,2} \Delta_n^{\beta_i(\overline{H}_i - H_1)} w_n^{\beta_i-2} |\log \Delta_n| \right),
	\end{align*}
\end{lemma}
\begin{proof}[Proof of Lemma \ref{lem:gradient-threshold-moment}]
	Just as in Lemma \ref{lem:moment-gradient-adaptive}, we find that, 
	\begin{align*}
		&\begin{pmatrix}
			\partial_{\widetilde{a}_2} & \partial_{H_2} 
		\end{pmatrix}
		\E_\theta f(\lambda u_n(\theta) X_{1,n,\gamma}) \\
		&= \int \widehat{f}(v) \exp(-\psi_n(v\,\lambda u_n(\theta),\theta,\gamma)) |w_n \lambda v|^{2} \Delta_n^{2(H_2-H_1)} \gamma^{2 H_2} \widetilde{a}_2 \\
		&\qquad \quad 
		\begin{pmatrix}
			1/\widetilde{a}_2 \\ 
			2 [\log(\gamma) + \log(\Delta_n)]
		\end{pmatrix}^T\, dv \\
		&= \E_\theta f''(\lambda u_n X_{1,n,\gamma}) \left[ |w_n \lambda|^2 \Delta_n^{2(H_2-H_1)} \gamma^{2H_2} \widetilde{a}_2 
		\begin{pmatrix}
			1/\widetilde{a}_2 \\ 
			2 [\log(\gamma) + \log(\Delta_n)]
		\end{pmatrix}^T \right],
	\end{align*}
	such that
	\begin{align*}
		\begin{pmatrix}
			\partial_{\widetilde{a}_2} & \partial_{H_2} 
		\end{pmatrix}
		\E_\theta f(\lambda u_n(\theta) X_{1,n,\gamma}) \, \underline{C}_n^2(\theta) 
		&=  \E_\theta f''(\lambda u_n X_{1,n,\gamma}) \cdot \mathcal{O}\left( 1 \right),
	\end{align*}
	uniformly on compacts in $\overline{\Theta}$. 
Furthermore,
	\begin{align*}
		\partial_{\widetilde{a}_1} 	\E_\theta f(\lambda u_n(\theta) X_{1,n,\gamma})
		&= \E_\theta f''(\lambda u_n X_{1,n,\gamma})  |w_n \lambda|^2  \gamma^{2H_1}.
	\end{align*}
	The derivative $\partial_{H_1}$ takes the form
	\begin{align*}
		&\partial_{H_1}\E_\theta f(\lambda u_n(\theta) X_{1,n,\gamma}) \\
		&=  \int \widehat{f}(v) \exp(-\psi_n(v\,\lambda u_n(\theta),\theta,\gamma)) |w_n \lambda v|^{2}  \gamma^{2 H_1} \widetilde{a}_1 [2 \log \gamma ] \\
		&\quad -  \int \widehat{f}(v) \exp(-\psi_n(v\,\lambda u_n(\theta),\theta,\gamma)) |w_n \lambda v|^{2}  \gamma^{2 H_2} \widetilde{a}_2  \Delta_n^{2(H_2-H_1)} \, 2 \log \Delta_n \\
		&\quad -  \int \widehat{f}(v) \exp(-\psi_n(v\,\lambda u_n(\theta),\theta,\gamma)) \sum_{j=1}^2 |w_n \lambda v|^{\beta_j}  \gamma^{\beta_j \overline{H}_j} \widetilde{b}_j  \Delta_n^{\beta_j(\overline{H}_j-H_1)} [\beta_j \log \Delta_n] \\
		&= \E_\theta f''(\lambda u_n X_{1,n,\gamma}) |w_n \lambda|^2 \left[ 2\widetilde{a}_1 \gamma^{2H_1} \log \gamma - 2 \widetilde{a}_2 \gamma^{2H_2} \Delta_n^{2(H_2-H_1)} \log \Delta_n    \right] \\
		&+ \mathcal{O}\left( \max_{i=1,2} |w_n|^{\beta_i} \Delta_n^{\beta_i(\overline{H}_i-H_1)} |\log \Delta_n| \right) \\
		&= \E_\theta f''(\lambda u_n X_{1,n,\gamma}) \cdot \mathcal{O}(w_n^2 + 1) + \mathcal{O}\left( \max_{i=1,2} |w_n|^{\beta_i} \Delta_n^{\beta_i(\overline{H}_i-H_1)} |\log \Delta_n| \right).
	\end{align*}
	Hence,
	\begin{align*}
		&\begin{pmatrix}
			\partial_{\widetilde{a}_1} & \partial_{H_1} 
		\end{pmatrix}
		\E_\theta f(\lambda u_n(\theta) X_{1,n,\gamma}) \, \underline{C}_n^1(\theta) \\
		&= \E_\theta f''(\lambda u_n X_{1,n,\gamma}) \cdot \mathcal{O}\left( w_n^{-2} \right) + \mathcal{O}\left( \max_{i=1,2} |w_n|^{\beta_i-2} \Delta_n^{\beta_i(\overline{H}_i-H_1)} |\log \Delta_n| \right).
	\end{align*}
	where the implicit constant in the $\mathcal{O}(\ldots)$ term is bounded on compacts in $\overline{\Theta}$.
	Now use Lemma \ref{lem:mom}(i), and the upper bound on $w_n$, to find that
	\begin{align*}
		\E_\theta f''(\lambda u_n X_{1,n,\gamma}) = \mathcal{O}\left(\max_{i=1,2} |w_n|^{\beta_i} \Delta_n^{\beta_i(\overline{H}_i-H_1)}\right).
	\end{align*}
	Tracing the constants in the proof of Lemma \ref{lem:mom}, the latter upper bound is found to hold uniformly on compacts in $\overline{\Theta}$.
\end{proof}

Combining Lemma \ref{lem:gradient-threshold-moment} and Lemma \ref{lem:moment-gradient-adaptive-stochastic}, we obtain the following Corollary.

\begin{cor}\label{cor:gradient-stochastic-threshold}
	Let $f$ be a Schwartz function such that $f(x)=f(0)$ for $|x|<\eta$, for some $\eta>0$.
	Let $u_n(\theta) = \Delta_n^{-H_1(\theta)} w_n$ for $w_n \to 0$.
	Suppose that the order of differencing $k$ satisfies
	\begin{align*}
		k > H_j +\frac{1}{\beta_j}, \qquad j=1,\ldots, q.
	\end{align*}
	Then, for $j=1,2$, as $n\to\infty$, and $\Delta_n\to 0$, for any sequence $r_n\ll 1/|\log \Delta_n|$, 
	\begin{align*}
		&\quad \sup_{\theta \in \mathbf{B}_{r_n}(\theta_0)} \left\| [D_{(\widetilde{a}_j, H_j)}\, \mathcal{H}_n(\theta) ] \, \underline{C}_n^j(\theta) \right\| \\
		&= \mathcal{O}\left( \max_{i=1,2} \Delta_n^{\beta_i(\theta_0)(\overline{H}_i(\theta_0) - H_1(\theta_0))} w_n^{\beta_i(\theta_0)-2} |\log \Delta_n| \right) + o_{\mathbb{P}}\left( \frac{|\log \Delta_n|}{\sqrt{n}} \right).
	\end{align*}
\end{cor}

\begin{lemma}\label{lem:gradient-threshold}
	Let $\mathcal{H}_n(\theta)$ as in \eqref{eqn:ee-threshold}, and suppose that the order of differencing $k$ satisfies
	\begin{align*}
		k > H_j +\frac{1}{\beta_j}, \qquad j=1,\ldots, q.
	\end{align*}
	Let $u_n (\theta) = \Delta_n^{-H_1(\theta)} w_n$, and $w_n \to 0$.
	Then for any sequence $r_n\ll 1/|\log \Delta_n|$, 
	\begin{align*}
		\sup_{\theta \in \mathbf{B}_{r_n}(\theta_0)} \left\| B_n(\theta) D \mathcal{H}_n(\theta) \underline{C}_n(\theta) - \underline{W}(\theta) \right\| = o_{\mathbb{P}}(1),
	\end{align*}
	where convergence in probability holds under the probability measure induced by $\theta_0$.
\end{lemma}
\begin{proof}[Proof of Lemma \ref{lem:gradient-threshold}]
	\underline{Rows $r=1,\ldots, 4$:}
	Note that the matrices $\underline{C}_n^j$ are submatrices of the $C^j_n$, such that Corollary \ref{cor:gradient} and Lemma \ref{lem:moment-gradient-adaptive-stochastic} are applicable.
	In particular, for the first $r=1,2,3,4$, moments, which are based on the function $f_1$, Corollary \ref{cor:gradient} yields that
	\begin{align*}
		\sup_{\theta \in \mathbf{B}_{r_n}(\theta_0)} \left\| D_{(\widetilde{a}_j, H_j)} \mathcal{H}_n(\theta)_r \underline{C}_n^j(\theta) - W(f_1,\theta,\lambda_r,\gamma_r, 0)_{(2j-1, 2j)} \right\| = o_{\mathbb{P}}(1), \quad j=1,2.
	\end{align*}
	Moreover, the matrices $\doubleunderline{C}_n^j$ are identical to the matrices $C_n^{2+j}$, $j=1,2$, except for the scaling factor $d_n$. 
	Hence, Corollary \ref{cor:gradient} yields
	\begin{align*}
		\sup_{\theta \in \mathbf{B}_{r_n}(\theta_0)} \left\| D_{(\widetilde{b}_j, H_j,\beta_j)} \mathcal{H}_n(\theta)_r \doubleunderline{C}_n^j(\theta)\right\| = \mathcal{O}_{\mathbb{P}}(d_n), \quad j=1,2.
	\end{align*}
	Since $d_n\to 0$, we obtain for $r=1,\ldots, 4$, and $v=1,\ldots, 10$,
	\begin{align}
		\sup_{\theta \in \mathbf{B}_{r_n}(\theta_0)} \left| \left[B_n D \mathcal{H}_n(\theta) \doubleunderline{C}_n^j(\theta)\right]_{r,v} -\begin{pmatrix}
			\underline{W}_1(\theta) & 0
		\end{pmatrix}_{r,v} \right| = o_{\mathbb{P}}(1).
		\label{eqn:gradient-threshold-1}
	\end{align}

	\underline{Rows $r=5,\ldots, 10$:}
	Now consider the moments $r=5,\ldots, 10$, based on $f_2$. 
	Again, Corollary \ref{cor:gradient} yields
	\begin{align*}
		\sup_{\theta \in \mathbf{B}_{r_n}(\theta_0)} \left\| D_{(\widetilde{b}_j, H_j,\beta_j)} \mathcal{H}_n(\theta)_r \doubleunderline{C}_n^j(\theta) - d_n W(f_2,\theta,\lambda_r,\gamma_r, 0)_{4+3j-(2,1,0)}\right\| = o_{\mathbb{P}}(d_n), \quad j=1,2.
	\end{align*}
	Corollary \ref{cor:gradient-stochastic-threshold} yields
	\begin{align*}
		\sup_{\theta \in \mathbf{B}_{r_n}(\theta_0)} \left\| D_{(\widetilde{a}_j, H_j)} \mathcal{H}_n(\theta)_r \underline{C}_n^j(\theta)\right\| = o_{\mathbb{P}}(d_n), \quad j=1,2,
	\end{align*}
	by our choice of $d_n$.
	Hence, we obtain for $r=5,\ldots, 10$, and $v=1,\ldots, 10$,
		\begin{align}
		\sup_{\theta \in \mathbf{B}_{r_n}(\theta_0)} \left| \left[B_n D_{(\widetilde{b}_j, H_j,\beta_j)} \mathcal{H}_n(\theta) \doubleunderline{C}_n^j(\theta)\right]_{r,v} -\begin{pmatrix}
			0 & \underline{W}_2(\theta)
		\end{pmatrix}_{r,v} \right| = o_{\mathbb{P}}(1).
		\label{eqn:gradient-threshold-2}
	\end{align}
	Now \eqref{eqn:gradient-threshold-1} and \eqref{eqn:gradient-threshold-2} yield the desired claim.
\end{proof}

Note that Lemma \ref{lem:gradient-threshold} requires $r_n\ll 1/|\log \Delta_n|$, whereas Lemma \ref{lem:ABC-rate} requires $r_n\ll 1/|\log\Delta_n|$.
We will thus apply the theory of estimating equations, i.e.\ Theorem \ref{thm:ee-consistency}, with $r_n\ll 1/|\log\Delta_n|^2$.

\subsubsection{Limit theory}

\begin{theorem}\label{thm:clt-threshold}
	Set $u_n(\theta) = w_n \Delta_n^{-H_1(\theta)}$ for $w_n\to 0$ satisfying \eqref{eqn:wn}.
	Suppose that 
	\begin{align*}
		k > \max\left(H_j(\theta_0) + \frac{1}{2},\quad  \overline{H}_j(\theta_0)+\frac{1}{\beta_j(\theta_0)}\right),\quad j=1,2, \\
		n\, |\log \Delta_n|^{-\frac{\beta_1(\theta_0)}{2}}\Delta_n^{\beta_1(\theta_0)[\overline{H}_1(\theta_0)-H_1(\theta_0)]} \to \infty, \qquad \frac{n}{|\log \Delta_n|^2} \to \infty.
	\end{align*}
	If $w_0\leq \frac{\eta}{9\sigma}$, then there exist symmetric positive semidefinite matrices $\Sigma_1\in\R^{4\times 4}$, $\Sigma_2\in\R^{6\times 6}$, such that
	\begin{align*}
		\Lambda_n(\theta_0)^{-1}  \mathcal{H}_n(\theta_0) 
		\;\wconv\; \mathcal{N}\left(0, \begin{pmatrix} \Sigma_1 & 0 \\ 0 & \Sigma_2 \end{pmatrix}\right).
	\end{align*}
	The matrices $\Sigma_1$ and $\Sigma_2$ are given by \eqref{eqn:asymp-cov-1} and \eqref{eqn:asymp-cov-2}, respectively.
\end{theorem}
\begin{proof}[Proof of Theorem \ref{thm:clt-threshold}]
	We apply Theorem \ref{thm:clt} to the array of multivariate time series $\left[u_n X_{l,n,\gamma_r}\right]_{r=1}^{10}$ with
	\begin{align*}
		a_{n,j} &= u_n(\theta_0) \Delta_n^{H_j(\theta_0)} = w_n(\theta_0) \Delta_n^{H_j(\theta_0)-H_1(\theta_0)}, \qquad j=1,2,\\
		b_{n,j} &= u_n(\theta_0) \Delta_n^{\overline{H}_j(\theta_0)} = w_n(\theta_0) \Delta_n^{\overline{H}_j(\theta_0) - H_1(\theta_0)}, \qquad j=1,2, \\
		\delta_0 &= H_1 - \frac{1}{2} - k, \\
		\delta_j &= \overline{H}_j - \frac{1}{\beta_j} -k, \qquad j=1,2, 
	\end{align*}
	and with some finite $\sigma<\infty$, depending on $\theta_0$ and $k$.
	In particular, by our choice of $u_n$, most notably the choice of $w_n$, 
	\begin{align*}
		\exp\left( - \frac{\eta^2}{2d \sum_{i=1}^2 a_{n,i} \sigma^2} \right) 
		&\leq \exp\left( - \frac{\eta^2  |\log \Delta_n|}{40\, w_0^2 \sigma^2 } \right) = \Delta_n^{\frac{\eta^2}{40w_0^2 \sigma^2}}, 
	\end{align*}
	because $d=10$.
	Now choose $w_0$ small enough such that the exponent is strictly larger than $2$, such that case (v) of Theorem \ref{thm:clt} is applicable, just as in the proof of Lemma \ref{lem:clt-G}.
	For this, it is sufficient that $w_0\leq \frac{\eta}{9\sigma}$.
	
	We now compute explicit expressions for the asymptotic covariance matrices $\Sigma_1\in\R^{4\times 4}$ and $\Sigma_2\in\R^{6\times 6}$.
	Using formula \eqref{eqn:asymp-var}, an application of the Cramer-Wold device yields
	\begin{align*}
	    (\Sigma_1)_{r,r'} 
	    &= \frac{1}{4} \left[\lambda_r^2f_1''(0)\right] \left[ \lambda_{r'}^2 f_1''(0) \right] \left[ \sum_{z\in\Z} (\Sigma_{r r'}^{1,\infty, z})^2\right],\\
	    \Sigma_{r r'}^{1,\infty, z}
	    &= \int_{-\infty}^0 \tilde{h}_r(s-z) \tilde{h}_{r'}(s)\, ds, \\
	    \tilde{h}_r(s)
	    &= \sum_{v=0}^k \binom{k}{v} (-1)^v (\gamma_r v - s)_+^{H_1- \frac{1}{2}}.
	\end{align*}
	Recalling the notation $Y_{l,\gamma}^{H,\beta}$ from \eqref{eqn:lfsm-kdiff}, this may be written as
	\begin{align*}
	    \Sigma_{r r'}^{1,\infty, z} 
	    &= \Cov \left(Y_{z,\gamma_r}^{H_1, 2}, Y_{0,\gamma_{r'}}^{H_1, 2}\right),
	\end{align*}
	so that
	\begin{align}
	    (\Sigma_1)_{r,r'} 
	    &= \frac{\lambda_r^2 \lambda_{r'}^2}{4} f''(0)^2 \sum_{z\in\Z} \Cov \left(Y_{z,\gamma_r}^{H_1, 2}, Y_{0,\gamma_{r'}}^{H_1, 2}\right)^2. \label{eqn:asymp-cov-1}  
	\end{align}
	For the matrix $\Sigma_2$, we use formula \eqref{eqn:asymp-var} and the Cramer-Wold device to obtain
	\begin{align}
	    & (\Sigma_2)_{(r-4)(r'-4)} \nonumber \\
	    &= \sum_{z\in\Z} \frac{d}{dh} \E \left[f_2(h^{\frac{1}{\beta_1}}\lambda_{r} Y_{z, \gamma_{r}}^{\overline{H}_1, \beta_1}) \cdot f_2(h^{\frac{1}{\beta_1}} \lambda_{r'} Y_{0, \gamma_{r'}}^{\overline{H}_1, \beta_1}) \right] \Big|_{h=0}, \qquad r,r'=5,\ldots, 10. \label{eqn:asymp-cov-2}
	\end{align}
\end{proof}

\begin{proof}[Proof of Theorem \ref{thm:clt-threshold-estimator}]
	We have already verified the conditions of Theorem \ref{thm:ee-consistency} with $A_n = \Lambda_n(\theta_0)^{-1}$, $B_n(\theta)$ and $C_n(\theta)$, for any $r_n\ll 1/|\log\Delta_n|^2$. 
	In particular, Theorem \ref{thm:clt-threshold} ensures condition (E.1); Lemma \ref{lem:gradient-threshold} ensures the first part of (E.2); and Lemma \ref{lem:ABC-rate} ensures the second part of (E.2), and (E.3).
	Note that both results are applicable because of our identifiability condition \eqref{eqn:identifiability-threshold}.
	
	Note that the matrix $\underline{W}(\theta)$ is regular by assumption.
	Hence, we find that there exists a sequence $\widehat{\theta}_n$ of random vectors such that $\mathbb{P}(\mathcal{H}_n(\widehat{\theta_n})=0)\to 1$, and 
	\begin{align*}
		\Lambda_n^{-1}(\theta_0) B_n^{-1} \underline{W}(\theta_0) \underline{C}_n^{-1}(\theta_0) \left[ \widehat{\theta}_n-\theta_0 \right] \wconv \mathcal{N}(0,\widetilde{\Sigma}),
	\end{align*}
	for some asymptotic covariance matrix $\widetilde{\Sigma}\in\R^{10\times 10}$.
	To simplify the scaling matrix, we may exploit that $\underline{W}(\theta_0)$ is block-diagonal, and $\Lambda_n$ and $B_n$ are of the matching form $\diag(x,\ldots, x, y,\ldots, y)$ for some $x,y\in\R$.
	Thus,
	\begin{align*}
		\Lambda_n^{-1}(\theta_0) B_n^{-1} \underline{W}(\theta_0) \underline{C}_n^{-1}(\theta_0) 
		= \underline{W}(\theta_0) \Lambda_n^{-1}(\theta_0) B_n^{-1}  \underline{C}_n^{-1}(\theta_0)
		= \underline{W}(\theta_0) R_n(\theta_0)^{-1}.
	\end{align*}
	Since $\underline{W}(\theta_0)$ is regular by assumption, we find that
	\begin{align*}
		R_n(\theta_0)^{-1} \left[ \widehat{\theta}_n-\theta_0 \right] \wconv \mathcal{N}(0,\underline{W}(\theta_0)^{-1}\widetilde{\Sigma} (\underline{W}(\theta_0)^{-1})^T).
	\end{align*}
\end{proof}

\subsubsection{Regularity of $\underline{W}$}

\begin{proof}[Proof of Proposition \ref{prop:condition-estimating}]
    \noindent\underline{Case (i):}
    The matrix $\underline{W}_1(\theta)$ may be written as
    \begin{align*}
        \left(\underline{W}_1(\theta)_{i,1},\ldots, \underline{W}_1(\theta)_{i,4}\right)
        &= f_1''(0) \lambda_i^2 \left( \gamma_i^{2H_1}, \widetilde{a}_1\gamma_i^{2H_1}\log(\gamma_i), \gamma_i^{2H_2}, \widetilde{a}_2\gamma_i^{2H_2}\log(\gamma_i) \right),
    \end{align*}
    for $i=1,\ldots, 4$.
    Via equivalent matrix operations, we find that $\underline{W}_1(\theta)$ is regular if and only if $\underline{W}_1'(\theta) \in\R^{4\times 4}$ is regular, where
    \begin{align*}
        \left(\underline{W}'_1(\theta)_{i,1},\ldots, \underline{W}'_1(\theta)_{i,4}\right)
        &= \left( \gamma_i^{2H_1},\; \gamma_i^{2H_1}\log_2(\gamma_i),\; \gamma_i^{2H_2},\; \gamma_i^{2H_2}\log_2(\gamma_i) \right).
    \end{align*}
    Computing the determinant symbolically, we find that
    \begin{align*}
        \det(\underline{W}_1(\theta))
        &= 2^{2 H_{1}+2 H_{2}} \left(2^{H_{1}} - 2^{H_{2}}\right)^{4} \left(2^{H_{1}} + 2^{H_{2}}\right)^{4},
    \end{align*}
    which is zero if and only if $H_1=H_2$. 
    Note that the latter equality is excluded by our parameter set $\Theta$.

    \noindent\underline{Cases (ii) and (iii):}
    Denote 
    \begin{align*}
        \chi_j &= \int \widehat{f}_2(v) |v|^{\beta_j}\, dv, \qquad j=1,2,\\
        \zeta_j &= \int \widehat{f}_2(v) |v|^{\beta_j} \log |v|\, dv = \partial_{\beta_j} \chi_j,\qquad j=1,2. 
    \end{align*}
    By virtue of the Lévy-Khinchine formula, $\chi_j = c(\beta_j) \int \frac{f_2(x)}{|x|^{1+\beta_j}}\, dx$, which is nonzero for our choice of $f_2$.
    The matrix $\underline{W}_2$ may be written as
    \begin{align*}
        &(\underline{W}_2(\theta)_{i, 3j-2}, \underline{W}_2(\theta)_{i, 3j-1}, \underline{W}_2(\theta)_{i, 3j}) \\
        &=  \widetilde{b}_j\gamma_{i+4}^{\beta_j\overline{H}_j} |\lambda_{i+4}|^{\beta_j} \;
        \left(\tfrac{\chi_j }{\widetilde{b}_j},\; \chi_j \beta_j\log(\gamma_{i+4}), \; \chi_j \overline{H}_j \log(\gamma_{i+4}) + \chi_j \log(\lambda_{i+4}) + \zeta_j \right),
    \end{align*}
    for indices $i=1,\ldots, 6$, and $j=1,2$. 
    Since $\chi_j\neq 0$, we may rescale the columns and rows, and eliminate terms by linear addition of columns, to obtain the equivalent matrix $\underline{W}'_2(\theta)$ of the form
    \begin{align*}
        (\underline{W}'_2(\theta)_{i, 3j-2}, \underline{W}'_2(\theta)_{i, 3j-1}, \underline{W}'_2(\theta)_{i, 3j}) 
        \;=\;  \gamma_{i+4}^{\beta_j\overline{H}_j} |\lambda_{i+4}|^{\beta_j} \;
        \left(1,\; \log_2(\gamma_{i+4}), \;  \log_2(\lambda_{i+4}) \right).
    \end{align*}
    In particular, the matrix $\underline{W}_2$ is regular if and only if $\underline{W}_2'$ is regular.
    
    We may now compute the determinant of the $6\times 6$ matrix $\underline{W}_2'$ symbolically, e.g.\ using a computer algebra system. 
    Using the Python library SymPy (v1.9), we obtain in case (ii)
    \begin{align*}
        &\det(W_2'(\theta)) \\
        &= 2^{3 \beta_{1}+3 \beta_{2}+2 H_{1} \beta_{1}+2 H_{2} \beta_{2}} \left(2^{\beta_{1}+H_{1} \beta_{1}} - 2^{\beta_{2}+H_{2} \beta_{2}}\right) \left(2^{2 \beta_{1}+H_{1} \beta_{1}} - 2^{2 \beta_{2}+H_{2} \beta_{2}}\right)^{4},
    \end{align*}
    which is zero if and only if $\beta_1(1+H_1) = \beta_2(1+H_2)$ or $\beta_1(2+H_1) = \beta_2(2+H_2)$.
    
    In case (iii), we obtain
    \begin{align*}
        &\det(W_2'(\theta)) \\
        &=- 2^{\beta_{1}+\beta_{2}+2 H_{1} \beta_{1}+2 H_{2} \beta_{2}} \left(2^{H_{1} \beta_{1}} - 2^{H_{2} \beta_{2}}\right)^{5} \left(2^{H_{1} \beta_{1}} + 2^{H_{2} \beta_{2}}\right),
    \end{align*}
    which is zero if and only if $H_1\beta_1 = H_2\beta_2$.
\end{proof}

\subsection{Identifiability for continuous observations}\label{sec:identifiability}

The following technical result is implicit in the proof of \cite[Lemma 4.7]{takashima1989sample}.
\begin{lemma}[\cite{takashima1989sample}]\label{lem:Holder-local}
    Let $Y_t^{H,\beta}$ be a fractional stable motion, and $\alpha = H-\frac{1}{\beta}>0$.
    With probability $1$, the following holds: For any $\epsilon>0$, there exists some $m\in\N$, real numbers $t_1,\ldots, t_m\in[0,1]$, $\eta_1,\ldots, \eta_m>0$, and $\zeta>0$, such that the following holds:
    \begin{enumerate}[(i)]
        \item $[0,1] \subset \bigcup_{k=1}^m (t_k - \frac{\eta_k}{2}, t_k + \frac{\eta_k}{2})$
        \item If $|\Delta Z_t| > \epsilon$, then $t\in \{t_1,\ldots, t_m\}$.
        \item For any $h<\zeta$, and for any $0\leq s < t \leq 1$ such that $s,t\in (t_k-\frac{\eta_k}{2}, t_k+\frac{\eta_k}{2})$, and $|s-t|<h$,
        \begin{align*}
            \left| \frac{Y_t^{H,\beta}-Y_s^{H,\beta}}{h^{\alpha}} - \Delta Z^\beta_{t_k} \left(\frac{t-s}{h}\right)^{\alpha} \phi_k(t,s) \right| 
            \leq \epsilon \left[ |Z_{t_k-}| + |\Delta Z_{t_k}| + 4 \right], \\
            \text{where} \quad \phi_k(t,s) 
            = \left( 1-\tfrac{t_k-s}{t-s} \right)_+^\alpha - \left(-\tfrac{t_k-s}{t-s} \right)_+^\alpha.
        \end{align*}
    \end{enumerate}
\end{lemma}

\begin{proof}[Proof of Theorem \ref{thm:identifiability-continuous}]
    \underline{Identifying jumps of the roughest components:}
    For any $\alpha>0$, the quantity 
    \begin{align*}
        T_\alpha 
        &= \lim_{h\downarrow 0} \sup_{0<s<t\leq 1, |t-s|\leq h} \frac{|X_{t}-X_s|}{h^\alpha} 
    \end{align*}
    is a measurable function of $(X_t)_{t\in[0,1]}$, taking values in $[0,\infty]$. 
    Denote $\alpha^* = \max_j H_j - \frac{1}{\beta_j}$, and $J^* = \{j: H_j-\frac{1}{\beta_j} = \alpha^*\}$. 
    By virtue of Lemma \ref{lem:Holder-local}, we find that, with probability $1$,
    \begin{align*}
        T_\alpha 
        = 
        \begin{cases}
            \infty, & \alpha > \alpha^*, \\
            \max_{j\in J^*} \sup_{t\in(0,1)} b_j |\Delta Z_t^{\beta_j}|, & \alpha = \alpha^*, \\
            0, & \alpha < \alpha^*.
        \end{cases}
    \end{align*}
    Then 
    \begin{align*}
        \hat{\alpha}^* 
        &= \sup \left\{ \alpha>0 : T_\alpha = 0 \right\} 
        = \sup \left\{ \alpha>0, \alpha\in\Q : T_\alpha = 0 \right\}
    \end{align*}
    is measurable as well, because $\{\hat{\alpha}^* \geq \alpha_0\} = \bigcap_{\alpha<\alpha_0, \alpha\in\Q} \{ T_{\alpha}=0 \}$.
    Moreover, $P(\hat{\alpha}^* = \alpha)=1$.
    
    Now introduce the process
    \begin{align*}
        Z_t^* = \sum_{j \in J^*} b_j Z_t^{\beta_j}.
    \end{align*}
    We show how to reconstruct $Z_t^*$ from $X_t$.
    Since $Z_t^*$ is a Lévy process, it has countably many jumps. 
    Moreover, with probability one, all jumps have different sizes, because the Lévy admits a Lebesgue density.
    Furthermore, with probability one, no two constituent processes $Z_t^{\beta_j}$ jump at the same time. 
    Therefore, $T_{\alpha^*} = \sup_{t\in(0,1)} |\Delta Z_t^*|$.
    
    Denote the jump times by $(\tau_i)_{i\in\N}$ such that $|\Delta Z_{t_1}^*| > |\Delta Z_{t_2}^*| >\ldots$.
    For any $n\in\N$, define
    \begin{align*}
        \widehat{\tau}_1(n) 
        = \inf \left( \arg\max_{t\in(0,1-\frac{1}{n})\cap \Q} \frac{|X_{t+\frac{1}{n}} - X_t|}{n^{-\hat{\alpha}^*} } \right).
    \end{align*}
    Then $\widehat{\tau}_1(n)\to \tau_1$ almost surely by virtue of Lemma \ref{lem:Holder-local}, and 
    \begin{align*}
        \widehat{\Delta}_1(n) = \frac{X_{\widehat{\tau}_1(n) + \frac{1}{n}} - X_{\widehat{\tau}_1(n)}}{n^{-\hat{\alpha}^*}} \to \Delta Z^*_{\tau_1}, 
    \end{align*}
    almost surely as $n\to\infty$. 
    That is, we identified the largest jump and its jump time, i.e.\ there exist measurable random variables $\widehat{\tau}_1$ and $\widehat{\Delta}_1$ such that 
    \begin{align*}
        P(\widehat{\tau}_1 = \tau_1) = P(\widehat{\Delta}_1 = \Delta Z^*_{\tau_1}) = 1.
    \end{align*}
    
    How to identify the smaller jumps?
    For any $\delta>0$, define
    \begin{align*}
        \widehat{\tau}_2(n,\delta) 
        = \inf \left( \arg\max_{t\in(0,1-\frac{1}{n})\cap \Q, |t-\widehat{\tau}_1|>\delta} \frac{|X_{t+\frac{1}{n}} - X_t|}{n^{-\hat{\alpha}^*} } \right).
    \end{align*}
    Invoking Lemma \ref{lem:Holder-local} again, with probability one,
    \begin{align*}
        \lim_{n\to\infty} \widehat{\tau}_2(n,\delta) = \widehat{\tau}_2(\delta) = \arg\max_{t\in(0,1), |t-\tau_1|>\delta} |\Delta Z^*_t|,
    \end{align*}
    and the $\arg\max$ is unique because all jumps of $Z^*$ have different sizes, with probability one. 
    Thus,
    \begin{align*}
        \widehat{\tau}_2  = \lim_{\delta\to 0} \widehat{\tau}_2(\delta) =\arg\max_{t\in(0,1), t\neq \tau_1} |\Delta Z^*_t| = \tau_2.
    \end{align*}
    As before, we may now construct $\widehat{\Delta}_2$ such that
    \begin{align*}
        \mathbb{P}(\widehat{\tau}_2 = \tau_2) = \mathbb{P}(\widehat{\Delta}_2 = \Delta Z^*_{\tau_2}) = 1.
    \end{align*}
    Iterating this procedure, we find a sequences $\widehat{\tau}_i$ and $\widehat{\Delta}_i$ of random variables such that
    \begin{align*}
        \mathbb{P}\left((\widehat{\tau}_i, \widehat{\Delta}_i) = (\tau_i, \Delta Z^*_{\tau_i}) \; \forall i\in\N\right)  = 1.
    \end{align*}
    In particular, the $(\widehat{\tau}_i, \widehat{\Delta}_i)$ are measurable functions of $(X_t)_{t\in[0,1]}$.
    
    \underline{Disentangling the jumps:}
    Let $N$ be the Poisson jump measure on $[0,1]\times \R$ which corresponds to the Lévy process $(Z_t^*)_{t\in[0,1]}$, i.e.\
    \begin{align*}
        N(A) &= \sum_{t\in [0,1]} \mathds{1}\left( (t, \Delta Z_t)\in A \right) \\
        &= \sum_{i=1}^\infty \mathds{1}\left( (\tau_i, \Delta Z_{\tau_i}) \in A \right)
    \end{align*}
    for any set $A\subset [0,1]\times (\R\setminus (-\epsilon,\epsilon))$, and any $\epsilon>0$.
    If we define the infinite point measure
    \begin{align*}
        \widehat{N}(A) 
        &= \sum_{i=1}^\infty \mathds{1}\left( (\widehat{\tau}_i, \widehat{\Delta}_i) \in A \right),
    \end{align*}
    then $N=\widehat{N}$ almost surely. 
    Mimicking the construction of $Z_t^*$ via the Lévy-Itô-decomposition, we obtain a cadlag process $\widehat{Z}_t^*$ such that
    \begin{align*}
        \mathbb{P}\left( \widehat{Z}^*_t = Z^*_t\; \forall t\in[0,1] \right) = 1.
    \end{align*}
    and $(\widehat{Z}^*_t)_{t\in[0,1]}$ is a measurable function of $(X_t)_{t\in[0,1]}$.
    
    The process $Z_t^*$ is a sum of independent stable Lévy processes.
    Because the $(\beta_j, H_j)$ are pairwise distinct, it follows that the $\beta_j$ are pairwise distinct for all $j\in J^*$. 
    Moreover, since $\beta_j>1$, it trivially holds that $\max_{j\in J^*}\beta_j / \min_{j\in J^*} \beta_j > \frac{1}{2}$, for both parameter sets $\theta_1$ and $\theta_2$.
    The latter inequality is known to ensure (pairwise) identifiability in the following sense: 
    If $\{ (b_j(\theta_1), \beta_j(\theta_1) \}_{j\in J^*(\theta_1)} \neq \{ (b_j(\theta_2), \beta_j(\theta_2) \}_{j\in J^*(\theta_2)}$, then the measures induced by $Z_t^*$ on the Skorokhod space $D[0,1]$ are mutually singular for $\theta_1$ and $\theta_2$, see \cite{ait2012identifying}.
    Since $(\hat{Z}_t)_{t\in[0,1]}$ is a measurable function of $(X_t)_{t\in[0,1]}$, this implies that the measures induced by $X_t$ on $C[0,1]$ are also mutually singular for $\theta_1$ and $\theta_2$.
    Thus, to show mutual singularity in the remaining cases in the sequel, we suppose that $\{ (b_j(\theta_1), \beta_j(\theta_1) \}_{j\in J^*(\theta_1)} = \{ (b_j(\theta_2), \beta_j(\theta_2) \}_{j\in J^*(\theta_2)}$.
    
    Knowing $\beta_j$ and $\alpha^* = H_j - \frac{1}{\beta_j}$, we may thus also identify $H_j$.
    In particular, there exists $(\widehat{b}_j, \widehat{\beta}_j, \widehat{H}_j)$ which are measurable functions of $(X_t)_{t\in[0,1]}$ such that
    \begin{align*}
        \mathbb{P}_{\theta}\left( (\widehat{b}_j, \widehat{\beta}_j, \widehat{H}_j) = (b_j, \beta_j, H_j)\; \forall j\in J^* \right) = 1, \qquad \theta=\theta_1,\theta_2.
    \end{align*}
    
    \underline{Identifying the smoother components:}
    Knowing $\alpha^*$ and $Z_t^*$, we may consider the process
    \begin{align*}
        X^*_t 
        &= \int_0^t (t-s)^{\alpha^*}\, dZ_s^*, \qquad t\in [\tfrac{1}{2}, 1].
    \end{align*}
    To be precise, the integral can be constructed for all $t\in[0,1]\cap \Q$ as the almost sure limit of Riemann sums, and extended to all $t\in[\tfrac{1}{2}, 1]$ as $X^*_t = \lim_{q\downarrow t, q\in\Q} X^*_q$.
    This construction yields a process with continuous paths, because of the following argument: 
    As in \cite[Lemma 4.5]{takashima1989sample}, the process $X_t^\dagger = \int_0^t Z_s \alpha^* (t-s)^{\alpha^*-1}\, ds$ is well-defined and has the same finite-dimensional distribution as $X_t^*$. 
    Since $X_t^\dagger$ has continuous paths, we have with probability $1$ for all $t\in[\tfrac{1}{2}, 1]$ that $\lim_{q\downarrow t, q\in\Q} X_q^\dagger$. 
    As $X_t^\dagger$ is a version of $X_t^*$, the limit also exists for $X_t^*$ and also yields a continuous path.
    
    Using the process $X_t^*$, we obtain
    \begin{align*}
        X_t - X_t^* 
        &= \sum_{j\notin J^*} b_j Y_t^{H_j,\beta_j} + L_t^*, \\
        L_t^*
        &=\int_{-\infty}^0 \left[(t-s)_+^{\alpha^*} - (-s)_+^{\alpha^*}\right]\, dZ^*_s.
    \end{align*}
    Suppose for the moment that $L_t^*$ is almost surely Lipschitz continuous on $[\frac{1}{2}, 1]$. 
    Then we may repeat the steps above for $(X_t-X_t^*)$ on the interval $[\frac{1}{2},1]$ to inductively identify the parameters of the remaining $(q-|J^*|)$ components.
    
    It thus remains to establish Lipschitz continuity of the process $L_t$. 
    Since $X_t$, $X_t^*$, and the $Y_t^{H_j,\beta_j}$ are continuous, the process $L_t^*$ is continuous as well.
    First, note that 
    \begin{align*}
        L_t^* \deq \overline{L}_t = \int_{0}^\infty \left[(s-t)_+^{\alpha^*} -(s)_+^{\alpha^*} \right]\, dZ_s^*,
    \end{align*}
    and we may choose a continuous version of $\overline{L}_t$ as well.
    Moreover, as in \cite[Lemma 4.5]{takashima1989sample}, we find that the process
    \begin{align*}
        \widetilde{L}_t 
        &= \int_0^\infty Z_s^* \frac{d}{ds}\left[(s-t)_+^{\alpha^*} -(s)_+^{\alpha^*} \right]\, ds \\
        &= \int_0^\infty Z_s^* \,\alpha^* \left[(s-t)_+^{\alpha^*-1} -(s)_+^{\alpha^*-1} \right]\, ds
    \end{align*}
    is well-defined, and $(\widetilde{L}_t)_{t\in[\frac{1}{2},1]} \deq (\overline{L}_t)_{t\in[\frac{1}{2},1]}$.
    Now observe that, for $t,v\in[\frac{1}{2},1]$
    \begin{align*}
        |\widetilde{L}_{v}-\widetilde{L}_t| 
        &\leq \alpha^* \int_0^\infty |Z_s^*| \left| (s-v)^{\alpha^*-1} - (s-t)^{\alpha^*-1} \right|\, ds\\
        &\leq |t-v| \alpha^*\int_0^\infty |Z_s^*| \left| (s-\tfrac{1}{2})^{\alpha^*-2}\right|\, ds.
    \end{align*}
    The latter integral is finite because $\limsup_{s\to\infty} |Z_s^{\beta_j}| s^{-\frac{1}{\beta_j}-\delta} = 0$ almost surely, for any $\delta>0$ \cite[47.13]{sato1999levy}. 
    Since $\beta_j>1$ for all $j=1,\ldots, q$, by assumption, this also implies that $\limsup_{s\to\infty} |Z_s^* | |s|^{-1} = 0$ almost surely.
    Noting that $\alpha^*<1$, the latter integral is almost surely finite.
    This yields almost sure Lipschitz continuity of $\widetilde{L}$, and thus the same holds for $L_t^*$ because $(\widetilde{L}_t)_{t\in[\frac{1}{2},1]} \deq (L_t^*)_{t\in[\frac{1}{2},1]}$.
\end{proof}

\noindent
\textbf{Acknowledgement} Mark Podolskij gratefully acknowledges financial support of ERC Consolidator Grant 815703 “STAMFORD: Statistical Methods for High Dimensional Diffusions”.

\bibliography{fsm-aos}
\bibliographystyle{apalike}

\end{document}